%% file: main.tex
\documentclass[11pt]{article}
\usepackage[utf8]{inputenc}
\usepackage[a4paper,left = 1.1in,right = 1.1in, top = 1.15in, bottom = 1in]{geometry}
\usepackage{amsmath}
\usepackage{amsthm}
\usepackage{bm}
\usepackage{cancel}
\usepackage{amssymb}
\usepackage{dsfont}
\usepackage{commath}
\usepackage{mathtools}
\usepackage{tikz,tikz-cd}
\usepackage{float}
\usepackage{cases}
\usepackage{url}
\newcommand{\vertiii}[1]{{| \! | \! | #1 | \!
    | \! |}}
\usepackage{enumitem,linegoal}

\usepackage[numbers]{natbib}

\input{Math_Operators}

\usepackage{subcaption}
\usepackage{graphicx}
\usepackage{color}
\usepackage{algorithm}
\newcommand*{\avint}{\mathop{\ooalign{$\int$\cr$-$}}}

\usepackage{xcolor}
\usepackage{hyperref}
\hypersetup{colorlinks,linkcolor=blue,citecolor=blue,urlcolor = magenta}
\usepackage[capitalise]{cleveref}

\tikzset{
    myarr/.style={-stealth, thick},
    mylab/.style={font=\small},
}

\def\avint{\mathop{\,\, \rlap{--}\hspace{-1.7mm} \int \hspace{-1.6mm}}\nolimits} 

\def\avintin{\mathop{\,\, \rlap{--}\hspace{-1.05mm} \int \hspace{-1.2mm}}\nolimits} 

\usepackage{fancyhdr}
\usepackage{authblk}
\title{\vspace{-10mm} 
Occupied Processes: Going with the Flow
\vspace{-3mm}}

\date{\today}
\author{Valentin Tissot-Daguette\footnote{Email: {\tt vtissotdague@bloomberg.net}. Research supported by the 2022-2023 and 2023-2024 Bloomberg Quantitative Finance Ph.D. Fellowship.  }}%

\affil
{\footnotesize Quantitative Research, Office of the  CTO, Bloomberg \vspace{-2mm} }

\begin{document}
\maketitle
\vspace{-1cm}

\input{Abstract}
\newpage
\setcounter{tocdepth}{2}
\tableofcontents

\newpage 

\input{Intro}

\input{ITO}

\input{SpotLT}

\input{Financial_Examples/main}
\appendix
\input{Appendix/Invariance}
\input{Appendix/Proofs}
\input{Appendix/Algorithm}

\bibliographystyle{abbrvnat}
\bibliography{main.bib}
\end{document}

%% file: Math_Operators.tex
\def\be{\begin{equation}}
\def\ee{\end{equation}}

\def\bea{\begin{align}}
\def\eea{\end{align}}

\def\bea*{\begin{align*}}
\def\eea*{\end{align*}}
\def\mo{\mathfrak{o}}
\theoremstyle{plain}
\newtheorem{theorem}{Theorem}[section]

\newtheorem{corollary}{Corollary}
\newtheorem{proposition}{Proposition}

\theoremstyle{definition}
\newtheorem{definition}{Definition}
\newtheorem{example}{Example}
\newtheorem{remark}{Remark}

\DeclareMathOperator{\E}{\mathbb{E}}
\DeclareMathOperator{\F}{\mathbb{F}}

\DeclareMathOperator{\N}{\mathbb{N}}

\DeclareMathOperator{\Q}{\mathbb{Q}}
\DeclareMathOperator{\R}{\mathbb{R}}

\DeclareMathOperator{\calB}{\mathcal{B}}
\DeclareMathOperator{\calC}{\mathcal{C}}
\DeclareMathOperator{\calD}{\mathcal{D}}

\DeclareMathOperator{\calF}{\mathcal{F}}

\DeclareMathOperator{\calJ}{\mathcal{J}}

\DeclareMathOperator{\calL}{\mathcal{L}}
\DeclareMathOperator{\calM}{\mathcal{M}}
\DeclareMathOperator{\calN}{\mathcal{N}}
\DeclareMathOperator{\calO}{\mathcal{O}}

\DeclareMathOperator{\calR}{\mathcal{R}}
\DeclareMathOperator{\calS}{\mathcal{S}}
\DeclareMathOperator{\calT}{\mathcal{T}}

\DeclareMathOperator{\frakD}{\mathfrak{D}}

\DeclareMathOperator{\frakF}{\mathfrak{F}}

\DeclareMathOperator{\frakS}{\mathfrak{S}}

\DeclareMathOperator{\frakm}{\mathfrak{m}}

%% file: Abstract.tex
\vspace{0mm}
\begin{abstract}
 
A stochastic process $X$ becomes \textit{occupied} when it is enlarged with its occupation flow $\calO$ that tracks the time spent by the path at each level. When $X$  is Markov, the occupied process  $(\calO,X)$ enjoys a Markov structure as well. We develop an Itô calculus for occupied processes that lies midway between Dupire's functional Itô calculus and the classical version. We derive  Itô formulae and, through Feynman-Kac, unveil a broad class of path-dependent PDEs where $\calO$ plays the role of time.  The space variable, given by the current value of $X$,  remains finite-dimensional, thereby paving the way for standard elliptic PDE techniques and numerical methods. 

The framework's benefits are   illustrated via an  optimal stopping problem involving local times, followed by financial  applications. For the latter, we  show how occupation flows provide  unified Markovian lifts for exotic options and variance instruments, allowing financial institutions to  price derivatives books with a single numerical solver. We finally explore an extension of forward variance models so as to leverage the entire forward occupation surface.
\end{abstract}
\vspace{2mm}

\\ \vspace{-2mm}
	



\vspace{13mm}

\textbf{Acknowledgments.} 
I would like to thank Bruno Dupire, Ibrahim Ekren, Julien Guyon, Yuxing Huang, Thibault Jeannin,  Bryan Liang, Jin Ma, Alexandre Pannier, Mete Soner, Nizar Touzi,  Jianfeng Zhang, and  Xin Zhang  for stimulating discussions. My gratitude goes also to two anonymous referees for their feedback and interesting ideas. 
 

%% file: Intro.tex
\section{Introduction}
This work  originates from an   identity  of \citet[Chapter VI]{RevuzYor}, 
\begin{equation}\label{eq:revuzYor}
    \sup_x L^x_T = \sup_{t\le T} L_t^{X_t}, 
\end{equation}
where $X$ is  Brownian motion and $L^x$ its local time process at $x$.   
We refer to $L_t^{X_t}$  as the \textit{spot local time}. 
To prove \eqref{eq:revuzYor}, let $x^*$ be a maximizer\footnote{The maximum is achieved since $x \mapsto  L_T^x$ has compact support and  admits a continuous version  \cite[Chapter VI]{RevuzYor}.} of  $x\mapsto L^x_T 
$   and define the random time $\tau^* = \sup\{t\le T \ | \ X_t = x^*\}$. Then $\sup_x L^x_T =  L_{\tau^*}^{X_{\tau^*}}$, which shows that $\sup_x L^x_T \le \sup_{t\le T} L_t^{X_t}$. The reversed inequality is immediate.  
One question arises: 
what if 
the anticipative strategy $\tau^*$ is replaced by stopping times?
I.e., what is the  optimal rule and value of the optimal stopping (OS)  problem, 
\begin{equation}\label{eq:OSLTIntro}
    \sup\{ \E[L_\tau^{X_\tau}] \ : \ \tau = \text{stopping time}\}?
\end{equation}
To our best knowledge, there is no  definitive approach in the literature to tackle stochastic control  or optimal stopping  problems involving local times as in \eqref{eq:OSLTIntro}. 
Our work addresses this by recovering a Markovian framework in these situations, classically leading  to considerable simplifications \cite{FlemingSoner,PeskirShiryaev}.  
This is achieved by exploiting the  fact that for continuous semimartingales, the local time field $(L_t^x)_{x\in \R}$ 
is 
the Radon-Nikodym derivative of the occupation measure of $X$,
\begin{equation}\label{eq:occMeasureIntro}
    \calO_t(A) = \int_0^t \mathds{1}_{A}(X_s)d\langle X \rangle_s, \quad A= \text{Borel set of } \R.  
\end{equation}
See \cite{GemanHorowitz}, \cite[Chapter VI]{RevuzYor}, and \cref{sec:occProcess}.  
The \textit{occupation flow} $\calO = (\calO_t)_{t\ge 0}$ is a measure-valued process adapted to the filtration of $X$. 
Moreover,  the \textit{occupied process} $(\calO,X)$ is Markov provided that $X$  is  Markov. 

The goal of this work is to develop a  differential calculus for occupied processes and demonstrate its relevance and practicality. 
The study of occupied processes strikes a middle ground  between the classical Itô calculus and the fully path-dependent setting introduced by  \citet{DupireFITO} (see also \cite{ContBally, ContFournie}).  
We shall see throughout that numerous 
examples in mathematical finance and  stochastic analysis   fit 
into our framework. 
Moreover, the occupation measure gives rise to  nice structural and regularity properties otherwise  unavailable in  pathwise calculi.

Introduced by \citet{Dawson75,Dawson93} and  \citet{FlemingViot}, measure-valued processes  have gained increasing attention  recently, e.g.,  in mean field games and  McKean-Vlasov controls \cite{CardDelLasLio,CarmonaDelarue,LasryLions,TTZ,TTZ2}, robust pricing and hedging   \cite{CoxKallblad,Larsson}, and infinite-dimensional polynomial diffusions \cite{CuchieroLarsson,Cuchiero2021}. The measure may represent the marginal law of a process 
or  
a conditional distribution 
which concentrates as the filtration unfolds \cite{CoxKallblad,Larsson}. 
Either way, the process takes values 
in a   suitable subspace  of  probability measures. 
In contrast, $\calO$ 
evolves  in the convex cone $\calM$ of finite positive  measures. The  meaning  behind  the occupation measure also greatly differs from flows of distributions: it is a pathwise object whose total mass  increases 
with the passage of time. Besides, occupation measures can be  added to one another. To study occupation functionals, 
this  presents   a clear advantage over  working in
the path space,  deprived of vector space structure.  

The study of functionals  of occupied processes and their dynamics appears to be new. 
In fact, the  occupation flow  is rarely the main object of interest and 
comes  rather
as a  
tool. Let us 
mention  the use of the expected occupation measures 
to formulate relaxed optimal stopping or stochastic control problems \cite{Bhatt,Tankov, FlemingVermes}. In contrast, the occupation flow is herein  given at  the outset, and may directly appear in the objective function of control problems as 
 exemplified by the optimal stopping problem \eqref{eq:OSLTIntro}. 
 Moreover, local times may control the dynamics of the system, as shown in the recent work by \citet{Bethencourt} which  adopted the framework developed here. The control problem considered therein relates to the evolution of filamentous fungi  \cite{tomaševićFungus,BoswellDavidson} offering another promising line of application for occupied processes.
  For the theoretical foundation of controlled occupied processes, we refer the interested reader to our companion  paper \cite{SonerTissotZhang}
 
\subsection{Contributions}
In \cref{thm:ItoOX},  we generalize Itô's formula to \textit{occupation functionals} $ f(\calO_t,X_t)$  and discover, to our fortunate  
surprise, that the expression is almost identical to  the classical case. The only difference is that the time derivative  in Itô's formula for functions $ f(t,X_t)$ 
is replaced by the \textit{occupation derivative} $\partial_{\mo}f(\calO_t,X_t) = \delta_{\mo}f(\calO_t,X_t)(X_t)$ where $\delta_{\mo}$ 
is the linear derivative \cite[Chapter 2]{CardDelLasLio} and  $\mo$ is the occupation flow variable which takes values in the space $\calM$ of finite measures. The projection of  $\delta_{\mo}f(\calO_t,X_t) \in \calC(\R)$ onto the spot  $X_t$ comes from  the singular dynamics of  the occupation measure   
$d\calO_t = \delta_{X_t} d\langle X \rangle_t$, where $\delta_x$ denotes the Dirac mass at $x$.  
 We are  then able to extend Itô's formula in \cref{thm:ItoOX} to occupied continuous semimartingales, namely
\begin{equation}\label{eq:Itointro}
    df(\calO_t,X_t) = \Big(\partial_{\mo} + \frac{1}{2} \partial_{xx}\Big)f(\calO_t,X_t) d\langle X \rangle_t + \partial_x f(\calO_t,X_t)dX_t.
\end{equation}
 To compare the above formula with other It\^o calculi, we shall introduce the \textit{calendar time occupation flow} 
$\tilde{\calO}(A) = \int_0^{\cdot} \mathds{1}_{A}(X_s)ds,$ leading to the more familiar expression
\begin{equation}\label{eq:ItointroStd}
    df(\tilde{\calO}_t,X_t) = \partial_{\mo}f(\tilde{\calO}_t,X_t) dt +  \frac{1}{2} \partial_{xx}f(\tilde{\calO}_t,X_t) d\langle X \rangle_t + \partial_x f(\tilde{\calO}_t,X_t)dX_t.
\end{equation} 
We can then show that  \eqref{eq:ItointroStd} is consistent with  Dupire's \textit{functional Itô formula} \cite{DupireFITO} and the classical one. Precisely, the occupation derivative $\partial_{\mo}$ replaces   the (functional) time derivative. We then apply It\^o's formula in the case where $(\calO,X)$, $(\tilde{\calO},X)$ is solution of  \textit{occupied} stochastic differential equations (OSDEs),
\begin{align}
     dX_t  &=   b(\calO_t,X_t)dt + \sigma(\calO_t,X_t)dW_t, \qquad d\calO_t =  \delta_{X_t} \sigma(\calO_t,X_t)^2dt\label{eq:OSDEIntro} \\[0.5em] 
  dX_t  &=   b(\tilde{\calO}_t,X_t)dt + \sigma(\tilde{\calO}_t,X_t)dW_t, \qquad \ d\tilde{\calO}_t =  \delta_{X_t} dt\label{eq:OSDEIntroStd} 
\end{align}
where their wellposedness is proved in \cref{sec:OSDE}. OSDE \eqref{eq:OSDEIntroStd} shares similarities with \textit{self-interacting diffusions} \cite{DurrettRogers, CranstonLeJan, Raimond, Benaim1,Benaim2,Benaim3,Benaim4}  where the drift uses the  calendar time occupation flow to create mean reversion or repulsion. 

Applying \eqref{eq:ItointroStd} to $f(\tilde{\calO}_t,X_t) = \E^{\Q}[\varphi(\tilde{\calO}_T,X_T) | \calF_t]$ where 
$(\tilde{\calO},X)$ evolves according to \eqref{eq:OSDEIntroStd},  
we then obtain by Feynman-Kac's formula (\cref{thm:FKDirichletStd}) that $f$ solves  the  backward  PDE (possibly in a weak sense)  
\begin{equation}\label{eq:PDEINTRO}
    \Big(\partial_{\mo} + b \partial_x + \frac{\sigma^2}{2}\partial_{xx}\Big) u(\mo,x) = 0, \;  \mo(\R) < T, \quad  u(\mo,x) = \varphi(\mo,x), \;   \mo(\R) = T.  
\end{equation}

Under the stochastic  clock $\calO_t(\R) = \langle X \rangle_t$, one obtains a Dirichlet problem instead; see  \cref{thm:FKDirichlet}. 
We can therefore recast a large class of path-dependent PDEs  involving the occupation derivative $\partial_{\mo}$ in lieu of the usual time derivative $\partial_t$.  
As the space derivatives in \eqref{eq:PDEINTRO} are the classical ones, classical results in viscosity theory can be leveraged to  study  a large class of (possibly nonlinear)  path-dependent PDEs; see the accompanying paper \cite{SonerTissotZhang}. 
 
Proving the   regularity of solutions of path-dependent PDEs is notoriously hard  \cite{ZhangEkren2014,PengWang}  (see also \cite[Chapter 11]{ZhangBook}) and still a subject of active research. 
 In a recent work, \citet{BouchardTan} study  linear parabolic PPDEs for functionals depending on  $X_t$ and  $I_t = \int_0^t X_s dA_s$ for some continuous function $A$   of finite variation. Note that their setting  relates to ours as the integral process $I$ can be retrieved from the occupation measure in some situations. Indeed, if    $A_t := \langle X \rangle_t$  is deterministic (e.g., when $X$ is a Gaussian martingale) and  fulfills the conditions in \cite{BouchardTan}, then  $I_t  = \int_{\R} x \calO_t(dx)$. 

Effective state enlargement ensures that Markovian representations of conditional expectations are both regular and numerically tractable; see 
\cite{ViensZhang} for Gaussian Volterra processes or  the \textit{Better Asian PDE} example  in \cite{DupireFITO}.  
While  the  path functional $$\textnormal{v}(t,\omega) = \E^{\Q}[\varphi(\calO_T,X_T) | \calF_t](\omega), \quad \omega \in \calD([0,T], \R), $$ 
  trivially induces a Markovian framework,  this representation is often impractical. Moreover,  the continuity in $\omega$ of \textnormal{v} is typically not guaranteed: we shall see in \cref{ex:heatspotLT}, related to spot local times,  that   \textnormal{v}   is  discontinuous in the input path while  $v$ is smooth. This further supports choosing a more adequate Markovian lift, like  $(\calO,X)$.

\subsection{Applications}

The  OS problem \eqref{eq:OSLTIntro} 
is  analyzed in \cref{sec:spotLT}. After extending standard tools  to our context, we shed light in  on  the numerical resolution of \eqref{eq:OSLTIntro} using 
a least square Monte Carlo approach  \cite{LS}.  
The stopping of spot local time  reveals great challenges  and comes as a novel  benchmark for numerical methods solving  path-dependent OS  problems such as \cite{Bayer,bayraktar,Becker}.

Section \ref{sec:finApp} is dedicated to  financial applications. We shall see that the occupation measure is ever-present, particularly  in the eponymous occupation time derivatives \cite{Detemple,Hugonnier,ChesneyJeanblancYor}, timer options, and corridor variance swaps.  
It is a challenge for financial institutions to consistently price and manage complex derivatives books, and  we show in  \ref{sec:finApp} that occupation flows offer a unified framework for this task.  We finally propose a generalization of forward variance models in \cref{sec:FwdOccupation},  opening the doors to \textit{nonlinear} corridor variance instruments.

The calendar time occupied SDE \eqref{eq:OSDEIntroStd} can also be applied to model stock price dynamics directly. This is the focus of our companion paper \cite{TissotLOV} that introduces  \textit{local occupied volatility} (LOV) models, at the intersection of  local volatility \cite{DupireLV} and  path-dependent volatility  \cite{Guyon2014,GuyonSlides,GuyonLekeufack,HobsonRogers}.

\section{Occupied Processes}\label{sec:occProcess}  

We restrict ourselves to one space dimension to simplify the exposition. 
  Let $\lambda$ be the Lebesgue measure on  $\R$. We may write $\int f d\lambda$ or $\int f(x)  dx$, whichever is more convenient. We also write $\calB(\R)$ for the Borel $\sigma-$algebra on $\R$  and $\calM$ for the set of finite, positive Borel measures on $\R$. 
Given $\mu \in \calM$,  $\avintin_A f d\mu = \frac{1}{\mu(A)}\int_A f d\mu$ denotes the average of $f$ with respect to $\mu$  inside $A\in \calB(\R)$.  We   write $B_{\varepsilon} = (-\varepsilon,\varepsilon)$ for the open ball of radius $\varepsilon$ centered at the origin and $B_{\varepsilon}(x) = x + B_{\varepsilon}$. Finally,  int$A$, $\partial A$, $\bar{A}$ denote the interior, boundary, and closure of a set $A$, respectively.

Let $X$ be a real-valued continuous semimartingale 
on a probability space $(\Omega,\calF,\F,\Q)$ where $\F = (\calF_t)_{t\ge 0}$ is the natural filtration of $X$. The  \textit{occupation measure  of $X$} at $t\ge 0$ is defined as  
\begin{equation}\label{eq:occMeasure}
    \calO_t(\omega): \calB(\R) \to \R, \quad  \calO_t(\omega)(A) =\int_0^t\mathds{1}_A(X_s
)d\langle X \rangle_s(\omega), \quad A \in \calB(\R), \quad \omega \in \Omega. 
 \end{equation}
Unless stated otherwise,  we  omit the dependence on $\omega$ throughout and write $\calO_t(A)$ for the random variable $\omega \mapsto \calO_t(\omega)(A)$. 
\begin{definition}
The collection of  occupation measures $\calO = (\calO_t)_{t\ge 0}$  is called the \textit{occupation flow of $X$}. Moreover, the pair $(\calO,X)$ is  an \textit{occupied process}, and takes values in the product space $\calD  : =  \calM \times \R$. 
\end{definition}
The  occupation flow can also be regarded as a measure-valued process; see  \cite{Dawson75,Dawson93,FlemingViot}. 
\begin{proposition}\label{prop:OccProperties} 
The occupation flow $\calO$ satisfies the following properties. 
    \begin{itemize}
        \item[\textnormal{I.}] \textnormal{(Strong Time Additivity)} 
        For all stopping times $\varsigma \le \tau$ and $A \in \calB(\R)$,  
\begin{equation}\label{eq:additive}
    \calO_{\tau}(A) = \calO_{\varsigma}(A) + \calO_{\varsigma,\tau}(A), \quad \Q-a.s., 
\end{equation}
with the incremental flow
$\calO_{\varsigma,\tau}(A) := \int_\varsigma^\tau \mathds{1}_{A}(X_s)d\langle X \rangle_s$. 

        \item[\textnormal{II.}] \textnormal{(Absolute Continuity)} $\calO_t \ll \lambda$ $ \ \Q-$a.s., where $\lambda$ is the Lebesgue measure on $\R$. 
             \item[\textnormal{III.}] \textnormal{(Occupation Time Formula)} For every Borel function $\phi:\R\to \R_{+}$, 
\begin{equation}\label{eq:OTF}
    \int_0^t \phi(X_s)d\langle X \rangle_s = \int_{\R} \phi(x)\calO_t(dx).  
\end{equation}
             
    \end{itemize}
\end{proposition}
\begin{proof}
See I. \cite[X]{RevuzYor}; II.  \cite[VI]{RevuzYor} or \cite{GemanHorowitz}; III. \cite[VI]{RevuzYor}.
\end{proof}
The Radon-Nikodym derivative of $\calO_t$ is the  \textit{local time field}  $L_t = (L_t^x)_{x\in \R} $, 
\begin{equation}\label{eq:LTdef}
     L_t^x = \frac{d\calO_t}{d\lambda}(x) =   \lim_{\varepsilon\to 0} \frac{1}{2\varepsilon}\int_0^t\mathds{1}_{B_{\varepsilon}(x)}(X_s) d\langle X \rangle_s. 
\end{equation} 
Evidently, the (strong) time additivity of $\calO$  carries over to $L$, that is $
    L_\tau = L_{\varsigma} + L_{\varsigma,\tau}$ with $L_{\varsigma,\tau} := \frac{d\calO_{\varsigma,\tau}}{d\lambda}$ for all stopping times $\varsigma \le \tau$.
 The local time 
 is in fact central 
in the theory of \textit{additive functionals}   \cite{BlumenthalGetoor}, \cite[Chapter X]{RevuzYor}. 
For local martingales,  it is well-known that the process $L_t= (L_t^x)_{x\in \R}$  
(indexed by $x$) 
admits a version that is  $\alpha-$H\"older continuous for all $\alpha <1/2$. For general semimartingales,  $x\mapsto L_t^x$ may only shown to be càdlàg \cite[Chapter VI]{RevuzYor}.  
Finally, note that  $\calO_t(\R) = \langle X \rangle_t$, i.e.,  the total mass of  $\calO$ coincides with the stochastic clock of $X$. 
 
Introduce also the space of compactly supported,   absolutely continuous Borel measures, 
\begin{equation}\label{eq:M}
     \calM^{\lambda} = \big\{\mo \in \calM \ : \ \mo \ll \lambda \text{ and }  \frac{d\mo}{d\lambda}:\R\to \R \text{ has compact support} \big\}.
\end{equation}
 Clearly,  $\calO_t \in \calM^{\lambda} \ \Q-$a.s. from \cref{prop:OccProperties} II. and  the compactness of $\{X_s :s\le t\}=\text{supp}(\calO_t)$  
 as $X$ is continuous.  
Note that $\calM^{\lambda}$ and  $\calM$ are convex cones. 
In particular, occupation measures can be  added to one another. 
This vector space structure  is 
 advantageous over  
a pathwise setting 
where the addition of  paths on different horizons cannot be reasonably defined. 

\subsection{Calendar Time Occupation Flow}\label{sec:calendarTimeOcc}
Let us  also introduce the \textit{calendar time occupation flow}, 
\begin{equation}\label{eq:stdTimeoccupation}
    \tilde{\calO}_t(A) :=\int_0^t\mathds{1}_A(X_s
)ds, \quad A \in \calB(\R), \quad t\in [0,T], 
     \end{equation}  
     and \textit{calendar time occupied process} $(\tilde{\calO},X)$. 
     Note that $\tilde{\calO}$ also satisfies properties I., III. in \cref{prop:OccProperties}, but not II (absolute continuity). Indeed, $\tilde{\calO}$ may generate atoms if $X$ is constant on a time interval of positive length. 
     Hence $\tilde{\calO}_t \in \calM \setminus \calM^{\lambda}$ in general, so there may not exist an associated local time process.  
     On the other hand, 
     the calendar time occupation measure is well-defined  for paths of any roughness 
     and nontrivial for 
     smooth trajectories. 
Therefore, $\tilde{\calO}$ is more universal than $\calO$ and may be  preferred in some applications;  see Section  \ref{sec:occPayoffs} and \cite{TissotLOV}.  

\subsection{Occupation Flow Dynamics}\label{sec:occFlowDynamics}

By definition, occupation flows accumulate mass at the levels visited by  $X$. As the latter is  continuous, the flow increases only in a neighborhood of the current value  within a short  time window. 
Indeed, as  $d\calO_t(A) = \mathds{1}_{A}(X_t)d\langle X \rangle_t$,  the occupation time $\calO_t(A)$ increases at $t$ if and only if $X_t\in A$. Noting also that $\mathds{1}_{A}(X_t) = \delta_{X_t}(A)$ leads to  the following dynamics 
\begin{align}
   d\calO_t = \delta_{X_t}d\langle X \rangle_t. \label{eq:flowDynamics} 
\end{align} 
Similarly,  the calendar time occupation flow  evolves according to   
\begin{align}
  d\tilde{\calO}_t = \delta_{X_t}dt, \label{eq:flowDynamicsStd} 
\end{align}
Occupation flows thus admit singular dynamics  that involves a Dirac impulse at the current level. 
See \cref{fig:dynamicO}.

 \begin{figure}[H]
\centering
\caption{Occupation flow  $\calO$ and its dynamics $d\calO_t = \delta_{X_t}d\langle X\rangle_t$. }
  \includegraphics[height=2.3in,width=3.2in]{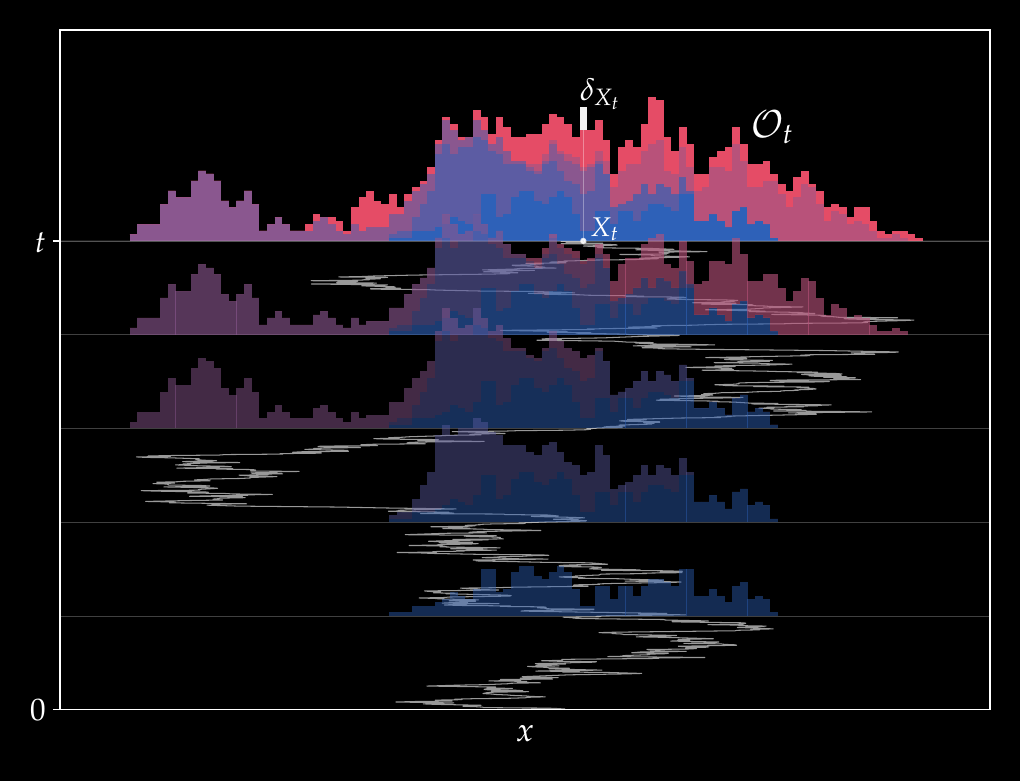}
 \label{fig:dynamicO}
 \end{figure}

 \subsection{Occupation Functional and Markov Property}

An \textit{occupation functional} is a map $f: \calD \to \R$, $(\mo,x) \mapsto f(\mo,x)$. 
Naturally, we are  interested in plugging in  occupied processes   and studying properties of  $f(\calO_t,X_t)$ or $f(\tilde{\calO}_t,X_t)$.  
We stress that occupation functionals do not explicitly depend on the time variable because  occupation flows already play  that role. 
This simplification   will  reveal its  grace in \cref{sec:ITOMAIN}.

We now establish the (strong) Markov property of occupied processes. 

\begin{proposition}\label{prop:Markov}
Let $X$ be a strong Markov process. Then the occupied processes $(\calO,X)$,  $(\tilde{\calO},X)$ are  strong Markov as well. 
\end{proposition}
\begin{proof}
This follows from 
the strong time additivity  of the occupation measure (\cref{prop:OccProperties} I): for  
all stopping times $0\le \tau\le \varsigma \le T$ and occupation functional   $f:\calD \to \R$, \begin{align*}
    \E^{\Q}[f(\calO_{\varsigma},X_{\varsigma}) \ |  \calF_{\tau}]
    =  
    \E^{\Q}[f(\calO_{\tau}+\calO_{\tau,\varsigma},X_{\varsigma}) \ |   \calF_{\tau}]
     = \E^{\Q}[f(\calO_{\varsigma},X_{\varsigma}) \ |  (\calO_\tau,X_\tau)], 
\end{align*} 
noting that  $\calO_{\tau,\varsigma}$ (and $X_{\varsigma}$) is conditionally independent of $\calF_{\tau}$ given $X_{\tau}$. The same arguments apply to the calendar time occupied process. 
\end{proof}

\begin{remark}\label{rem:DDS}
 Suppose that  $X$ is a continuous  local martingale with Dambins-Dubins-Schwartz (DDS) Brownian motion
 $$W_\theta := X_{\tau_\theta}, \qquad 
\tau_\theta = \inf\{t \in [0,T] \ : \ \langle X \rangle_t > \theta\},\qquad \theta \le \Theta :=  \langle X\rangle_T , $$ or conversely $X_t = W_{\theta_t}$, $\theta_t := \langle X\rangle_t$. If $\calO^X$ (respectively $\calO^W$) denotes the occupation flow of $X$ (resp. $W$), it is easily seen that $\calO_t^X(\omega)=\calO_{\theta_t}^W(\omega)$ for all $\omega \in \Omega$.   
 In particular, for every  $\F-$stopping time $\tau$, and functional $f:\calD\to \R$,
 \begin{equation}
  f(\calO_\tau^{X},X_\tau) = f(\calO_{\theta_{\tau}}^W,W_{\theta_{\tau}}). 
 \end{equation}
Note that $\theta_{\tau}$ is a stopping time under the Brownian filtration $\F^W = (\calF_{\tau_\theta})_{\theta\le \Theta}$. This observation will be particularly useful  in \cref{sec:spotLT} to derive simple 
formulation of 
optimal stopping problems.   
See also \cite[Chapter IV]{PeskirShiryaev}.      
\end{remark}

\subsection{Metrics on Space of Measures}\label{sec:metrics}
In connection with local time, we introduce  
 \begin{equation}\label{eq:metrics}
    \frakm_p(\mo,\mo') :=  \lVert \mo-\mo' \rVert_{p}, \quad \lVert \mo \rVert_{p} := \left\lVert \frac{d\mo}{d\lambda} \right\rVert_{L^{p}(\R)}, \quad \mo,\mo' \in \calM^{\lambda}, \quad p\in [1,\infty].  
 \end{equation}
 Since $\frac{d\mo}{d\lambda}$ is bounded and compactly supported for each $\mo\in \calM^{\lambda}$, then $\lVert \mo\rVert_p <\infty$  $\forall \ p\in [1,\infty]$. 
 Note that the distance between any two measures $ \calM^{\lambda}$ makes sense irrespective of their respective total mass. 
This is particularly convenient when comparing the occupation measure 
at different times, to wit,  
$$  \frakm_p(\calO_s,\calO_t) = \lVert\calO_{s,t}\rVert_p  = \lVert L_{s,t}\rVert_{L^{p}(\R)}, \quad  s<t. $$
Important values for $p$ are 
$p\in \{1,\infty\}$. First,  $\frakm_1$ 
coincides with the total variation distance of finite measures, so it can  be extended to $\calM$ and used for the calendar time occupation flow as well. Moreover, for all $s\le t,$  $\frakm_1(\calO_s,\calO_t) = \lVert \calO_{t-s} \rVert_1$ is conveniently the total mass of $ \calO_{t-s}$.  
For these reasons, we shall endow $\calM^{\lambda}$ (and $\calM$)  with $\frakm_1$   when establishing Itô's formula in \cref{sec:ITOMAIN}. 
Second,  the supremum distance $\frakm_{\infty}$ is pivotal to  
the spot local time in \cref{sec:spotLT}. 
Indeed, a strong metric is required to guarantee the continuity of $\mo \mapsto  \varphi(\mo,x) = \frac{d\mo}{d\lambda}(x)$ (hence $  \varphi(\calO_t,X_t)  = L_t^{X_t}$). 
Note that 
$\lVert\calO_{t}\rVert_{\infty} = \sup_{x\in \R} L^x_t$, which admits well-known bounds in $L^p(\Q)$; 
see \cite[XI]{RevuzYor}, \cite{BarlowYor}, and references therein.  

Besides, we shall use in \cref{sec:OSDE} the  \textit{bounded Lipschitz}  \textit{distance} \cite{Hanin} to prove  
   the existence and uniqueness of solutions of \textit{occupied} stochastic differential equations SDEs.  This metric proves    
   more suitable than 
   $\frakm_p$ due to its dual representation.

%% file: ITO.tex
\section{It\^o Calculus}\label{sec:ITOMAIN}


\subsection{Occupation Derivative}\label{sec:occDerivative}
We seek a suitable 
differential operator, say $\delta_\mo :\calC^1(\calM) \to \calC(\R)$, with the class  $\calC^1(\calM)$ still to be defined,  
such that $f(\calO_t)$ satisfies the chain rule
\begin{equation}\label{eq:chainRule}
    df(\calO_{t})  = \delta_\mo f(\calO_t) \cdot  d\calO_t, \quad f\in \calC^{1}(\calM),
\end{equation}
with the dual pairing 
\begin{equation}\label{eq:scalar}
    \phi \cdot \mo := \int_{\R} \phi(x)\mo(dx), \quad \phi:\R\to \R, \quad \mo \in \calM. 
\end{equation}
Plugging  the occupation flow dynamics $d\calO_t = \delta_{X_t}d\langle X \rangle_t$ into \eqref{eq:chainRule}, 
\begin{equation}\label{eq:dynamicsf2}
    df(\calO_t) = \delta_{\mo}f(\calO_t) \cdot d\calO_t = \delta_\mo f(\calO_t) \cdot \delta_{X_t}d\langle X \rangle_t = \delta_\mo f(\calO_t)(X_t)d\langle X \rangle_t. 
\end{equation}
It is natural to impose that \eqref{eq:dynamicsf2}   holds  for all linear maps $f(\mo) = \phi\cdot \mo$, $\phi \in \calC(\R)$. From the occupation time formula \eqref{eq:OTF}, observe that  $f(\calO_t) = \int_0^t \phi(X_s) d\langle X \rangle_s$. Together with  \eqref{eq:dynamicsf2},  this implies that  
\begin{equation}\label{eq:linearFctalCond}
     \delta_{\mo} (\mo \mapsto \phi \cdot \mo) \equiv \phi, 
\quad \forall \phi\in \calC(\R).
\end{equation}
Hence $\delta_{\mo}$ must be  the  \textit{linear derivative} 
\cite{CardDelLasLio,CarmonaDelarue}. 
Although typically used for probability measures, we  effortlessly extend its definition  to arbitrary (finite) measures. 
In what follows, we   endow  $\calM$ 
with the total variation metric $\frakm_1$  in \cref{sec:occProcess}. 
We also extend the dual pairing in \eqref{eq:scalar} to signed measures $\mo = \mo^+ - \mo^-$, $\mo^+, \mo^- \in \calM$, by setting   $\phi \cdot \mo = \phi \cdot \mo^+ - \phi \cdot \mo^-$. 
 
\begin{definition}\label{def:linearderivative}
   A functional   $f:\calM \to \R$  is  \textit{continuously linearly differentiable} 
 ($f\in \calC^1(\calM)$)  if there exists  $\delta_\mo  f: \calM \to \calC(\R)$,
   the \textit{linear derivative of $f$},  such that

   \begin{itemize}
       \item[I.]   $\calM\times \R \ni   (\mo,x) \mapsto \delta_\mo  f(\mo)(x)$ is continuous in the product topology, 

        \item[II.]  $|\delta_\mo  f(\mo)(x)| \le C(1+|x|^2)$ for some  $C\in (0,\infty)$, uniformly in $\mo$,
       
       \item[III.]   For all $\mo,\mo'\in \calM$ and $\Delta = \mo' - \mo$, then 
    \begin{equation}\label{eq:linDer}
        f(\mo') - f(\mo) = \int_0^1  \delta_\mo f(\mo + \eta \Delta) \cdot  \Delta  \ d\eta. 
    \end{equation}
   \end{itemize}

\end{definition}
 In light of \eqref{eq:linDer}, we can regard  $\delta_{\mo}f(\mo)(x)$ as the Gateaux derivative of $f$ at $\mo$ in the direction  $\delta_x$, namely
   $ \frac{d}{dh} f(\mo+h\delta_x)|_{h=0}.$ 
In the mean field game literature \cite{CarmonaDelarue,CardDelLasLio,TTZ},
the linear derivative is defined up to an additive constant as it is restricted to  probability measures.   Herein, 
the  total mass of the test measures $\mo,\mo'$ may differ, which pins down  $\delta_\mo  f$  uniquely. 

Next, we introduce  a notion of differential tailored to occupation measures. 

\begin{definition}\label{def:occDerivative}
    The \textit{occupation derivative} $\partial_{\mo} f: \calM \times  \R \to \R$, is   defined as 
    \begin{equation}\label{eq:defOccDerivative}
    \partial_{\mo}f(\mo,x) = \begin{cases}
        \delta_\mo  f(\mo)(x), \quad & f\in \calC^{1}(\calM), \\[0.5em]
        \delta_\mo  f(\mo,x)(x), \quad  & f\in \calC^{1,0}(\calM \times \R),
    \end{cases}  
\end{equation}
where $\calC^{1,0}(\calM \times \R)$ denotes the space of continuous functions $f:\calM \times \R\to \R $ in the product topology such that $f(\cdot,x)\in \calC^1(\calM)$ for all $x\in \R$. 
\end{definition} 
Evidently, the occupation derivative may be restricted to   $\calD$. When $f = f(\mo,x)$, we  write $f\in \calC^{1,2}(\calD)$ if $\delta_{\mo} f$, $\partial_x f$, $\partial_{xx} f$   exist and are jointly continuous.\footnote{Note also  that condition II. in \cref{def:linearderivative} is replaced by II'.  $|\delta_\mo  f(\mo,x)(y)| \le C(1+|y|^2)$ for some constant $C\in (0,\infty)$, uniformly in $(\mo,x)\in \calD$.}  
The motivation behind introducing the occupation derivative 
becomes clear when plugging in  the occupied process in \eqref{eq:defOccDerivative}. Given  $f\in \calC^{1}(\calD)$, then  
$\partial_{\mo} f(\calO_t,X_t)$ takes the linear derivative $\delta_\mo  f(\calO_t,X_t)(\cdot)$ and evaluates it at  $X_t$. 
In terms of directional derivative, this reads $$\partial_{\mo} f(\calO_t,X_t) =  \frac{d}{dh} f(\calO_t+h\delta_{X_t},X_t)|_{h=0}.$$
See \cref{fig:occDer}. The local nature of the occupation derivative comes directly from the  dynamics of the occupation flow  \eqref{eq:flowDynamics}, whence the terminology. 
As shown in \cref{sec:FITO}, the occupation derivative is intrinsically linked to the time derivative in the functional It\^o calculus \cite{DupireFITO}. Similar to the functional derivatives, $\partial_{\mo}$ and $\partial_x$ do not commute. Hence, the \textit{Lie bracket} $$[\partial_{\mo},\partial_x]f := \partial_{x\mo}f - \partial_{\mo x}f,$$ may be  nonzero in general. 
A simple example is $f(\mo,x) = \int_{\R}y \mo(dy)$, where $[\partial_{\mo},\partial_x]f = 1$. %
Next, we  see how the linear derivative and occupation measure interact with each other and provide a precise meaning of the chain rule \eqref{eq:chainRule}. 
 \begin{proposition}\label{prop:itoO} 
      If $X$ is a continuous semimartingale and 
      $f\in \calC^{1}(\calM)$, then
     \begin{align}
      df(\calO_t)  &=  \partial_{\mo} f(\calO_t,X_t) d\langle X \rangle_t,\label{eq:ITOO}\\[1em]
        df(\tilde{\calO}_t)  &=  \partial_{\mo} f(\tilde{\calO}_t,X_t) dt. 
        \label{eq:ITOOStd}
      \end{align}
  \end{proposition}
  \begin{proof}
      It is a byproduct of \cref{thm:ItoOX} below. 
  \end{proof}

         \begin{figure}[t]
\centering
\caption{The occupation derivative $\partial_{\mo}f(\calO_t,X_t)$ gives the sensitivity of $f$ with respect to a Dirac impulse at the spot $X_t$. } 
 \includegraphics[height=2.1in,width=3.0in]{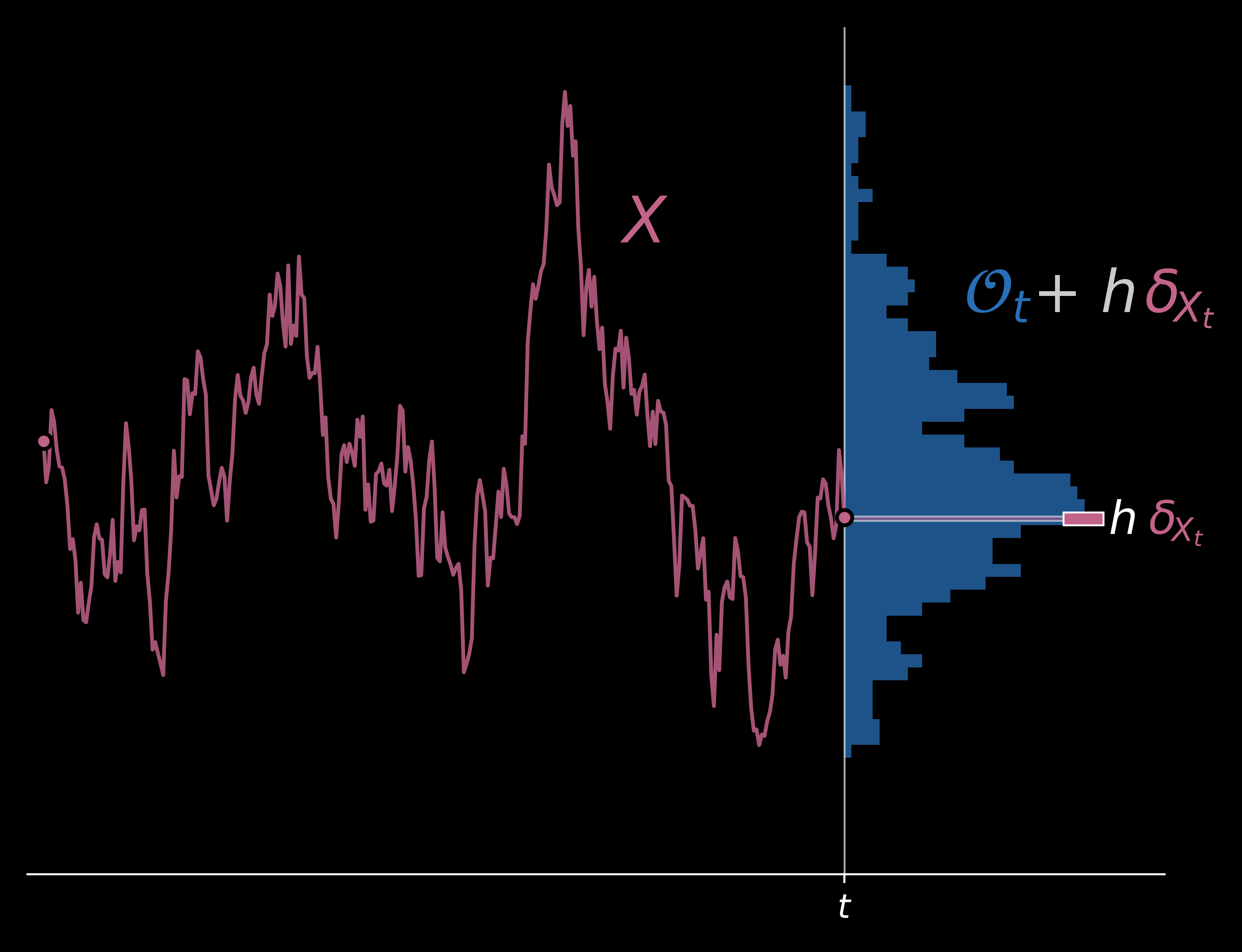}
 \label{fig:occDer}
 \end{figure}
      
\begin{example}\label{ex:linear}
Consider the linear functional $f(\mo) = \phi\cdot \mo$. In light of the occupation time formula \eqref{eq:OTF}, $f(\calO_t) = \int_{0}^t \phi(X_u) d\langle X \rangle_u$. Since  $\delta_\mo  f(\mo) \equiv  \phi$, 
 \begin{equation*}
    f(\calO_{t }) - f(\calO_{0}) =  \int_0^1  \delta_\mo f(\eta \calO_{t}) \cdot \calO_{t} d\eta = \int_{\R} \phi(x) \calO_{t}(dx) = \int_0^t\phi(X_u)d\langle X \rangle_u, 
\end{equation*}
which indeed yields \eqref{eq:ITOO}.  
\end{example}

\begin{example}\label{ex:t}
Consider the trivial but key  example $f(\mo) = g(\mo(\R))$ for some $g\in \calC^1([0,T])$. Then $f(\tilde{\calO}_t) = g(t)$, and as   $\mo(\R) = \mathds{1} \cdot \mo$ with $\mathds{1}(x)\equiv 1$, we obtain from \cref{ex:linear}   
and the chain rule that 
$$\partial_{\mo} f(\tilde{\calO}_t,X_t) = \delta_{\mo} f(\tilde{\calO}_t)(X_t) = g'(\tilde{\calO}_t(\R)) \ \underbrace{\delta_{\mo}(\mathds{1} \cdot \tilde{\calO}_t)(X_t)}_{= \ 1} = g'(t).$$
Applying \eqref{eq:ITOOStd}, we  recover the fundamental theorem of calculus $g(t)-g(0) = \int_0^t g'(u)du$. 
\end{example}

\begin{example}\label{ex:runMax} 
\textnormal{(Running maximum)} 
Consider the  functional
\begin{equation}
   M(\mo) =  \sup\Big\{ x \in \R \ : \ \frac{d\mo}{d\lambda}(x) > 0\Big\} = \sup \textnormal{supp}(\mo). 
\end{equation} 
Hence $M(\calO_t) = \sup_{s\le t} X_s$ is the maximum of  $X$ on $[0,t]$. If $t=0$, we use the convention that $\textnormal{supp}(\calO_0) = \{X_0\}$, i.e.,  $M(\calO_0) = X_0$. 
Note that  $M \notin \calC^1(\calM)$ as its linear derivative is discontinuous at the maximum $x=M(\mo)$ since $\lim_{y\uparrow  x} \delta_{\mo}M(\mo)(y) = 0 \ne 1 = \lim_{y\downarrow  x} \delta_{\mo}M(\mo)(y)$. 
We thus consider a mollified version of $M$ to apply \cref{prop:itoO}.  Given $\mo \in \calM^{\lambda}$, $\mo(\R)>0$, consider the smooth functional (akin to Laplace's principle in large deviations theory)  
\begin{equation}\label{eq:mollifMax}
    M^{\varepsilon}(\mo) =
        \varepsilon \log \avint_{\R} e^{x/\varepsilon}\mo(dx). 
\end{equation}
where we recall that $\avintin_A f do = \frac{1}{\mo(A)}\int_A f do$. 
It is straightforward to see that  $M^{\varepsilon} \to M$ pointwise. 
Moreover, using Examples \ref{ex:linear}, \ref{ex:t}, and the chain rule,   
\begin{equation}\label{eq:linDerMax}
    \partial_{\mo}M^{\varepsilon}(\mo,x) = \delta_{\mo}M^{\varepsilon}(\mo)(x) 
    = \frac{\varepsilon}{\mo(\R)}\left(e^{[x - M^{\varepsilon}(\mo)]/ {\varepsilon}} - 1\right). 
\end{equation}
Hence 
$M^{\varepsilon}\in \calC^{1}(\calM^{\lambda})$, and    $M^{\varepsilon}(\calO_t) = X_0 + \int_0^t \delta_{\mo}M^{\varepsilon}(\calO_s) \cdot d\calO_s$ using \cref{prop:itoO}. 
Next, a straightforward calculation shows that $\delta_{\mo}M^{\varepsilon}(\mo)$  converges,  as $\varepsilon \to 0$,  
to the Dirac mass at 
$M(\mo)$. 
If  $M_t := M(\calO_t)$, we thus have that 
$$M_t -X_0 = \lim_{\varepsilon \downarrow 0} M^{\varepsilon}(\calO_t) -X_0 =  
 \int_0^t \delta_{M_s}(X_s) d\langle X \rangle_s = L_t^0(M  - X), $$
where $L^0(M  - X)$ is the local time at zero of   
$M_t - X_t = \sup_{s\le t} X_s-X_{t}$. We thus  recover the  equality $M_t = X_0+ L^0_t(M  - X) \ \Q-a.s.$ established by \citet[Example 2]{DupireFITO},  generalizing   Lévy's distributional result.   
\end{example}

\subsection{It\^o's Formula}


\begin{theorem}\label{thm:ItoOX}
    \textnormal{{(Itô's formula)}} 
    Let $X$ be a continuous semimartingale and $\calO$ its flow of occupation measures. 
    If $f\in \calC^{1,2}(\calD)$, then  
       \begin{align}
       df(\calO_t,X_t) &= \Big(\partial_{\mo} + \frac{1}{2} \partial_{xx}\Big)f(\calO_t,X_t) d\langle X \rangle_t + \partial_x f(\calO_t,X_t)dX_t.\label{eq:itoOX}
    \end{align} 
Similarly, if $\tilde{\calO}$ is the calendar time occupation flow of $X$, then 
  \begin{align}\label{eq:itoOXstd}
       df(\tilde{\calO}_t,X_t) =  \partial_{\mo} f(\tilde{\calO}_t,X_t) dt   + \frac{1}{2} \partial_{xx}f(\tilde{\calO}_t,X_t) d\langle X \rangle_t + \partial_x f(\tilde{\calO}_t,X_t)dX_t.
    \end{align}
    
\end{theorem}

\begin{proof}
    See \ref{app:ItoOX}. 
\end{proof}


\begin{example}\label{ex:tx}
Let   $f(\mo,x) = g(\mo(\R),x)$,  $g\in \calC^{1,2}([0,T]\times \R)$. Then $f(\tilde{\calO}_t,X_t) = g(t,X_t)$. As in \cref{ex:t}, we have  $\partial_{\mo} f = \partial_t g$.  
As $\partial_x g =\partial_x f$, $\partial_{xx} g =\partial_{xx} f$, so we  recover from \eqref{eq:itoOXstd} the classical It\^o formula, 
\begin{align*}
    dg(t,X_t) &= \partial_t g(t,X_t) dt  + \partial_x g(t,X_t)dX_t + \frac{1}{2} \partial_{xx}g(t,X_t) d\langle X \rangle_t. 
\end{align*}
\end{example}

\begin{remark}\label{rem:multiD}
    Let $X = (X^1,\ldots,X^d)$ be a continuous semimartingale in $\R^d$, $d\ge 2$, with covariation process $\langle X \rangle = (\langle X^i,X^j \rangle)\in \R^{d\times d}.$ The occupation flow of $X$  is given  by  
    $\calO_t = \int_0^t \delta_{X_s}\textnormal{tr}(d\langle X \rangle_s)$, taking values in the space  $\calM_d$ of  Borel measures on $\R^d$.  Note that $\calO$ is still a single entity,  akin to  the  time variable. 
    \cref{thm:ItoOX} is then easily generalized: if $f\in \calC^{1,2}(\calM_d\times \R^d)$,  
      \begin{align*}
        df(\calO_t,X_t) &=  \partial_{\mo} f(\calO_t,X_t) \textnormal{tr}(d\langle X \rangle_t)  +   \frac{1}{2} \textnormal{tr}(\nabla^2 f(\calO_t,X_t) d\langle X \rangle_t) \\
        &+ \nabla f(\calO_t,X_t) \cdot  dX_t,\label{eq:itoGeneralmultiD}
    \end{align*}
    where $\nabla, \nabla^2$ denote  the gradient and Hessian operator. 
    The occupation derivative 
    is again given by $\partial_{\mo} f(\mo,x) = \delta_{\mo}f(\mo,x)(x)$. 
    A general It\^o's formula with arbitrary stochastic clock can be found in \cite[Proposition 3.3]{SonerTissotZhang}. 
\end{remark}

  \subsubsection{Connection with the Functional It\^o Calculus}\label{sec:FITO} 
  
We first recall the functional It\^o formula, established by \citet{DupireFITO} and further studied   by  \citet{ContBally, ContFournie}. Adopting the canonical setting, let  $\Omega_T$ be the space of càdlàg paths  $\omega:[0,T] \to \R$ for some $T>0$, and introduce the coordinate process $X_t(\omega) = \omega_t$.  Fix also a non-anticipative path functional $\textnormal{f}:[0,T]\times \Omega_T \to \R$ meaning that $\textnormal{f}(t,\omega)$ only depends on $\omega$ up to time $t$.  If $\textnormal{f}$ is regular enough, we then have from \cite{DupireFITO} that 
     \begin{equation}\label{eq:FITO}
        d\textnormal{f}(t,\omega) = \Delta_t\textnormal{f}(t,\omega)dt  
        + \Delta_{x}\textnormal{f}(t,\omega) dX_t(\omega) +   \frac{1}{2}\Delta_{xx} \textnormal{f}(t,\omega) d\langle X \rangle_t(\omega),  \quad \Q-a.s.,
    \end{equation}
    where   $\Delta_t\textnormal{f}(t,\omega) = \frac{d}{dh} \textnormal{f}(t+h,\omega_{\cdot \wedge t})|_{h=0}$,  $\Delta_x\textnormal{f}(t,\omega) = \frac{d}{dh} \textnormal{f}(t,\omega + h \mathds{1}_{[t,T]})|_{h=0}$. 
To connect \eqref{eq:FITO} with \cref{thm:ItoOX}, fix $f \in \calC^{1,2}(\calM\times \R)$ and introduce the functional 
 $$\textnormal{f}(t,\omega) = f(\tilde{\calO}_t,X_t)(\omega) = f(\tilde{\calO}_t(\omega),\omega_t).$$ 
 First, note that  $\Delta_x \textnormal{f} = \partial_x f$ and $\Delta_{xx}\textnormal{f} = \partial_{xx} f$ 
as a vertical bump of $\omega$ at $t$ does not affect the Lebesgue integrals $\int_0^t \mathds{1}_{\cdot} (\omega_s)ds$, therefore leaving $\tilde{\calO}_t$ unchanged.   
For the functional time derivative, we have for all $h\in [0,T-t]$ that 
\begin{align*}
\textnormal{f}(t+h,\omega_{\cdot \wedge t}) = f(\tilde{\calO}_{t+h}(\omega_{\cdot \wedge t}), \omega_t)  
= f(\tilde{\calO}_{t}(\omega), \omega_t) + h\int_0^1  \delta_\mo  f(\tilde{\calO}_t(\omega) + \eta h \delta_{\omega_t},\omega_t) \cdot \delta_{\omega_t} d\eta, 
\end{align*}
using that $\tilde{\calO}_{t+h}(\omega_{\cdot \wedge t}) - \tilde{\calO}_{t}(\omega) = h\delta_{\omega_t}$ since $\omega$ is frozen on $[t,t+h]$.  Taking derivative with respect to $h$ and letting $h\downarrow 0$ leads to 

\begin{equation}\label{eq:OCCvsFITO}
 \Delta_t \textnormal{f}(t,\omega) = \partial_{\mo} f(\tilde{\calO}_t,X_t)(\omega),
\end{equation}
further confirming that $\mo$ replaces time. 
All in all, we indeed obtain that 
\begin{align*}
df(\tilde{\calO}_t,X_t)(\omega)  &= d\textnormal{f}(t,\omega) \\
 &=    \Delta_t\textnormal{f}(t,\omega)dt  
        + \Delta_{x}\textnormal{f}(t,\omega) dX_t(\omega) +   \frac{1}{2}\Delta_{xx} \textnormal{f}(t,\omega) d\langle X \rangle_t(\omega)\\
&=  \left(\partial_{\mo} f(\tilde{\calO}_t,X_t) dt+ \partial_x f(\tilde{\calO}_t,X_t)dX_t  + \frac{1}{2} \partial_{xx}f(\tilde{\calO}_t,X_t) d\langle X \rangle_t\right)(\omega).
\end{align*}

\subsubsection{Connection with the Malliavin Calculus}\label{sec:Malliavin}
 
 This section regards terminal occupation functionals $\calO_T \mapsto f(\calO_T)$ through the lens of Malliavin calculus; see, e.g.,  \cite{DiNunno}.  
 Suppose that $\Q$ is the Wiener measure 
 so that the canonical process $X$ is Brownian motion. Then $\calO_T = \tilde{\calO}_T = \int_0^T \delta_{X_t} dt$ $\Q-$a.s. due to  $\langle X \rangle_t = t$.  We  then obtain the following result.

\begin{proposition} \label{prop:Malliavin} Let $f\in \calC^{1}(\calM)$ such that $x\mapsto  \partial_\mo f(\mo,x)$ is differentiable for every $\mo\in \calM^{\lambda}$, and $\textnormal{f}(\omega) :=f(\calO_T(\omega))$, $\omega\in \Omega_T$, is Malliavin differentiable (see \cite{DiNunno})). Then, the Malliavin derivative of $\textnormal{f}$ at $t\in [0,T]$ admits the alternative expression 
\begin{equation}\label{eq:malliavin}
        D_{t}\textnormal{f}(\omega)  := \frac{d}{dh}\textnormal{f}(\omega + h\mathds{1}_{[t,T]})\Big |_{h=0}  = \int_t^T \partial_{x\mo}f(\calO_T(\omega),\omega_s)ds, \quad \Q-a.s. 
    \end{equation}
\end{proposition} 
\begin{proof}
    See \cref{app:Malliavin}. 
\end{proof}

\begin{remark}
Introduce  the Fr\'echet derivative  $\frakD \! \textnormal{f} \in L^2([0,T])$ of  $\textnormal{f}(\omega) =f(\calO_T(\omega))$, defined  via   $\frac{d}{dh}\textnormal{f}(\omega + h\omega') |_{h=0} =  
 (\frakD \! \textnormal{f}(\omega), \omega')_{L^2([0,T])}$  $\forall \omega,\omega' \in \Omega_T$. Then  \eqref{eq:malliavin} implies that 
   $$\frakD_t \! \textnormal{f}(\omega) = - \partial_tD_t\textnormal{f}(\omega) 
   = \partial_{x\mo}f(\calO_T(\omega),\omega_{t}).$$ 
Interpretating the path functional  $\textnormal{f}$ as a lift of $f$, i.e.,   $\textnormal{f}(\omega) = f(\mo)$ if $\calO_T(\omega) = \mo$, 
then the relation 
$\frakD \! \textnormal{f} = \partial_{x\mo}f  = \partial_{x} \delta_{\mo}f $  
is analogous to the link between the \textit{Lions} derivative and  $\delta_{\mo}$ in mean field games  \cite{CardDelLasLio,CarmonaDelarue,LasryLions}.
\end{remark}

\subsection{Feynman-Kac's Formula for the  Backward Heat Equation}\label{sec:FKHeat}
We first  prove a Feynman-Kac theorem for occupied Brownian motion, while a more general result will be given in  \cref{sec:OSDE}. 
Write $\varphi\in \calC^{2}(\calD)$ if $\varphi \in \calC^{2,2}(\calD)$  and $(x,y)\mapsto \delta_\mo \varphi(\mo,x)(y) $,   $(y,y')\mapsto \delta_{\mo\mo} \varphi(\mo,x)(y,y')$ are twice continuously differentiable with bounded derivatives. 

\begin{theorem}\label{thm:FK}
\textnormal{\text{(Feynman-Kac)}}  Let $\varphi:\calD\to \R$ such that $\varphi(\calO_T,X_T) \in L^1(\Q)$ and consider  the  backward heat equation, 
\begin{subnumcases}{}
\Big(\partial_{\mo} + \frac{1}{2}\partial_{xx}\Big) u = 0, & \text{ on } \ \ $\calD_{T} := \{(\mo,x) \in  \calD  :  \mo(\R) < T\}$, \label{eq:heatPDE}\\[0.2em]
      u = \varphi, & \text{ on } \hspace{-0.75mm} $\partial \! \calD_{T} := \{(\mo,x) \in  \calD  :  \mo(\R) = T\}$. \label{eq:heatTerminal}
\end{subnumcases}
If $v\in \calC^{1,2}(\calD)$ is a classical solution of  \eqref{eq:heatPDE}$-$\eqref{eq:heatTerminal}, then  
\begin{equation}\label{eq:valueFctal}
    v(\mo,x) = \E_{\mo,x}^{\Q}[\varphi(\calO_T,X_T)]  := \E^{\Q}[\varphi(\mo+\calO_{T-t}^x,x + X_{T-t})], \quad \mo(\R) = t, 
\end{equation}
where  $\calO_{T-t}^x(A) = \calO_{T-t}(A-x) $. 
Conversely, if $\varphi\in \calC^{2}(\calD)$, 
then $v$ in \eqref{eq:valueFctal} belongs to $\calC^{2}(\calD)$ as well and is the unique classical solution of \eqref{eq:heatPDE}$-$\eqref{eq:heatTerminal}.  
 \end{theorem}
 \begin{proof}
     See \ref{app:FKBM}.
 \end{proof}

We finally discuss several examples where $\varphi$ has little regularity while the value functional is still smooth. As we shall see, this is in contrast with other path-dependent calculi 
where the  associated value functional typically fails to be even continuous. 
\begin{example}\label{ex:LT0}
\textnormal{(Local time at zero)} 
Let $X$ be Brownian motion and  $\varphi(\mo,x) =\frac{d\mo}{d\lambda}(0)$, irrespective of $x$. Hence  $\varphi(\calO_T,X_T) = L_T^0$ is the local time at $0$. 
If $(\mo,x)\in \partial\calD_t$, a simple calculation shows that 
  \begin{align*}
            v(\mo,x) &= \E^{\Q}_{\mo,x}[L_T^{0}] 
             =   \E^{\Q}_{\mo,x}\left[L_t^0 + L_{t,T}^{0}\right]  
               = \frac{d\mo}{d\lambda}(0)  +  \E^{\Q}[L_{T-t}^{|x|}].  
        \end{align*}
         The first term in the last equality is linear in $\mo$ and Lipschitz continuous with respect to the metric $\frakm_{\infty}$ seen in \cref{sec:occProcess}. 
         Moreover, the second term is continuous in $x$  and even piecewise smooth on $\R\setminus \{0\}$ for all $t<T$. From  the pathwise setting  in \cref{sec:FITO}, we note that the functional 
         \begin{equation}\label{eq:loc0Fctal}
             \textnormal{v}(t,\omega) = \E^{\Q}[L_T^{0}  |   \calF_t](\omega) = L^0_t(\omega) + \E^{\Q}\big[L_{T-t}^{|x|}\big]\big |_{x=\omega_t}, \quad \omega\in \Omega_T, 
         \end{equation} 
         is \textit{not} continuous in $\omega$ (with respect to the supremum norm), regardless of the definition of  $L^0_t:\Omega_T \to \R$. 
         First, one may set 
         $L^0_t(\omega) = \lim_{\varepsilon \to 0} \frac{1}{2\varepsilon}\int_0^t \mathds{1}_{B_{\varepsilon}}(\omega_s)ds $ as the  occupation measures $\calO,\tilde{\calO}$ almost surely coincide under the Wiener measure. 
         However, if $(\omega^N)$ is a sequence of piecewise linear paths   that converges to  $\omega$ (e.g. using Schauder's theorem) it may  happen that $L^0_t(\omega^N) = \infty$ as the occupation measure of $\omega^N$ is not absolutely continuous with respect to $\lambda$. 
         Another possibility  is 
         $L^0_t(\omega) := \lim_{\varepsilon \to 0} \frac{1}{2\varepsilon}\int_0^t \mathds{1}_{B_{\varepsilon}}(X_s)d\langle X \rangle_s(\omega) $ 
         when $\langle X \rangle$ is well-defined. But  $L^0_t(\omega^N) = 0$ for the above sequence $(\omega^N)$, leading again to the discontinuity of $L^0_t$ with respect to $\omega$.  
         Finally, we could derive  the local time at zero (or any other level) from Tanaka's formula 
         \cite[Chapter VI]{RevuzYor}. 
   However, the expression would involve stochastic integrals  which  are infamously discontinuous with respect to the input path. This representation  is thus subject to the same fate as the previous ones. 
\end{example}
\begin{example} \label{ex:heatspotLT} \textnormal{(Spot Local time) } Let $\varphi(\mo,x) = \frac{d\mo}{d\lambda}(x)$, that is $\varphi(\calO_t,X_t) = L_t^{X_t}$. Then $\varphi$ is continuous in $x$, since $X$ is a continuous martingale, and also continuous in  $\mo$  with respect to $\frakm_{\infty}$. 
However, $\varphi$ is  not differentiable in $x$.  
For every $(\mo,x) \in \partial\calD_t$, $t<T$, let us  write the value function as  
  \begin{align*}
            v(\mo,x) = \E^{\Q}_{\mo,x}[L_T^{X_T}] 
              = \int_{\R} \frac{d\mo}{d\lambda}(y) \phi_{T-t}(x-y) dy +  \E^{\Q}\big[L_{T-t}^{X_{T-t}}\big],
        \end{align*}
        where $\phi_{\tau}$ is the density 
        of $\calN(0,\tau).$ 
        From $\frac{d\mo}{d\lambda}(y)dy = \mo(dy)$ and \cref{prop:EUPrice} below,  we obtain 
    \begin{equation}\label{eq:condexpLT}
              v(\mo,x)= \int_{\R} \phi_{T-t}(x-y) \mo(dy) + \sqrt{\frac{2}{\pi}(T-t)}. 
          \end{equation}
          Since $x\mapsto \phi_{T-t}(x-y)$ is smooth for all $y\in \R,$ $t<T$, and the second term is  independent of $x$, then $v$ is smooth in $x$. For the occupation derivative, recall that $t = \mo(\R)$ as $(\mo,x) \in \partial\calD_t$,  and $\delta_{\mo}(\mo\mapsto \mo(\R)) \equiv 1$. Hence,  
          \begin{align*}
            \partial_{\mo}v(\mo,x)  &= \int_{\R} \partial_t\phi_{T-t}(x-y) \mo(dy) + \underbrace{\phi_{T-t}(0)  -\big(2\pi (T-t)\big)^{-1/2}}_{=0}\\
            &= -\frac{1}{2}\int_{\R} \partial_{xx}\phi_{T-t}(x-y) \mo(dy),
          \end{align*}
          since  $\phi$ solves the heat equation $\partial_{\tau} \phi_{\tau} = \frac{1}{2}\partial_{xx}\phi_{\tau}$. 
          We conclude that $v$ belongs to $\calC^{1,2}(\calD)$  and  is thus a classical solution of \eqref{eq:heatPDE}$-$\eqref{eq:heatTerminal}. 
         On the other hand, the  path functional $\textnormal{v}(t,\omega) = \E^{\Q}[L_T^{X_T}  |   \calF_t](\omega)$ 
         is  discontinuous in  $\omega$  similar to \cref{ex:LT0}. 
\end{example}

\subsection{Occupied SDEs and Dirichlet Problems} \label{sec:OSDE}
 Consider the  \textit{occupied stochastic differential equation (OSDE)}, 
\begin{subnumcases}{}
d\calO_t =  \delta_{X_t} d\langle X \rangle_t,  & $\calO_0 \in \calM$, \label{eq:OSDEOcc}\\[0.2em] 
 dX_t  =   b(\calO_t,X_t)dt + \sigma(\calO_t,X_t)dW_t, \; & $X_0\in \text{supp}(\calO_0),$
\label{eq:OSDE}
\end{subnumcases}
with  coefficients $b,\sigma: \calD\to \R$. 
Note that \eqref{eq:OSDEOcc} can be equivalently written as $d\calO_t = \delta_{X_t} \sigma^2(\calO_t,X_t) dt$.  
We first discuss the wellposedness of  \eqref{eq:OSDEOcc}$-$\eqref{eq:OSDE}. 
Similar to classical SDEs, the coefficients $b,\sigma$ must satisfy Lipschitz and growth conditions with respect to an appropriate metric for $\mo$. While the 
  metrics $\frakm_p$ in \cref{sec:occProcess} may lead to similar results, it proves convenient  to use a distance with explicit dual representation, such as  the \textit{bounded Lipschitz} 
\textit{distance}    \cite{Hanin}. 
 Let $\phi:\R\to \R$ be a bounded Lipschitz function and $[\phi]_{\text{Lip}}$ denote its  Lipschitz constant. If  $\lVert \phi \rVert_{\textnormal{\tiny{BL}}} =\lVert\phi \rVert_{L^\infty(\R)} \vee [\phi]_{\text{Lip}}$ and $\mo$ is a signed Radon measure   on $\R$,  define its  bounded Lipschitz norm  by 
 \begin{equation}\label{eq:MongeKantorovic}
    |\mo|_{\textnormal{\tiny{BL}}} = \sup\{\phi \cdot \mo \ : \  \lVert \phi \lVert_{\textnormal{\tiny{BL}}} \ \le \  1 \}.
 \end{equation}
Also,  introduce  the parabolic norm  $\varrho(\mo,x) = \sqrt{|\mo|_{\textnormal{\tiny{BL}}} +|x|^2}$, $(\mo,x)\in \calD$. 
\begin{theorem}
\label{thm:OSDE}
Suppose that $b,\sigma\in \calC(\calD)$  satisfy the following Lipschitz  and growth conditions: there exists $K>0$ such that for all $(\mo,x),(\mo',x')\in \calD$, 
\begin{align}
     |b(\mo,x)|^2 &\le K^2[1+\varrho(\mo,x)^2], \quad  |\sigma(\mo,x)|^2 \le K^2[1+\varrho(\mo,x)], \label{eq:OSDEGrowth}\\[1em]  
|f(\mo',x')-f(\mo,x)| &\le K\varrho(\mo'-\mo,x'-x), \quad f\in \{b,\sigma\}. \label{eq:OSDELip}
\end{align}
   Then the OSDE \eqref{eq:OSDEOcc}$-$\eqref{eq:OSDE} 
   admits a unique strong solution. 
\end{theorem}
\begin{proof}
    See \cref{app:OSDE}. 
\end{proof}
\noindent
Notice the unusual linear growth condition for the \textit{squared} diffusion coefficient $\sigma^2$ in \eqref{eq:OSDEGrowth} instead of $\sigma$  as in the classical case \cite[Chapter IX]{RevuzYor}. This is because $\sigma^2(\calO,X)$ appears in the dynamics of the occupation flow.  

OSDEs
 builds on concepts found in self-attracting diffusions \cite{CranstonLeJan} inspired by  Brownian polymer models   \cite{DurrettRogers}.     
Therein, the drift is influenced by the
the calendar time occupation flow so as to create ergodic behaviors. 
These diffusions were further studied by Raimond \cite{Raimond}, and extended in \cite{Benaim2,Benaim3,Benaim4,Benaim1}   to   \textit{self-interacting diffusions}, also including self-repelling SDEs. While these  diffusions use the occupation flow in the drift,  OSDEs involve the occupation in the diffusion coefficient as well. 
Let us also mention the recent papers \cite{ChassagneuxPages,DuJiangLi}  linking self-interacting diffusions to ergodic McKean-Vlasov SDEs.

\begin{theorem}\label{thm:ItoOSDE}
    \textnormal{\text{(Itô's formula for  OSDEs)}} 
    Let $(\calO,X)$ be a strong solution of the OSDE \eqref{eq:OSDEOcc}$-$\eqref{eq:OSDE} with coefficients $b,\sigma$. 
    If $f\in \calC^{1,2}(\calD)$, then
\begin{align}\label{eq:itoOSDE}
        df(\calO_t,X_t)  &=   \big(\sigma^2\partial_{\mo} + b \partial_x + \frac{\sigma^2}{2}\partial_{xx}\big) f(\calO_t,X_t) dt + (\sigma \partial_x) f(\calO_t,X_t)dW_t,
\end{align} 
\end{theorem}
\begin{proof}
    This is a direct consequence of \cref{thm:ItoOX}.
\end{proof}
Next, we  extend \cref{thm:FK} to OSDEs. If $(\calO^{\mo,x},X^{\mo,x})$ denotes the solution of the OSDE \eqref{eq:OSDEOcc}$-$\eqref{eq:OSDE}   
given 
$(\calO_t,X_t) = (\mo,x)$, 
it is tempting to define 
$v(\mo,x) = \E^{\Q}[\varphi( \calO_{T}^{\mo,x},X_{T}^{\mo,x})]$  as in \eqref{eq:valueFctal}. 
However, 
 unlike Brownian motion or any continuous Gaussian martingale where $\calO_T(\R) = \langle X \rangle_T$ is  deterministic, the total mass  may well be random. Take for instance $b\equiv 0$ and $\sigma(\mo,x) = x$.   
The terminal condition is  therefore misspecified in general. 
A remedy is to replace the  horizon $T$ by the  time at which 
the total mass $\calO(\R)$ first exceeds a 
threshold, 
say $\Theta >0$.   
 With this new interpretation, the terminal date becomes stochastic and corresponds to the exit time 
\begin{equation}\label{eq:ExitTime}
    \tau_{\Theta}  =  \inf\{t\ge 0 \ :  (\calO_t,X_t) \notin \calD_\Theta\} =  \inf\{t\ge 0 \ :  \langle X \rangle_t \ge \Theta\}.
\end{equation} 
This is precisely the logic behind the DDS transformation in  \cref{rem:DDS}.  
We  now state a Feynman-Kac's theorem for OSDEs. 
Its proof is almost identical to  the first part in \cref{thm:FK}. 

\begin{theorem} \textnormal{\text{(Feynman-Kac)}}\label{thm:FKDirichlet}
  Let  $\varphi:\calD\to \R$ such that $\varphi(\calO_T,X_T) \in L^1(\Q)$ and  consider  the Dirichlet problem 
\begin{subnumcases}{}
\big(\sigma^2\partial_{\mo} + b \partial_x + \frac{\sigma^2}{2}\partial_{xx}\big)u(\mo,x) = 0, & \text{ on } \ \ $\calD_{\Theta},$ \label{eq:dirPDE}\\[0.2em] 
u(\mo,x)  = \varphi(\mo,x), & \text{ on } \hspace{-0.8mm} $\partial \! \calD_{\Theta},$
\label{eq:dirTerminal}
\end{subnumcases}
If $(\calO,X)$ is  a strong solution of  \eqref{eq:OSDEOcc}$-$\eqref{eq:OSDE}, $\E^{\Q}[\tau_\Theta]<\infty$,  and $v\in \calC^{1,2}(\calD)$ is a   classical solution  \eqref{eq:dirPDE}$-$\eqref{eq:dirTerminal}, then 
\begin{equation}\label{eq:FKDirichletExp}
    v(\mo,x) = \E^{\Q}[\varphi(\calO_{\tau_{\Theta}}^{\mo,x},X_{\tau_{\Theta}}^{\mo,x})],  \quad (\mo,x) \in \overline{\calD_{\Theta}}. 
\end{equation}
\end{theorem}

The condition $\E^{\Q}[\tau_\Theta]<\infty$ in \cref{thm:FKDirichlet} holds for instance if 
$\sigma(\mo,x)\ge \underline{\sigma}$ for some  constant $\underline{\sigma}>0$. As $\calO_t(\R) = \langle X \rangle_t = \int_0^t\sigma(\calO_s,X_s)^2ds$, it is then immediate that $\tau_{\Theta} \le \Theta/\underline{\sigma}^2$. 

\begin{remark}
    When $b\equiv 0$, then \eqref{eq:dirPDE} reads 
$\sigma^2(\partial_{\mo} + \frac{1}{2}\partial_{xx})u = 0$ 
which is equivalent to the heat equation \eqref{eq:heatPDE}$-$\eqref{eq:heatTerminal} when  $\sigma$ is positive. 
This is not surprising since $X$ 
is then  a local martingale, hence a time-changed Brownian motion; see  \cref{rem:DDS} and \cref{ex:timer}. 
\end{remark}

We finish this section by stating Feynman-Kac's theorem for functionals of the calendar time occupation flow. To this end, consider  
\begin{subnumcases}{}
d\tilde{\calO}_t =  \delta_{X_t} dt,  & $\tilde{\calO}_0 \in \calM$, \label{eq:OSDEOccStd}\\[0.2em] 
 dX_t  =   b(\tilde{\calO}_t,X_t)dt + \sigma(\tilde{\calO}_t,X_t)dW_t, \; & $X_0\in \text{supp}(\tilde{\calO}_0).$
\label{eq:OSDEStd}
\end{subnumcases}
The wellposedness of \eqref{eq:OSDEOccStd}-\eqref{eq:OSDEStd} may be established through a straightforward adaptation of \cref{thm:OSDE}; see Theorem 4.2.14 in \cite{TissotThesis}. 

\begin{theorem} \textnormal{\text{(Feynman-Kac, calendar time)}}\label{thm:FKDirichletStd}
   Consider  $\varphi:\calD\to \R$ such that  $\varphi(\tilde{\calO}_T,X_T) \in L^1(\Q)$  and   the backward occupied PDE
\begin{subnumcases}{}
\big(\partial_{\mo} + b \partial_x + \frac{\sigma^2}{2}\partial_{xx}\big)u(\mo,x) = 0, & \text{ on } \ \ $\calD_{T},$ \label{eq:PDEStd}\\[0.2em] 
   u(\mo,x) = \varphi(\mo,x), & \text{ on } \hspace{-0.8mm} $\partial \! \calD_{T},$
\label{eq:dirTerminalStd}
\end{subnumcases}
If \eqref{eq:OSDEOccStd}$-$\eqref{eq:OSDEStd} admits  a  strong  solution and  \eqref{eq:PDEStd}$-$\eqref{eq:dirTerminalStd} has a  solution $v\in \calC^{1,2}(\calD)$, then 
\begin{equation}\label{eq:FKDirichletExpStd}
    v(\mo,x) = \E^{\Q}[\varphi(\tilde{\calO}_{T}^{\mo,x},X_{T}^{\mo,x})],  \quad (\mo,x) \in \overline{\calD_{T}}. 
\end{equation}
\end{theorem}
 The above results allows to characterize the pricing PDE of  financial derivatives contingent on  the calendar time occupation flow;  see \cref{sec:occPayoffs}. 

\begin{remark}
Similar to path-dependent PDEs, the occupied PDEs \eqref{eq:dirPDE}$-$\eqref{eq:dirTerminal},  \eqref{eq:PDEStd}$-$\eqref{eq:dirTerminalStd} typically do not admit a classical solution.
Fortunately, a viscosity theory  can be effectively established, even in the fully nonlinear case; see the follow-up paper \cite{SonerTissotZhang}. In particular, it is shown therein that,  under mild assumptions,  the value function $v$ in \eqref{eq:FKDirichletExp} is the unique viscosity solution of \eqref{eq:dirPDE}$-$\eqref{eq:dirTerminal}.
\end{remark}

%% file: SpotLT.tex
\section{Stopping Spot Local Time}\label{sec:spotLT}

This section intends to study the spot local time $L_t^{X_t}$ in more depth, particularly in the context of optimal stopping (OS). We dwell on this problem as we believe it is highly  challenging both theoretically and numerically. 
Hence, the stopping of spot local time can be used as a  benchmark for algorithms solving  path-dependent OS  problems. 
For instance, it would be worth checking the performance  on this example 
of  recent methods   involving deep learning and/or the path signature \cite{Bayer,bayraktar,Becker}.  
See \cref{sec:American} for further numerical considerations. 

\subsection{Optimal Stopping}\label{sec:OS}
Let us first adapt standard results in optimal stopping to occupied processes. 
Suppose  that $X$ is Brownian motion and refer to \cref{rem:DDS} for the general semimartingale case.  Given a reward functional $\varphi:\calD \to \R_+$ such that $\sup_{t\le T} \varphi(\calO_t,X_t) \in L^1(\Q)$ and   exercise dates $\calT \subseteq [0,T]$ containing $T$, consider the \textit{optimal stopping (OS) problem} \cite{PeskirShiryaev},  
\begin{equation}\label{eq:OS}
    \sup_{\tau \in \vartheta(\calT)} \E^{\Q}[\varphi(\calO_{\tau},X_\tau)],
\end{equation}  
where $\vartheta(\calT)$ is the set of $\calT-$valued $(\F,\Q)-$stopping times. 
Define the \textit{value function}, 
    \begin{equation}\label{eq:valuefct}
    v(\mo,x) = \sup_{\tau \in \vartheta(\calT_t)} \E^{\Q}_{\mo,x}[\varphi(\calO_{\tau}, X_{\tau}) ], \quad \calT_t = \calT \cap \ [t,T], \quad t = \mo(\R).
\end{equation} The \textit{stopping region} is traditionally defined as the coincidence set  $\calS_T = \{(\mo,x) \in  \calD_T : v(\mo,x) = \varphi(\mo,x)\}$. In particular,   $\partial\calS = \partial\calD_T$ since $T\in \calT$. 
As it is challenging to  
 represent $\calS$ either mentally or  visually, 
we can regard the stopping region as a subset of $\R$  
for a given occupation measure, i.e.,    $\calS(\mo) = \{x\in \text{supp}(\mo) \ : \ (\mo,x) \in \calS \}$.
For every $t
\in [0,T]$ , this gives the pathwise partition 
$$\{X_s  :  s\le t\} = \text{supp}(\calO_t) = \calS(\calO_t) \cup \calC(\calO_t),$$
with the \textit{continuation region} $\calC(\mo) = \{x \in \text{supp}(\mo)\ : \ v(\mo,x) > \varphi(\mo,x)\}$. See \cref{fig:stopcont} and  \cite[Chapter 7]{Detemple} for a related comment.  


The Dynamic Programming Principle (DPP) states that for all $\varsigma \in \vartheta(\calT_t)$, 
\begin{align}\label{eq:DPP}
    v(\mo,x) &= \sup_{\tau \in \vartheta(\calT_t)} \ \E^{\Q}_{\mo,x}[ \varphi(\calO_{\tau},X_{\tau}) \mathds{1}_{\{ \tau < \varsigma \}} + v(\calO_\varsigma,X_{\varsigma},)\mathds{1}_{\{ \tau \ge \varsigma\}} ].  
\end{align}
Let us  recall how to    
 derive the  dynamic programming equation satisfied by $v$. 
 If $v\in \calC^{1,2}(
 \calD)$ and  $Y_t := v(\calO_t,X_t)$ denotes the \textit{Snell envelope} of $\varphi$,   
then  Itô's formula (\cref{thm:ItoOX}) gives 
\begin{align} \label{eq:itoSnell}
    dY_t &= \Big( \partial_\mo + \frac{1}{2}\partial_{xx} \Big) v(\calO_t,X_t) dt + \partial_x v(\calO_t,X_t)dX_t. 
\end{align}
It is classical that the Snell envelope $Y$ is the smallest supermartingale dominating  $\varphi$ \cite{PeskirShiryaev}.  
In particular,  the drift in \eqref{eq:itoSnell} is $\Q-$a.s. non-positive and equal to zero outside of the stopping region $\calS$. This  leads to the \textit{dynamic programming equation}, 
 \begin{subnumcases}{}
\Big( \partial_\mo + \frac{1}{2}\partial_{xx} \Big) v = 0  & \; \text{on} \; $\calD_T$, \label{eq:DPECont}\\[0.2em] 
v = \varphi & \; \text{on} $\ \calS_T \cup \ \partial \calD_T. $
\label{eq:DPEStop}
\end{subnumcases}
The present discussion is of course formal as the value function is typically not regular enough to be a classical solution of \eqref{eq:DPECont}$-$\eqref{eq:DPEStop}.  Nevertheless, one can adapt the arguments from our companion paper \cite{SonerTissotZhang} to show that $v$ is the unique viscosity solution  of the obstacle problem \eqref{eq:DPECont}$-$\eqref{eq:DPEStop}. 





 \begin{figure}[t]
\centering
\caption{Occupation flow $\calO$ and stopped spot local time $L_{\tau}^{X_{\tau}}$. Two simulations.}
\vspace{-2mm}

\begin{subfigure}[b]{0.47\textwidth}
     \centering
\includegraphics[height=2.1in,width=2.6in]{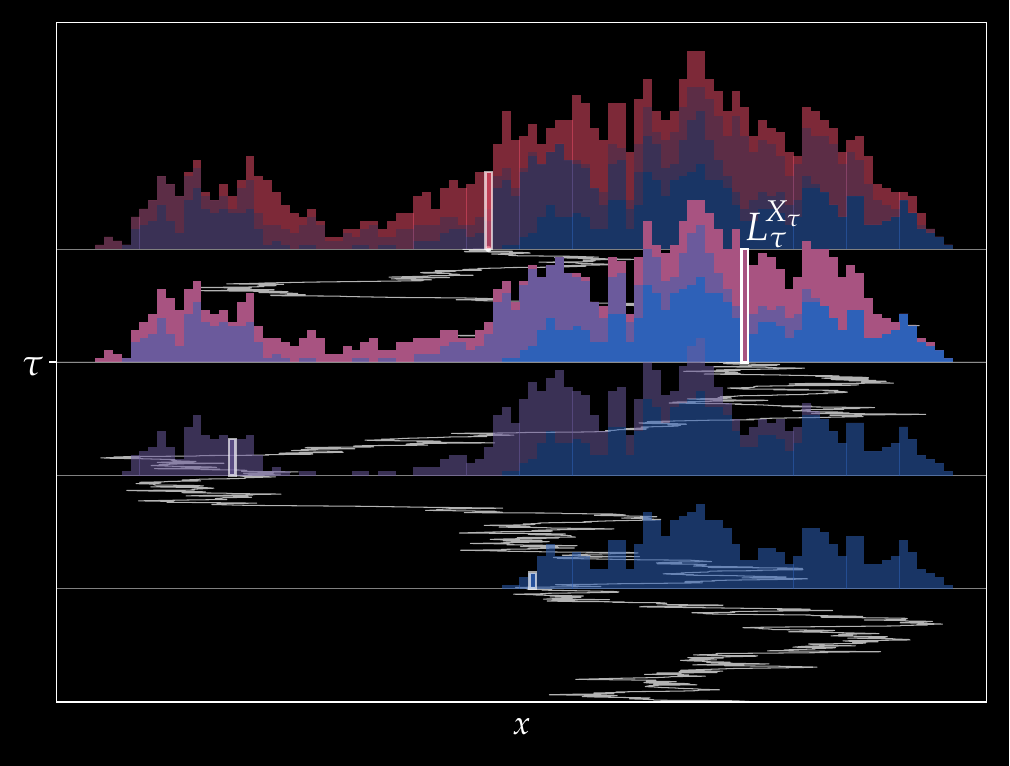}
\end{subfigure}
\begin{subfigure}[b]{0.47\textwidth}
     \centering
\includegraphics[height=2.1in,width=2.6in]{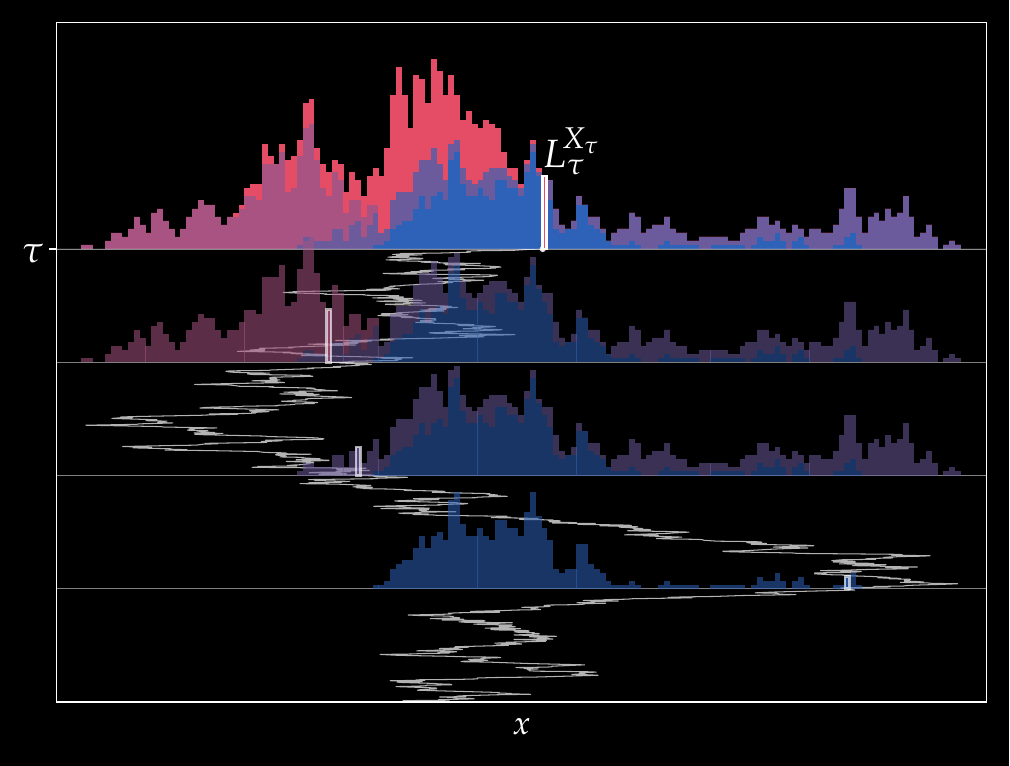}
 \end{subfigure}
\vspace{5mm}
 \label{fig:OSLT}
 \end{figure}
 
\subsection{Spot Local Time and Early Exercise Premium}\label{sec:spotLTpremium}

We are now set to study  the OS problem, 
\begin{equation}\label{eq:OSLT}
  \text{maximize } \;  v(\tau) := \E^{\Q}[L_{\tau}^{X_\tau}] \;\; \text{ over } \;\tau \in \vartheta(\calT).  
\end{equation}
The spot local time   corresponds to the reward $\varphi(o,x) = \frac{do}{d\lambda}(x)$ which is defined on $\calD^{\lambda}$, namely those elements $(o,x)$ of $ \calD$ such that $o\in \calM^{\lambda}$. Clearly,  $\sup_{t\le T} \varphi(\calO_t,X_t) \in L^1(\Q)$ since  $\varphi(\calO_t,X_t) \le \lVert \calO_T \rVert_{\infty} =  \sup_{x\in \R} L_T^{x} $ which is bounded in $L^p(\Q)$ for all $p\ge 1$ \cite{BarlowYor}. Hence the above OS  problem is well-defined. 
Within the family of metrics $\{\frakm_p :  p \in [1,\infty]\}$ introduced in \cref{sec:metrics},  note that  $\varphi$ is continuous \textit{only} with respect to $\frakm_{\infty}$. 

It is intuitive that \eqref{eq:OSLT} has a positive \textit{early exercise premium} when $\calT \ne \{T\}$ (recall our assumption $T\in \calT$). See also \cref{fig:OSLT}. 
In other words, there exists  $\tau \in \vartheta(\calT)$ such that $v(\tau) $ strictly dominates 
$v(T)=\E^{\Q}[L_T^{X_T}]$. 
First, let us compute 
the latter
explicitly. 

\begin{proposition} \label{prop:EUPrice}
Let $X$ be Brownian motion and  $T>0$. Then  $\E^{\Q}[L_T^{X_T}] = \E^{\Q}[L_T^{0}]  = \sqrt{\frac{2}{\pi}T}$.
\end{proposition}

\begin{proof} Consider the time reversed process $\overleftarrow{X}_s = X_{T-s}-X_T$, $s\in [0,T]$ and denote its local time  at $0$ by $\overleftarrow{L^0}$. By the chronology invariance   of occupation functionals (\cref{thm:chronology}), then   $L_T^{X_T}  = \overleftarrow{L^0_T}$ pathwise. 
    Moreover, it is classical that  $\overleftarrow{X}$ is also Brownian motion with respect to its own filtration, hence $\E^{\Q}[L_T^{X_T}] = \E^{\Q}[\overleftarrow{L^0_T}] = \E^{\Q}[L_T^{0}]$. The last equality in the statement  follows from  L\'evy's classical result that $L_T^{0}$ and $|X_T|$ have the  same law. 
\end{proof}
Next, we  verify the existence of an  early exercise premium  in the minimal case   $\calT = \{t,T\}$. 
Invoking the dynamic programming principle \eqref{eq:DPP} with $\varsigma = T$, then $Y_t :=   v(\calO_{t},X_t) = L_t^{X_t} \vee C_t$ with the continuation value $C_t =\E^{\Q}[L_T^{X_T} | \calF_t]$ derived in \cref{ex:heatspotLT}, namely 
\begin{equation}\label{eq:contVal}
    C_t = \int_{\R} L_t^y\  \phi_{T-t}(y-X_t) dy + \E^{\Q}[L_{T-t}^0], \quad \phi_s(x) = \frac{e^{-\frac{x^2}{2s}}}{\sqrt{2\pi s}}.
\end{equation}
We can finally approximate the initial value $Y_0 = \E^{\Q}[Y_t]$ using Monte Carlo simulations. For the numerical computation of the local times in \eqref{eq:contVal}, see \cref{rem:locTimeNum}. 
The optimal strategy is classically to stop at $t$ if the intrinsic value $L_t^{X_t}$ is greater than or equal to $C_t$; see \cref{fig:stopcont}. As the heat kernel $\phi_{T-t}$ concentrates around zero (especially when $T-t$ is small), we gather from  \eqref{eq:contVal} that the field $(L_t^x)_{x\in \R}$ is relevant for the stopping decision \textit{only} in a neighborhood of the spot.  This observation is crucial in \cref{sec:American} when truncating the occupation measure.  

\begin{figure}[t]
\centering
\caption{Spot local time OS problem with $\calT = \{t,T\}$.   Stopping region (red) and continuation region (green) at  $t$ depending whether the intrinsic value $L_t^x$ (if $X_t = x$) exceeds the continuation value  given in \eqref{eq:contVal}.   }
 \includegraphics[height=2.2in,width=2.9in]{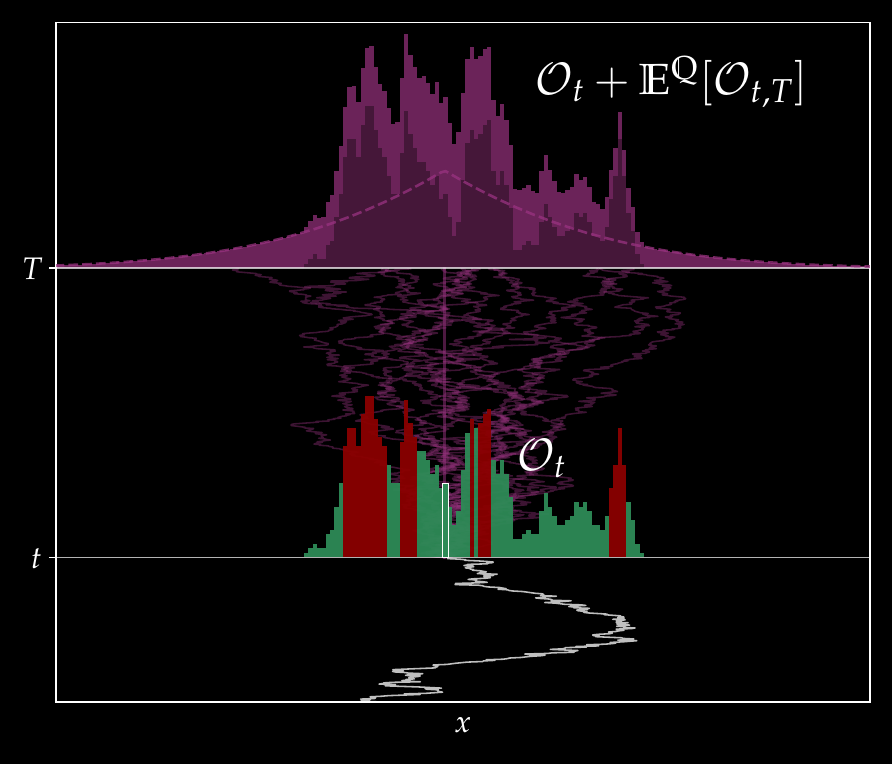}
 \label{fig:stopcont}
 \end{figure}

\cref{tab:earlyPremiumBermud2} summarizes the results in the case $T = 2t = 1$.   We simulate $2^{14} =16,384$ Monte Carlo paths on  a regular time grid with $N=400$ steps and set  $\varepsilon  = \sqrt{T/N} = 0.05$; see \cref{rem:locTimeNum}.  The early exercise premium is indeed significantly positive. 

\begin{table}[H]
    \centering
        \caption{Existence of an early exercise premium. Value of the OS problem \eqref{eq:OSLT} with $\calT = \{0.5,1\}$ and 
        $\calT = \{1\}$.    $N=400$,  $\varepsilon  = 0.05$, and $2^{14}$ Monte Carlo (MC) simulations. }
    \begin{tabular}{lcc}
      Exercise dates  &  Value & MC Error\\ \hline\hline
    $\calT = \{0.5,1\}$  & 0.8455 & 0.0050\\
     $\calT = \{1\}$  & 0.7979 & -   
    \end{tabular}
\label{tab:earlyPremiumBermud2}
\end{table}

\begin{remark}\label{rem:locTimeNum}
    \textnormal{(Numerical approximation of Brownian local time)}  Define 
\begin{equation}\label{eq:varphiEps}
     \varphi^{\varepsilon}(\calO_t,x) := \frac{\calO_t(B_{\varepsilon}(x))}{2\varepsilon}  = \frac{1}{2\varepsilon}\int_0^t \mathds{1}_{\{|X_s-x| \le \varepsilon\}}ds.
 \end{equation}
 In light of \eqref{eq:LTdef}, $\varphi^{\varepsilon}(\calO_t,x)$  converges to $ L_{t}^{x}$ pathwise. In the present context of optimal stopping, we shall also see  in \cref{sec:shrinkCorridor} that  
 estimating local time with
 narrow corridors  is legitimate in the sense that the value of the OS problem  \eqref{eq:OS} with $\varphi = \varphi^{\varepsilon}$  converges to the value of \eqref{eq:OSLT}. 
 If $X$ is a simulated Brownian path on a regular time grid $t_n = n\delta t$, $\delta t = \frac{T}{N}$,  we then discretize the  integral in \eqref{eq:varphiEps} to obtain
    $L_t^x \approx \frac{1}{2\varepsilon}\sum_{n=1}^N \mathds{1}_{\{|X_{t_n}-x| \le \varepsilon\}}\delta t. $
This estimate  induces   constraints on  $\varepsilon$ once $\delta t$ is fixed. 
  Intuitively, $\varepsilon \gg \delta t $,  and from our experiments, adequate choices are  
 $\varepsilon = c\sqrt{\delta t}$, e.g. $c = 1$, matching the scale of Brownian motion. 
\end{remark}

\subsection{Numerical Resolution in the case $\calT = [0,T]$} \label{sec:American}
We now turn to the case $\calT = [0,T]$. 
As the backward induction  in the previous section does not apply when $|\calT| > 2$ due to path-dependence, alternative methods  must be considered.    First, we outline a simple stopping rule giving a lower bound on the value of \eqref{eq:OSLT}. 

For fixed $\iota\in [0,T]$, interpreted as an inspection date, define\footnote{It is easy to show that the maximum of $x\mapsto L_\iota^x$ is attained and $\Q-$almost surely unique.}     $$\tau_\iota= \inf \{ t\in [\iota,T] \ : \ X_t = X^*_\iota\} \wedge T, \quad X^*_\iota = \text{argmax}_{x}L_\iota^x. $$
In other words, it is the hitting time after $\iota$ of the level which maximizes local time on $[0,\iota]$. 
If $\sigma^* = \sup \{s\in [0,T]  :  X_s = X^*_\iota \} \wedge 0$ denotes the last visit of the optimal level $X_\iota^*$,  then  on $\{\sigma^* \ge \iota\}$, the reward $L_{\tau_\iota}^{X_\iota^*}$ is precisely $L_{\iota}^{X_\iota^*} = \sup_x L_{\iota}^{x}$. 
This gives the decomposition, 
$$ \E^{\Q}[L_{\tau_\iota}^{X_{\tau_\iota}}] =   p_\iota\E^{\Q}[\sup_x L_{\iota}^{x}\ | \ \sigma^* \ge  \iota ] 
 + (1-p_{\iota})  \E^{\Q}[L_{T}^{X_{T}} \ | \ \sigma^* < \iota], \quad p_\iota = \Q(\sigma^* \ge \iota).$$
 There is a clear tradeoff between the reward $\sup_x L_{\iota}^{x}$ (increasing in $\iota$) and the probability $p_\iota$ to hit the  level $X^*_\iota$ again  (decreasing in $\iota$).
\cref{tab:earlyPremium2} summarizes the numerical results for values of  $\iota$. We use $T=1$ and  $2^{14} = 16,384$ simulated Brownian trajectories on a regular time grid with $400$  steps. To compute $X_{\iota}^*$, we use a regular space grid  of $[-2,2]$ with  $200$ subintervals. Despite the simplicity of the present stopping rule,  a  significant premium is observed. 
Note that the inspection date $\iota =0.7$ yields nearly a $40\%$ improvement over the case $\calT = \{T\}$  seen in \cref{tab:earlyPremiumBermud2}.  

\begin{table}[h]
    \centering
        \caption{Value of the stopping rule $\tau_{\iota} = \inf \{ t\in [\iota,T] \ : \ X_t = X^*_\iota\}$, $X^*_\iota = \text{argmax}_{x}L_\iota^x$ for varying inspection date $\iota$ and  parameters  $T=1$,  $N=400$, $\varepsilon =0.05$. }
    \begin{tabular}{ccc}
      Inspection date  &  Value & Monte Carlo Error\\ \hline\hline
       $\iota= 0.5$   & 1.0404 & 0.0035 \\ 
       $\iota= 0.6$   & 1.0897 & 0.0040\\ 
       $\iota= 0.7$   & 1.1116 & 0.0046\\ 
       
         $\iota= 0.8$   &  1.0892 & 0.0052\\ 
  $\iota= 0.9$   & 1.0182 & 0.0056\\    
    \end{tabular}

    \label{tab:earlyPremium2}
\end{table}


Next, we implement an approach à la Longstaff-Schwartz  where we refer to the original paper \cite{LS} and \cref{app:LSMC} for further details. Given the time grid $t_n = n\delta t, \ \delta t = \frac{T}{N}$, a  crucial step in \cite{LS} is the computation of the continuation value $C_{t_{n}} = \E^{\Q}[Y_{t_{n+1}} |\calF_{t_n}]$ where $Y = v(\calO,X)$ denotes the  Snell envelope of $L^{X}$. 
 Although we can write $C_{t_{n}} = c(\calO_{t_n},X_{t_n})$, $c:\calD \to \R$, using the Markov property of  $(\calO,X)$ (\cref{prop:Markov}), it is naturally challenging to estimate $c$  as the occupation measure is infinite-dimensional. Therefore, we carry out a space discretization to summarize $\calO$ with finitely many features. 

 
Given 
$\nu = \frac{1}{6}\big(\delta_{-\sqrt{3}} + 4\delta_0 + \delta_{\sqrt{3}}\big)$ matching the first five moments of $\mathcal{N}(0,1)$, introduce the random walk 
$X_{t_{n+1}} = X_{t_n} + Z_{n+1} \sqrt{\delta t}, \ Z_{n+1} \sim \nu,$  
taking values on a  trinomial tree with nodes $x_{m} = m\sqrt{3\delta t }$  
and   $|m| \le n$ at time $t_n$. 
In this model, the occupation measure  becomes finite-dimensional, namely 
\begin{equation}\label{eq:discreteOcc}
    \calO_{t_n} = (L_{t_n}^{x_{-n}}, \ldots, L_{t_n}^{x_{n}}), \quad L_{t_n}^{x_{m}} = \frac{\calO_{t_n}(B_{\varepsilon}(x_m))}{2\varepsilon} := \frac{1}{2\varepsilon}\sum_{i=1}^n  \mathds{1}_{\{X_{t_i} \ = \   x_m\}} \delta t.
\end{equation}
where $\varepsilon = \frac{x_{m+1}-x_m}{2} = \frac{\sqrt{3\delta t }}{2}$ in line with \cref{rem:locTimeNum}. 
The above expression for $L_{t_n}^{x_{m}}$
also reveals our  motivation behind the chosen  trinomial structure:  the local time process 
may accumulate at every step instead of every other step  in a (recombining) binomial tree.  

Notice that  the dimension of $\calO_{t_n}$ in \eqref{eq:discreteOcc} still expands with time, which becomes problematic for large $N$. Therefore,  we  introduce the truncation $\calO_{t_n}^{\bar{m}} = (L_{t_n}^{x_{m-\bar{m}}},\ldots, L_{t_n}^{x_{m+\bar{m}}} )$, $\bar{m} \in \N$, around the spot 
    $x_m = X_{t_n}$. The resulting pair $(\calO_{t_n}^{\bar{m}},X_{t_n})$ is of dimension $2\bar{m}+2$, hence a least square approach can be implemented for small values of $\bar{m}$ (say, less than $5$). 
Note that this is tailored to the payoff $L^X$: similar to \eqref{eq:contVal},   the continuation value primarily depends on the local time close to the spot. 
Evidently, other rewards may lead to other  truncations. 

We then follow the  least square Monte Carlo method  \cite{LS}, 
expounded in \cref{alg:LSMC} for completeness. 
We perform an  offline and online phase as  in \cite[
Remark 6.8]{GuyonHL} consisting of  $2^{11}$ and $2^{14}$ simulations,  respectively. See also \cref{app:LSMC} for other implementation details. The results are given in \cref{tab:LSMC}. Note that adding the local time of adjacent nodes ($\bar{m} = 1$) is indeed benefical, while  the value improves slowly beyond that. Comparing with \cref{tab:earlyPremium2}, we also see that the least square Monte Carlo approach indeed outperforms the more  naive inspection strategy. 
\begin{table}[h]
    \centering
        \caption{Value of OS problem \eqref{eq:OSLT} with  $\calT = \{n/N : n=1,\ldots,N\}$, $N = 400$, using Least square Monte Carlo for varying truncation level $\bar{m}$.  }
    \begin{tabular}{cccc}
      Truncation level $(\bar{m})$ &  Value & MC Error & Run time \\ \hline\hline
       $0$   & 1.1916 & 0.0044 & 182 sec. \\
        $1$   & 1.2180 & 0.0031 & 187 sec. \\ 
      $2$   & 1.2252 & 0.0030 & 194 sec. \\
       $3$   &  1.2277 & 0.0030 & 197 sec. \\
       $5$   & 1.2296 & 0.0030 & 210 sec. \\
      \end{tabular}
    \label{tab:LSMC}
\end{table}

 \subsection{Shrinking the Corridor}\label{sec:shrinkCorridor} 
Let  $\varphi^{\varepsilon}(\mo,x) = \frac{\mo(B_{\varepsilon}(x))}{2\varepsilon}$, 
$\varepsilon >0$ as   in \eqref{eq:varphiEps}. Hence $\varphi^{\varepsilon}(X_t,\calO_t)$ is a pathwise approximation of the spot local time $L_t^{X_t}$.  
We aim to show that the value $v_{*}^{\varepsilon}$ 
of the optimal stopping problem 
\begin{equation}\label{eq:OSLTEPS}
    \text{maximize } \;  v^{\varepsilon}(\tau) := \E^{\Q}[\varphi^{\varepsilon}(\calO_{\tau},X_{\tau})] \;\; \text{ over } \;\tau \in \vartheta(\calT),
\end{equation}
converges to the value $v_{*}$ of \eqref{eq:OSLT} in the extreme cases $\calT = \{T\}$ 
and $\calT = [0,T]$. 
We start with the former, where   $v_{*}^{\varepsilon}= v^{\varepsilon}(T)$ and $v_{*}= v(T)$.  

\begin{figure}
    \centering
       \caption{Verification of \cref{prop:vEpsEU} and \cref{cor:conv}  with $T=1$.  Estimate of   $\varepsilon \mapsto v_*^{\varepsilon}$  using the inspection strategy (\cref{sec:American}), $N=400$ time intervals and $2^{14}$  simulations. 
}\includegraphics[height = 2.1in,width = 2.85in]{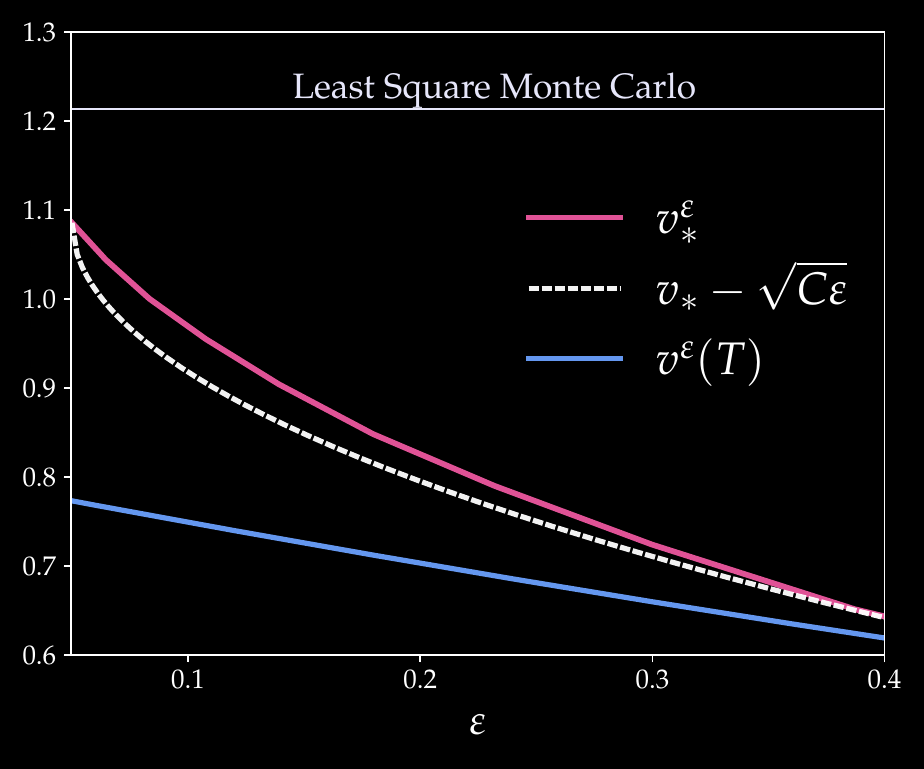}
    \label{fig:valueEps}
\end{figure}

\begin{proposition}\label{prop:vEpsEU}
    We have 
    $v^{\varepsilon}(T) \uparrow v(T)$ as  $\varepsilon \downarrow 0$. 
Precisely, $v^{\varepsilon}(T) = v(T) - \frac{\varepsilon}{2} + \frac{\varepsilon^2}{\sqrt{18\pi T}} (1 +  r_\varepsilon)$ with  $\lim_{\varepsilon\downarrow 0} r_\varepsilon = 0$.
\end{proposition}

\begin{proof}
    See \cref{app:vEpsEU}. 
\end{proof}

We gather from \cref{prop:vEpsEU} that the function $\varepsilon \mapsto v^{\varepsilon}(T)$ is essentially linear for small values of $\varepsilon$. This will be confirmed in \cref{fig:valueEps}. 
In the next results, we prove  that the convergence rate of  $v^{\varepsilon}(\tau)$ for general stopping times and  
$v_{*}^{\varepsilon}$
when $\calT = [0,T]$ is at least  $1/2$. In fact, this lower bound  seems  tight according to   \cref{fig:valueEps}.   

\begin{proposition}\label{prop:conv}
There exists a constant $C\in (0,\infty)$ such that for all $0< \varepsilon \le 1$, 
\begin{equation}
    \left\lVert\varphi^{\varepsilon}(\calO_{\tau},X_{\tau}) - L_{\tau}^{X_{\tau}} \right\rVert_{L^2(\Q)} \le \sqrt{C\varepsilon}\qquad \forall \tau\in \vartheta(\calT).
\end{equation}
In particular, $ |v^{\varepsilon}(\tau) -  v(\tau)| \le \sqrt{C\varepsilon}$ for all  $\tau\in \vartheta(\calT)$. 
\end{proposition}
\begin{proof}
    See \cref{app:conv}. 
\end{proof}

\begin{corollary}\label{cor:conv}
   Let $v_*^{\varepsilon}$, $v_{*}$ be respectively the value of  the OS problem \eqref{eq:OSLTEPS} and \eqref{eq:OSLT} with $\calT = [0,T]$. Then 
   $ |v^{\varepsilon}_{*} -  v_*| \le \sqrt{ C \varepsilon}$, with   $C$ as in \cref{prop:conv}. 
   In particular, 
    $\lim_{\varepsilon \to 0} v_*^{\varepsilon} = v_*$. 
\end{corollary}

\begin{proof} This is immediate: 
    let $\gamma>0$ be arbitrary and $\tau_{\gamma}$ such that $v(\tau_{\gamma}) \ge v_{*}-\gamma$. Then 
    \begin{align*}
    v_* -\gamma \le  v(\tau_{\gamma})  \le  v^{\varepsilon}(\tau_{\gamma}) + \sqrt{C\varepsilon} \le v_*^{\varepsilon} + \sqrt{C\varepsilon},
\end{align*}
 so $v_* - v_*^{\varepsilon} \le  \sqrt{C\varepsilon}$ as $\gamma \to 0$. The inequality  $v_*^{\varepsilon} - v_*  \le  \sqrt{C\varepsilon}$ is shown mutatis mutandis. 
\end{proof}

\cref{fig:valueEps} displays 
the optimal value of \eqref{eq:OSLT} as function of $\varepsilon$ in the case $\calT = \{1\}$ 
 and $\calT =[0,1]$. 
 The highest value from \cref{tab:LSMC} using the least square Monte Carlo approach is also reported. We use the inspection strategy laid out in \cref{sec:American}  to estimate  the value in the case  $\calT =[0,1]$. Note that the early exercise premium (spread between the magenta and blue lines) is significant and increases as the corridor width shrinks to zero. 
 This is because 
peaks in the local time field $x\mapsto L_{t}^{x}$ 
can be captured more accurately as $\varepsilon$ gets smaller. 

%% file: Financial_Examples/main.tex
\section{Unified Markovian  Lift of Financial Instruments}\label{sec:finApp}
\input{Financial_Examples/UnifiedLift}


\input{Financial_Examples/FOS}

%% file: Financial_Examples/UnifiedLift.tex
This section  demonstrate the 
broad range of financial applications  offered by occupied processes. 
First, we show how occupation flows provide a unified way to express payoffs of financial derivatives  while preserving the Markov property.  
Precisely, if $X$ is the asset price process and $\xi:\Omega \to \R$ a European style, path-dependent claim, a common approach in derivatives pricing  consists of choosing a path-dependent feature $F_t(\omega) = \textnormal{f}(t,\omega) \in \frakF$ where $\textnormal{f}: [0,T]\times \Omega \to \R$ is a non-anticipative functional (see \cref{sec:FITO}), such that 
\begin{equation}
    \xi(\omega)  = \varphi(F_T,X_T)(\omega),  \quad \varphi: \frakF \times \R \to \R, \quad \ \omega \in \Omega, 
\end{equation}
and $ (F,X)$ is a Markov process. 
The feature space $\frakF$ is chosen to be "nicer" than $\Omega$, e.g., a finite-dimensional Euclidean space $\frakF = \R^d$  or $\frakF = \calM$ in the case of occupation flows. 
Due to the Markov property, the derivatives price will be function of these variables as well, leading to  tractable numerical  methods for  the  pricing PDE in the projected space.  
Setting $F = \tilde{\calO}$ or $F = \calO$  leads to occupied pricing PDEs, allowing  the model to depend on occupation times as well, mainly  through volatility. 
In the case $F = \calO$,  
an extension of forward variance models is explored in \cref{sec:FwdOccupation}, while the use of calendar time  occupation flows shall be presented in a companion paper \cite{TissotLOV}. 

It turns out that relevant features in finance are very often expressible in terms of the  occupation flow, 
as shown next.

\subsection{Exotic Options and Structured Products}\label{sec:occPayoffs}

We here gather financial instruments  
whose payoff  is function of 
the calendar time occupied process, that is 
$\xi = \varphi(\tilde{\calO}_T,X_T)$.  
If $\Q$ is a risk-neutral measure and assuming zero interest rates,  the time $t$   price of a European-style exotic option with payoff $\varphi$   is  $v(\tilde{\calO_t},X_t) = \E^{\Q}[\varphi(\tilde{\calO}_T,X_T)  |  \calF_t ]$. If the asset follows the dynamics of an occupied SDE as in \eqref{eq:OSDEStd} (see also the companion paper \cite{TissotLOV}), we   can then invoke Feynman-Kac's formula in \cref{thm:FKDirichletStd} to deduce the corresponding pricing PDE. The latter is then integrated using a numerical solver, the \emph{same} for all occupation-dependent derivatives, from which the initial price and  sensitivities (Greeks) can be extracted. 

This framework is summarized in \cref{fig:diagramExotics}. We stress that  the occupation measure is discretized  in practice, e.g., using occupation times as in \cref{sec:American}, leading to a finite-dimensional Markovian lift. For the numerical solver, a promising  direction is to adapt   deep BSDE methods \cite{HanJentzenE1,hure2020deep} to our context, as explained in the forthcoming work \cite{HuangTissotZhang}.

\begin{figure}[H]
    \centering
    \caption{Unified Markovian lift for exotic options and structured products.}
    \label{fig:diagramExotics}

\begin{tikzpicture}

  
        \node (A) at (-0.25,4.5) {\underline{Instrument/Feature}}; 
    \node (E) at (4,4.5)  {\underline{Occupied}};  
     \node (O) at (8,4.5)  {\underline{Model/Method}}; 
    \node (A1) at (-0.25,3) [draw,text width=4.3cm, align=center, fill=red!5] {\textbf{Asian} \\ $\int_0^{T}X_tdt$}; 
    \node (A2) at (-0.25,1.25) [draw,text width=4.3cm, align=center, fill=red!5] {\textbf{Barrier/Lookback} \\ 
    $[\min\limits_{t\le T}X_t, \max\limits_{t\le T}X_t] $}; 
    \node (A3) at (-0.25,-0.5) [draw,text width=4.3cm, align=center, fill=red!5] {{\small \textbf{Range Accrual/Parisian}} \\ 
    $\int_0^T  \mathds{1}_{A}(X_t)dt$}; 

    \node (E1) at (4,3) [draw,text width=1.9cm, align=center, fill=cyan!50!blue!10] {$\int_{\R} x  \tilde{\calO}_T(dx)$}; 
    \node (E2) at (4,1.25) [draw,text width=1.9cm, align=center, fill=cyan!50!blue!10] {$\text{supp} (\tilde{\calO}_T)$}; 
    \node (E3) at (4,-0.5) [draw,text width=1.9cm, align=center, fill=cyan!50!blue!10] {$\tilde{\calO}_T(A)$}; 

    \node (F) at (8,1.25) [draw,text width=2cm, align=center, fill=green!40!blue!15] {\textbf{BSDE} \\ \textbf{Solver}}; 

     \node (G) at (11.5,1.25) [draw,text width=1.5cm, align=center, fill=green!70!blue!10] {\textbf{Price} \\ \& \textbf{Greeks}}; 

    \node (H) at (8, 3) [draw,text width=3.75cm, align=center, fill=cyan!30!blue!10] {\textbf{Occupied Volatility} \\ $\frac{dX_t}{X_t} = \sigma(\tilde{\calO}_t,X_t) dW_t$};

       \draw[myarr] (A1) -- (E1) node[midway, above] {};
       \draw[myarr] (A2) -- (E2) node[midway, above] {};
       \draw[myarr] (A3) -- (E3) node[midway, left] {};
       \draw[myarr] (E1) -- (F) node[midway, above] {};
       \draw[myarr] (E2) -- (F) node[midway, above] {};
       \draw[myarr] (E3) -- (F) node[midway, left] {};
       \draw[myarr] (F) -- (G) node[midway, above] {};
       \draw[myarr] (H) -- (F) node[midway, above] {};

\end{tikzpicture}
\end{figure}

\begin{example} \label{ex:Asian}  
   \textit{Asian options} are contingent upon the time average of the asset. One typically choose the running integral  $F_t =  \int_0^{t} X_s ds$ as feature over the running average due to the simpler dynamics of the former. 
   Clearly, the chronology of the path is here irrelevant, so we conclude from \cref{thm:chronology}  that $F_t$ is function of $\tilde{\calO}_t$.   
   This is  confirmed by 
    the occupation time formula, namely  $\int_0^{t} X_s ds = \int_{\R} x d\tilde{\calO}_{t}$. Hence, Asian payoffs can be rewritten in terms of $\tilde{\calO}_T$ and the terminal asset price when need be. 
   For instance,  the payoff of an floating  strike Asian call option reads
    $$\xi = \left( X_{T} - \frac{1}{{T}} \int_0^{T} X_t dt\right)^+ = \left( X_{T} - \frac{1}{{T}}\int_{\R} x d\tilde{\calO}_{T} \right)^+= \varphi(\tilde{\calO}_{T},X_{T}). $$
  The scaling factor $\frac{1}{{T}}$ may also be written as $\tilde{\calO}_{T}(\R)^{-1}$ in the occupation payoff, though the maturity is known from the contract's term sheet. 
\end{example}%

\begin{example}\label{ex:Lkbk}
   The payoff of  \textit{barrier and lookback options} involves the range  of the underlying asset. The relevant feature is thus given by $$F_t = [\min\limits_{s\le t}X_s, \max\limits_{s\le t}X_s],$$ 
   which coincides with the support of $\tilde{\calO}_t$ as seen in \cref{ex:runMax}.  
   Barrier options use the range to monitor whether the asset has breached a prescribed level$-$the barrier, which acts as a knock-in (creation) or or knock-out (termination) trigger. On the other hand, lookback derivatives mimic vanilla call/put payoffs 
   by replacing the terminal value by   the maximum/minimum price, or use the latter as a floating strike. 
E.g., the payoff of a floating lookback call options admits the form  
    $$\left(X_T - \min_{t\le {T}} X_t \right)^+ = (X_T - \inf \textnormal{supp}(\tilde{\calO}_{T}))^+= \varphi(\tilde{\calO}_{T},X_T).$$
\end{example}

We end this section with a 
popular class of contracts in the quantitative finance literature known as  \textit{occupation time derivatives}\footnote{not to be confused with the occupation  derivative in \cref{def:occDerivative}!} \cite{ChesneyJeanblancYor,Hugonnier}  \cite[Chapter 7]{Detemple}. 
The generic form for the payoff of such options is 
\begin{equation}\label{eq:occTimePayoff}
    \varphi(\tilde{\calO}_T,X_T) = \phi(\tilde{\calO}_T(A),X_T), \quad \phi:\R^2\to \R, \quad A\in \calB(\R).
\end{equation} 
\begin{remark}
    While it is technically sufficient to use  the two-dimensional process $(\tilde{\calO}(A),X)$ for the payoff in \eqref{eq:occTimePayoff}, this minimal approach becomes impractical  
when a book of occupation time derivatives contingent on several Borel sets is given, calling for the more general occupation lift $X_T \to (\tilde{\calO}_T,X_T)$. 
\end{remark}

\begin{example} 
\label{ex:rangeAccrual} \emph{Range accruals} are popular hybrid products in fixed income and foreign exchange markets. The holder receives a floating interest payment based on the proportion of time the reference rate (e.g. SOFR) has stayed within   a given range, say $[x_1,x_2]$. 
In other words, the paid amount is of the form $\frac{C}{T}\int_0^T \mathds{1}_{[x_1,x_2]}(X_t)dt$, so the relevant feature 
is precisely the occupation time $F_T = \tilde{\calO}_T([x_1,x_2])$. 

Note that range accruals are linear  in $\tilde{\calO}_T$, so their price is  theoretically known from vanilla options. Indeed,  recall from  \citet{BreedenLitzenberger} that 
$\Q(X_T\le x) = \partial_K P(x,T)$, with the undiscounted vanilla put price  $P(K,T) = \E^{\Q}[(K-X_T)^+]$. Consequently, the expected calendar occupation time of $X$ in $[x_1,x_2]$ becomes  
\begin{equation}\label{eq:rangeIntegral}
    \E^{\Q}[\tilde{\calO}_T([x_1,x_2])] = \int_0^T \Q(X_t \in [x_1,x_2])dt =    \int_0^T \big( \partial_K P(x_2,t) - \partial_K P(x_1,t)\big)dt, 
\end{equation} 
assuming that $\Q\circ X_T^{-1}$ has no atoms. 
However, only finitely many strikes and maturities are available in practice. Even when the range accrual is monitored daily, so the right side of \eqref{eq:rangeIntegral} becomes a sum, options on the reference rate rarely expires on every business day. Hence, the valuation of range accruals  would still benefit from the unified  framework shown in  \cref{fig:diagramExotics}. 
\end{example}


\begin{example} 
\textit{Parisian options} \cite{ChesneyJeanblancYor,Hugonnier} were introduced to relax the knock-in/out component of  barrier options seen in \cref{ex:Lkbk}. 
 Precisely, the option is activated or terminated only once the asset has performed an excursion above/below the barrier of length at least $\delta>0$, referred to as the \textit{window}. 
 A variant, called cumulative Parisian or \textit{Parasian} option \cite{haberWilmott},  requires that the cumulative time above/below the barrier exceeds the window. As the payoff of excursion-based Parisian options depends on the chronology of the path, the occupied lift can only be applied to the cumulative counterparts  in view of \cref{thm:chronology}. 
For example, consider   an asset-or-nothing, up-and-out cumulative Parisian call option with strike $K$, barrier  $B>K$, and window $\delta >0$, namely 
\begin{equation*}
   \xi = \begin{cases}
       X_T, & \text{ if } \int_0^T\mathds{1}_{\{X_t \ge B\}} dt <\delta,  \\
       0, & \text{otherwise.}
   \end{cases}
\end{equation*}
Letting $A = [B,\infty)$, we can rewrite the above  payoff  as 
$ \xi  =  X_T \mathds{1}_{\{\tilde{\calO}_T(A) < \delta  \}},$ which is indeed of the form \eqref{eq:occTimePayoff}.  
The relaxed knock-in/out feature is particularly appealing to the buyer of knock-out Parisian options as  it (partially) prevents  price manipulation from the seller  to push the underlying towards the barrier,  making the contracts  worthless. From the seller's perspective,  Parisian options are easier to hedge as the soft barrier effectively smoothen the Greeks. 
Other instruments exhibit a similar relaxation feature. 
One example is the soft call provision of convertible bonds, where  the issuer  discontinues (calls) the contract once the underlying has spent a prescribed amount of time above a  threshold. 
\end{example}

\subsubsection{Deep BSDE Solver}\label{sec:PDE}
%

Back to the unified approach  in \cref{fig:diagramExotics}, the occupation payoffs and calibrated model are then given to a numerical solver. We here present an adaptation of deep BSDE methods \cite{HanJentzenE1,HanJentzenE2,hure2020deep}, which will be further discussed and tested in an accompanying paper \cite{HuangTissotZhang}. 
Let $W$ be Brownian motion and $X$ denotes the asset price process  adapted to  $W$. Given a risk-neutral measure $\Q$ and contingent claim   $\xi  \in L^1(\Q)$, then the value process $Y_t = \E^{\Q}_t[\xi]$ solves the backward SDE (BSDE),  
\begin{equation}\label{eq:BSDE1}
    Y_t= \xi -  \int_t^T Z_sdW_s,
\end{equation}
for some adapted process $Z$. Assuming  a Markov structure, the  deep BSDE method  consists of parameterizing the $Z$ process by a feedforward neural network,  taking the state variable as input.  Among other applications, the deep BSDE method has shown impressive performance in derivatives pricing, especially when the state process is high-dimensional \cite{HanJentzenE1,HanJentzenE2,hure2020deep}. The increasing dimensionality may come from  numerous underlying securities in multi-asset options or, perhaps of greater relevance in practice, from features to express exotic payoffs and path-dependent models as shown next. 

Suppose that $\xi = \varphi(\tilde{\calO}_T,X_T)$ for some occupation payoff $\varphi$ and that $X$ has dynamics $dX_t =  X_t\sigma(\tilde{\calO}_t,X_t)dW_t$; see \cite{TissotLOV}. 
Since $(\tilde{\calO},X)$ is Markovian, then  $Y_t  = v(\calO_t,X_t)$ for some price functional $v:\calD\to \R$.   
Assuming that  $v$ is smooth, we conclude from  It\^o's formula (\cref{thm:ItoOX}) that $Z_t = \partial_x v(\calO_t,X_t)  \sigma(\calO_t,X_t)$. One then parameterizes the Delta $\Delta = \partial_x v$, or $Z$, by a feedforward neural network $\Phi:\calD_T \times \Theta \to \R$ for some  parameter set $\Theta \subseteq \R^p$, $p\in \N$. In the present  context of pricing and hedging, it is more natural  to approximate the Delta process, namely
\begin{equation}\label{eq:ZTheta}
    \partial_x v(\tilde{\calO}_t,X_t)   \approx \Phi(\tilde{\calO}_t,X_t; \theta) = :\Delta^{\theta}(\tilde{\calO}_t,X_t).
\end{equation}
Moreover, the unknown initial value $Y_0$ is replaced by an extra parameter $\pi\in \R$, the  price of the claim, leading to the parameterized value process $ Y_t^{\pi, \theta} = \pi + \int_0^t \Delta^{\theta}(\tilde{\calO}_s,X_s)dX_s$. 
In light of \eqref{eq:BSDE1}, the parameters $\pi,\theta$ are trained so as to satisfy $Y_T^{\pi, \theta} \approx \xi$, e.g., by solving the following quadratic hedging problem using stochastic gradient descent,
\begin{equation}\label{eq:loss}
    \inf_{\pi, \theta} \ \calL(\pi,\theta), \quad \calL(\pi, \theta) := \E^{\Q}[|\xi - Y_T^{\pi, \theta}|^2]. 
\end{equation}
Crucially, the control $\Delta^{\theta}(\tilde{\calO}_t,X_t)$ takes the state process as input, so  can be parametrized by a single neural network across time. This is in stark contrast with  path-dependent PDEs where the deep BSDE method is not available due to the lack of Markov structure. Instead, one may use  recurrent neural networks, e.g. \cite{SaporitoPPDE} extending the deep Galerkin method to PPDEs.  However, the training and tuning of recurrent neural networks is much more involved and suffers from an explosion in input dimension when the number of time steps becomes large. Hence,  it is less clear that such approach would be  adopted by practitioners.

\subsection{Variance Derivatives}\label{sec:varDerivatives}
Similarly, the occupation flow $\calO$ provides a unified lift for variance instruments as shown in the following examples and diagram below. The forward occupation model mentioned in \cref{fig:diagramVariance} will be introduced in \cref{sec:FwdOccupation}.   
 
\begin{figure}[H]
    \centering
    \caption{Unified Markovian lift for variance derivatives.}
    \label{fig:diagramVariance}
\begin{tikzpicture}
        \node (A) at (-0.25,4.5) {\underline{Instrument/Feature}}; 
    \node (E) at (4,4.5)  {\underline{Occupied}};  
     \node (O) at (8,4.5)  {\underline{Model/Method}}; 
    \node (A1) at (-0.25,2) [draw,text width=4.4cm, align=center, fill=red!5] {{\small \textbf{Variance/Timer}} \\ $\int_0^{T}V_tdt$}; 
    \node (A2) at (-0.25,0) [draw,text width=4.4cm, align=center, fill=red!5] {{\small \textbf{Weighted Variance}} \\ 
    $\int_0^T  w(X_t)V_tdt$}; 

    \node (E1) at (4,2) [draw,text width=2.5cm, align=center, fill=cyan!50!blue!10] {$\calO_T(\R)$}; 
    \node (E2) at (4,0) [draw,text width=2.5cm, align=center, fill=cyan!50!blue!10] {$\int_{\R} w(x)\calO_T(dx)$};

    \node (F) at (8,1.25) [draw,text width=1.25cm, align=center, fill=cyan!70!blue!15] {\textbf{Solver}}; 
     \node (G) at (11.5,1.25) [draw,text width=1.5cm, align=center, fill=green!70!blue!10] {\textbf{Price} \\ \& \textbf{Greeks}}; 

    \node (H) at (8, 3) [draw,text width=3.95cm, align=center, fill=cyan!50!blue!10] {\textbf{Forward Occupation} \\ $d\calO_t^T = H_t^T \cdot dW_t$};

       \draw[myarr] (A1) -- (E1) node[midway, above] {};
       \draw[myarr] (A2) -- (E2) node[midway, above] {};
       \draw[myarr] (E1) -- (F) node[midway, above] {};
       \draw[myarr] (E2) -- (F) node[midway, above] {};
       \draw[myarr] (F) -- (G) node[midway, above] {};
       \draw[myarr] (H) -- (F) node[midway, above] {};

\end{tikzpicture}
\end{figure}


\begin{example} \label{ex:CVS}
    \textit{Corridor Variance Swaps} (CVS) 
are over-the-counter instruments providing exposure to the volatility of an asset when the latter lies within a prescribed range. They belong to the larger family of weighted variance swaps; see \citet[Chapter 5]{Bergomi} and \citet{Lee}. 
For concreteness, if the log price of an asset has dynamics
$dX_t = - \frac{1}{2}V_tdt + \sqrt{V_t} dW_t$, then the terminal profit and loss (P\&L) for the buyer  is given by  
\begin{equation}\label{eq:PnL}
    \text{P\&L} =  \frac{1}{T}\int_0^T \mathds{1}_{[x_{1},   x_{2}]}(X_t)  V_t dt \; \ - \; \   K^2_{x_1,x_2} \ , \qquad x_{1} < x_{2}.
\end{equation} 
The levels $x_1,x_2$ defines the corridor during which the instanteneous variance $V_t$ is accumulated. Since $V_t dt = d\langle X \rangle_t$, the floating leg 
is proportional to the  occupation time $\calO_T([x_1,x_2])$.  
The \textit{fair strike} $K_{x_1,x_2}$ is set such that the swap  has zero initial cost. By the law of one price,  it is thus given by the (scaled) expected occupation time   $ K^2_{x_1,x_2} = \frac{1}{T}\E^{\Q}[\calO_T([x_1,x_2])]$. Hence, corridor variance can be viewed as the primitive securities associated to $\calO$ and will play a central role in the forward occupation models introduced in \cref{sec:FwdOccupation}. 


Corridor variance swaps offer several  advantages over a standard variance swap which  delivers the cumulative realized variance 
irrespective of the underlying price level. 
First, CVS's are less expensive
the variance accumulated outside the prescribed corridor is simply discarded. Second, the holder may have some views on the market and can speculate a range on the asset's future levels by entering a suitable CVS. See \cite{Lee,BNP} for further details. 
\end{example}

\begin{remark}
Towards a parallel with the stop local time problem
(\cref{sec:spotLT}), we may envision a \emph{floating American CVS}  where the holder 
chooses when to receive the realized variance within 
 a floating corridor around the spot price. 
That is,  $x_{1} = X_{\tau}-\varepsilon$, $x_{2} = X_{\tau}+\varepsilon$, in which case  we have\footnotemark[\value{footnote}] 
\footnotetext{Note that $\calO_T(\overline{B_{\varepsilon}(x)}) = \calO_T(B_{\varepsilon}(x))$ as $\calO_T \ll \lambda$ $\Q-a.s.$ }
\begin{equation}\label{eq:CVSStrikeAmFloat}
      K^2(\varepsilon) = \sup_{\tau\le T}  \ \E^{\Q}\left[\frac{1}{2\varepsilon}\int_0^\tau \mathds{1}_{B_{\varepsilon}(X_{\tau})}(X_t) V_t dt \right] = \sup_{\tau\le T}  \ \E^{\Q}\left[\varphi_{\varepsilon}(\calO_\tau,X_\tau) \right], 
\end{equation} 
with $\varphi_{\varepsilon}$ seen in \eqref{eq:varphiEps}.
As seen \cref{sec:shrinkCorridor}  that  $K^2(\varepsilon) = v_{\varepsilon}^*$ converges to the value of the spot local time problem \eqref{eq:OSLT}. 
While these contracts are not traded in financial markets,  their introduction could benefit investors seeking exposure to volatility around the asset price when it is exercised. 
\end{remark}
 

\begin{example}\label{ex:timer}
    
Introduced by Société Générale  \cite{SocGen}, \textit{timer options}   
differ from traditional derivatives by having a floating maturity  determined by the asset's realized volatility 
\cite{carole}  
\cite[Chapter 1]{GuyonHL}. Concretely, the  payoff is delivered once 
the total variance exceeds a prescribed budget. This characteristic presents benefits for both parties: 
the buyer can save on the  premium, typically positive, between  implied  and realized volatility, while the seller is less exposed, if exposed at all,  to model risk arising from volatility, called Vega risk. This is because these contracts they are less sensitive, if not insensitive to volatility as explained next. 

Let $S$ be an asset with   dynamics $dS_t = \sigma_t S_t dW_t$, with $\sigma_t \ge \underline{\sigma}$ for some constant $\underline{\sigma}>0$. Define also the log price  $X = \log S$,   quadratic variation  process 
 $Q_t = \langle X \rangle_t = \int_0^t \sigma_s^2ds$, 
 and stopping time $\tau_q = \inf\{t\ge 0 :  Q_t \ge q\}$. Fix a variance budget $\bar{q} >0$ and payoff $\varphi: \R_+\to \R$ such that  $w(q,s) := \E^{\Q}[\varphi(S_{\tau_{\bar{q}-q}}) \ | \ S_0 = s]$ belongs to $\calC^{1,2}(\R_+^2)$. Then, it is easily  seen from Itô's formula that $w$ 
 solves the PDE
  \begin{subnumcases}{}
\Big(\partial_q + \frac{s^2}{2}\partial_{ss}\Big)u(q,s) = 0,  & \; $q <  \bar{q}$, \label{eq:PDETimer}\\[0.2em] 
u(q,s) = \varphi(s) & \; $q = \bar{q}. $
\label{eq:PDETimerTerminal}
\end{subnumcases}
The volatility does not appear in \eqref{eq:PDETimer} since  it is absorbed by the quadratic variation process.   Consequently,  the value of the contract  
 is insensitive to volatility. 
However,  this conclusion does not hold anymore in the presence 
 of (nonzero) 
 interest rates or dividends; see \cite{carole}.   

Importantly, $q$  plays the role of time in  \eqref{eq:PDETimer}-\eqref{eq:PDETimerTerminal} as does the occupation flow, and one naturally expects a connection between  the above PDE and  the Dirichlet problem \eqref{eq:dirPDE}-\eqref{eq:dirTerminal}. Indeed, if $\calO$ is the occupation flow of   $X$, then $Q_t = \calO_t(\R)$, and set $v(\mo,x) := w(\mo(\R),e^x)$. Hence $v\in \calC^{1,2}(\calD)$ with $\partial_x v = s \partial_s w$, $\partial_{xx} v = s^2 \partial_{ss} w$,   and  $\partial_{\mo} v = \partial_q w$ from    \cref{ex:t}. 
If the volatility process satisfies $\sigma_t = \sigma(\calO_t,X_t)$ for some map $\sigma:\calD\to [\underline{\sigma},\infty)$, then $dX_t = -\frac{1}{2}\sigma(\calO_t,X_t)^2dt + \sigma(\calO_t,X_t)dW_t$,  and according to  \cref{thm:FKDirichlet},  $v$ solves   the occupied PDE 
  \begin{subnumcases}{}
\Big(\sigma^2\partial_{\mo}   -\frac{\sigma^2}{2}\partial_{x}  + \frac{\sigma^2}{2}\partial_{xx}\Big)u(\mo,x) = 0,  & \text{ on } \hspace{-0.8mm} $\calD_{\Theta},$ \label{eq:oPDETimer}\\[0.2em] 
u(\mo,x) = \varphi(e^x) & \text{ on } \hspace{-0.8mm} $\partial \! \calD_{\Theta},$
\label{eq:oPDETimerTerminal}
\end{subnumcases}
where $\Theta = \bar{q}$. 
Again, we can factor out $\sigma^2$  from \eqref{eq:oPDETimer},  
confirming  the irrelevance of the volatility  to price the option.  
\end{example}

%% file: Financial_Examples/FOS.tex
\subsection{Forward Occupation Models}\label{sec:FwdOccupation}

We conclude this section with a generalization of  
\textit{forward variance models} (\cite{Bergomi,GatheralKeller}) that uses the information content of vanilla options not only through the forward variance curve, but through all forward occupation times.\footnote{I would like to thank an anonymous referee for suggesting this promising direction. } 
Consider a standard Brownian motion $W = (W^1,\ldots, W^d)$ in $\R^d$, $d\ge 1$, and denotes its natural filtration by $(\calF_t)$. Write also $\E^{\Q}_t[\cdot ]= \E^{\Q}[ \cdot  |\calF_t]$. Assume that the log price  $X = \log S$ of an asset has dynamics 
$$dX_t = - \frac{1}{2}V_tdt + \sqrt{V_t}dW_t^1,$$
where  $V$ may depend on every component of $(W^1,\ldots, W^d)$.  Forward variance models \cite{Bergomi,GatheralKeller} rely on the following embedding $V_t \hookrightarrow \xi_t^{\cdot}$, with the  forward variance curve $\xi_t^T =  \E^{\Q}_t[V_T]$ (noting that $V_t= \xi_t^t$). 
Then $\xi^T$ is a $\Q-$martingale on $[0,T]$, and one typically specifies a dynamics for the forward variance curve of the form 
\begin{equation}\label{eq:fwdVarDynamics}
    d\xi_t^T = \eta(t,T,\xi_t^{\tau})\cdot dW_t = \sum_{i=1}^d \eta_i(t,T,\xi_t^{\tau}) dW_t^i,
\end{equation}
where $\tau \ge 0$ depends on $t,T$. In Bergomi models \cite{Bergomi}, $\tau = T$ with $\eta(t,T,\xi_t^{T})$ linear in $\xi_t^{T}$, while in  affine forward variance models \cite{GatheralKeller},  $\tau= t$, namely $\eta(t,T,\xi_t^{\tau}) = \eta(t,T,V_t)$. See \cref{ex:BergomiModels,ex:AffineFwdModels}. 
Given a variance swap that delivers the realized variance on $[0,T]$, then 
the value of its  (unannualized) floating leg  at time $t\in [0,T]$ can then be written  as 
\begin{equation}\label{eq:VSt}
    \text{VS}_t^T := \E^{\Q}_t[\int_0^T V_sds] = \underbrace{\int_0^t V_sds}_{\text{realized}} + \underbrace{\int_t^T \xi_t^sds}_{\text{implied}}. 
\end{equation}
Note that $\text{VS}^T$ is also a $\Q-$martingale, starting at the fair strike $\text{VS}^T_0$  and delivering the (unannualized) realized variance at $t = T$. 
Shifting gears, on can embed the forward variance curve into the \textit{instantaneous forward occupation measure},  
\begin{equation}\label{eq:instFwdOcc}
\xi_t^{T} \hookrightarrow   \mo_t^T(A) = \E^{\Q}_t[\mathds{1}_{A}(X_T)V_T],  \quad A\in \calB(\R), 
\end{equation}
where $\xi_t^T = \mo_t^T(\R)$. 
Also, $\mo_t^T$ can be seen as a surface $(K,T) \mapsto  \E^{\Q}_t[\delta_{K}(X_T)V_T]$ in strikes and maturities, an object initially studied by \citet{dupire2004unified}.  
Integrating  $\mo_t^{\cdot}$ for fixed $t$   gives the   forward occupation measure
\begin{equation}
    \calO_t^T(A) = \int_0^T \mo_t^{s}(A) ds = \E^{\Q}_t\Big[\int_0^T \mathds{1}_{A}(X_s)V_sds\Big] 
    = \E^{\Q}_t[\calO_T(A)], 
\end{equation}
which is equivalently described by the local time surface $(K,T) \mapsto \E^{\Q}_t[L_T^K]$; see \cref{fig:fwdOccSurface}. 
When $A = [x_1,x_2]$,  $\calO_T^T(A) = \calO_T(A)$ is  the (unannualized) floating leg of the associated corridor variance swap (\cref{ex:CVS}), while $\calO_0^T(A) = \E^{\Q}[\calO_T(A)]$ is its fair strike. 
Classically, the latter can be statically replicated using vanilla options through the formula
$$\calO_0^T(A) = 2\int_A \text{OTM}(K,T) \frac{dK}{K^2}, $$
where $\text{OTM}(K,T)$ is the  price  of the corresponding out-of-the-money  vanilla contract (put  if $K< S_0$ and call otherwise). As $\mo_0^T = \partial_T \calO_0^T$, the instantaneous forward occupation surface can be extracted from a weighted sum of calendar spreads 
$\mo_0^T(A) = 2\int_A \partial_T\text{OTM}(K,T) \frac{dK}{K^2}. $ An illustration is given in  \cref{fig:fwdOccSurface} for S\&P 500 Index (SPX)  options. In short,  forward variance models can be refined by adding all the tradeable forward occupation times $\calO_o^{\cdot}(A)$.

\begin{figure}[H]
\centering
\caption{Forward occupation surfaces of SPX Options as of 2025/02/27, 4pm EST. Left: $\mo^T_0$ (instantaneous).   Right:  $\calO^T_0 = \int_0^T \mo_0^{s} ds$ (cumulative). }

\begin{subfigure}[b]{0.49\textwidth}
     \centering
\includegraphics[height=2.1in,width=2.75in]{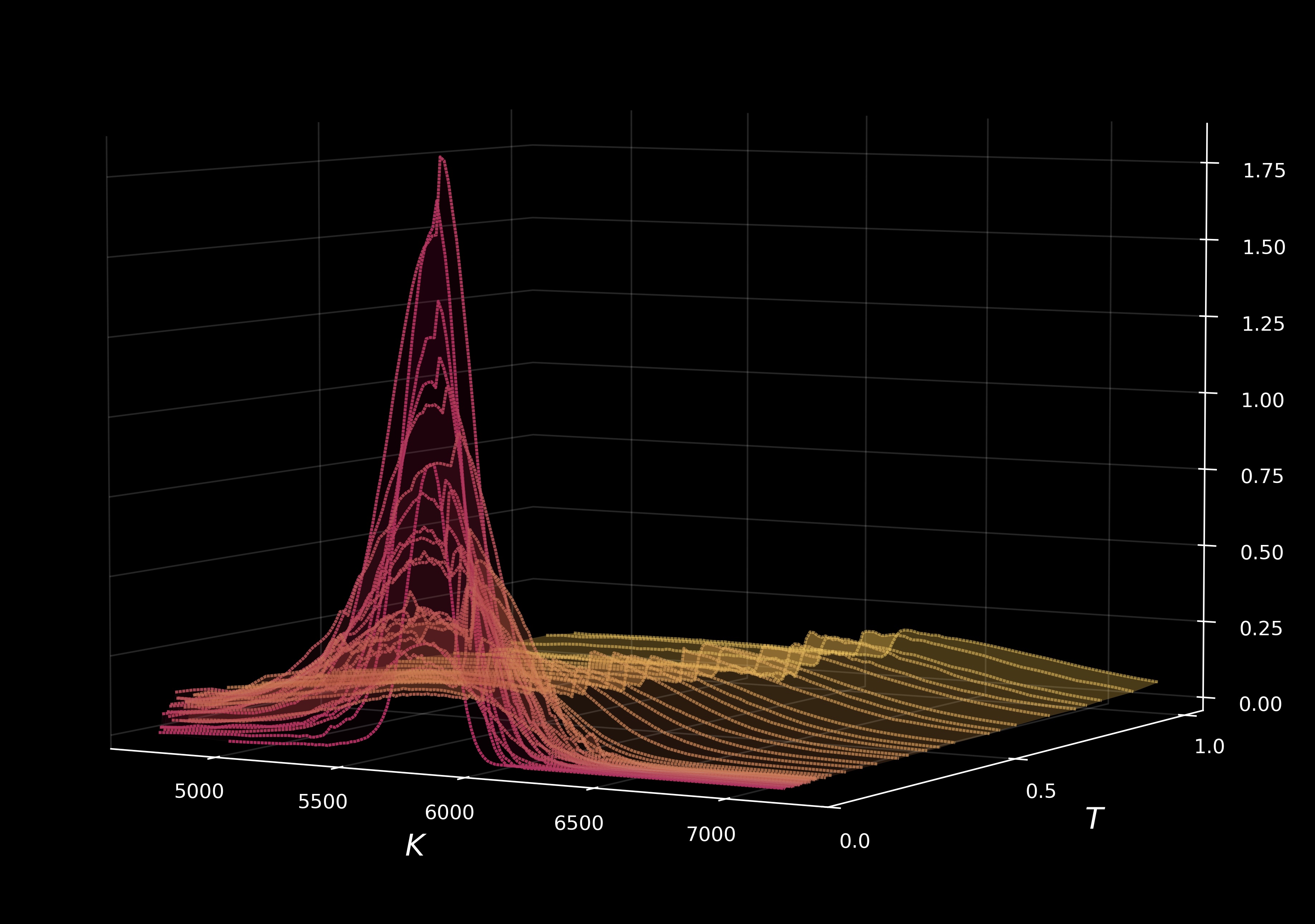}
\end{subfigure}
\begin{subfigure}[b]{0.49\textwidth}
     \centering
\includegraphics[height=2.1in,width=2.75in]{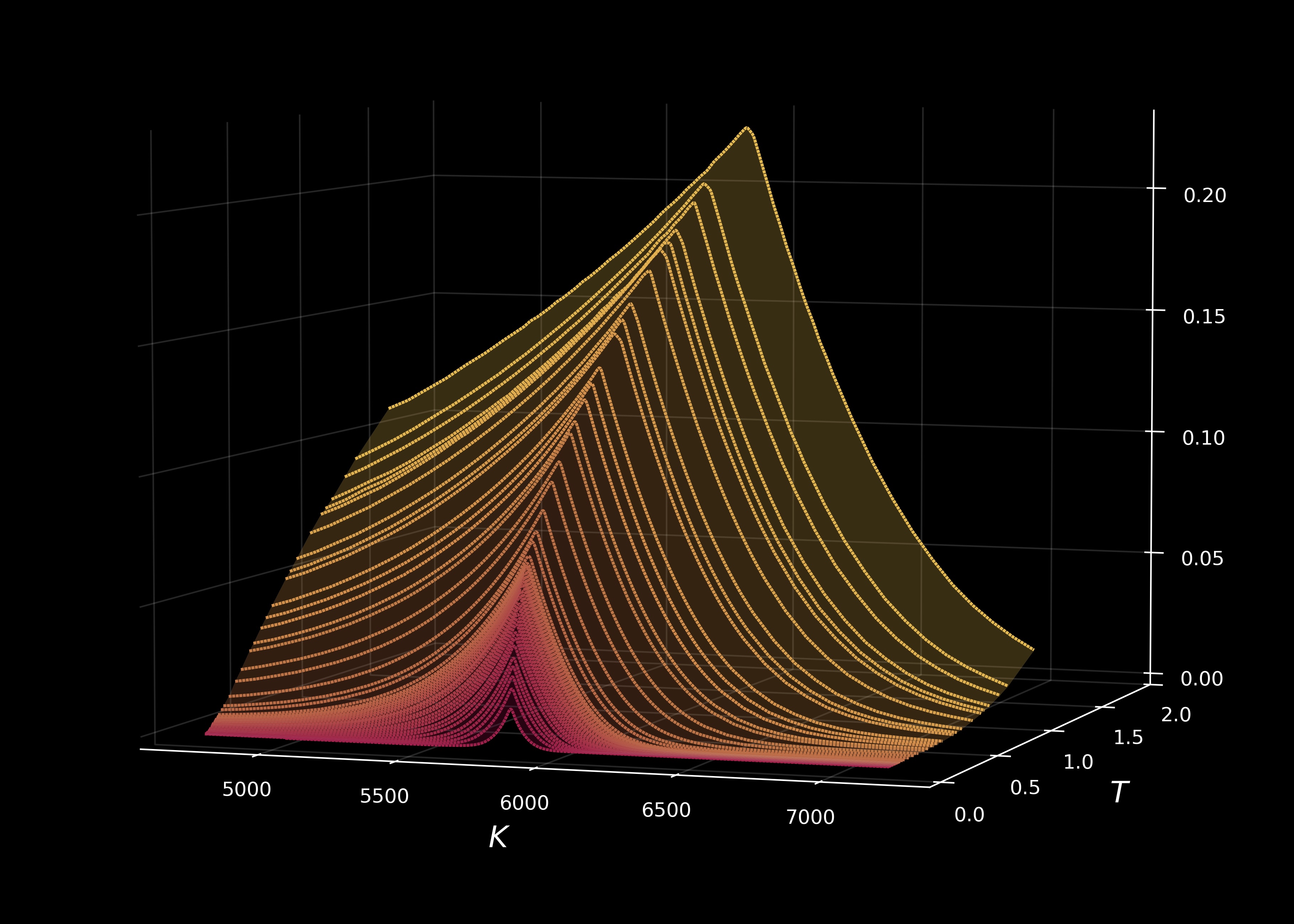}
 \end{subfigure}
\vspace{5mm}
 \label{fig:fwdOccSurface}
 \end{figure}



\paragraph{Dynamics of  Forward Occupation Flows}
For all $A\in \calB(\R)$,  $\mo^T_{\cdot}(A)$ is a martingale  adapted to the filtration of $W$. Hence, it admits the  representation
\begin{equation}\label{eq:martingale}
\mo_t^T(A) = \mo_0^T(A) + \int_0^t \eta^T_u(A)\cdot dW_u, \quad t\le T, 
\end{equation}
for some  process $\eta^T_{\cdot}(A)$ in $\R^d$. 
Moreover, for  disjoint $A,A' \in \calB(\R)$, 
$$\int_0^t \eta^T_u(A\cup A')\cdot dW_u = \mo_t^T(A \cup A') 
= \mo_t^T(A) + \mo_t^T(A') = \int_0^t (\eta^T_u(A) + \eta^T_u(A'))\cdot dW_u .$$
Hence $\eta^T_t(A\cup A') = \eta^T_t(A) + \eta^T_t(A')$ for all $ t\le T$, by   uniqueness of the martingale representation. Similarly, one can show that $\eta^T_t$ is $\sigma-$additive, hence a signed random measure. We can then write the dynamics of the instantaneous forward occupation flow as 
\begin{equation}\label{eq:omartingale}
    d\mo_t^T = \eta^T_t \cdot dW_t = \sum_{i=1}^d\eta^T_{t,i} \ dW_t^i,
\end{equation}
where $\eta^T_t \in \calM^d$ is a vector-valued measure.   
Similar to  \eqref{eq:fwdVarDynamics}, suppose  that $\eta_t^T = \eta(t,T,\mo_t^{\tau})$ 
for some map $\eta:\R_{+}^2\times \calM \to \calM^d$ 
and time $\tau = \tau(t,T)$. When $t \ge T$, w set $\eta_t^T = 0$ as $\mo_t^T$ is constant after $T$ (equal to $\delta_{X_T}V_T$).  
Moreover,  the forward occupation measure $\calO_t^T $ has $t-$dynamics, 
$$d\calO_t^T = d\left(\int_0^T \mo_t^{s} ds\right) = d\left(\int_t^T \mo_t^{s} ds\right)  = H_t^T \cdot dW_t,   $$
with volatility $H_t^T = \int_t^T\eta_t^sds = \int_t^T\eta(t,s,\mo_t^{\tau(t,s)})ds \in \calM^d$. Unlike $\calO$, the forward occupation flow $\calO^T$ does not play the role of time since its total mass process, the  variance swap rate $\text{VS}^T$, is 
of infinite variation. Leveraging the tools in \cref{sec:ITOMAIN} leads to the following  chain rule 
\begin{align}\label{eq:fwdChain}
   df(\calO^T_t) &=   (\delta_{\mo}f(\calO_t^T)\cdot H_t^T)\cdot dW_t +  \frac{1}{2} \delta_{\mo\mo}f(\calO_t^T)\cdot (H_t^T)^{\otimes 2}dt, \; f\in \calC^2(\calM), 
\end{align}
where $\phi\cdot H = \int_{\R}\phi dH$ if $\phi = \delta_{\mo}f(\calO_t^T)$, $H \in \calM^d$,  and  when  $\phi = \delta_{\mo\mo}f(\calO_t^T)$, 
$$\phi\cdot H^{\otimes 2} = \int_{\R^2}\phi(x,x')H(dx)\cdot H(dx') = \sum_{i=1}^d \int_{\R^2}\phi(x,x')H_i(dx) H_i(dx'). $$ As expected, a second-order term appears due to the martingality of $\calO^T$, further  confirming that $\calO^T$ differs from time. 

Directions for future research include the development of Feynman-Kac's formula in this context, and  numerical methods for the associated PDEs. The latter could then be used to price variance derivatives contingent on the terminal occupation measure $\calO_T$. While weighted variance swaps are readily available from the prices of vanilla options, the proposed model would permit the valuation of weighted variance derivatives, which do not seem to  be traded \textit{yet}. 
For instance, it would be relevant to  popularize  vanilla options on downside/upside variance,  
$$\varphi(\calO_T) = (\calO_T(A^{\pm}) - K)^+, \quad A^{-} = (-\infty,X_0],  \;  A^{+} = (X_0,\infty), $$ 
or  options on \textit{weighted VIX indexes},  
$$\text{VIX}_{T,w}^2 = \frac{1}{\Delta}w\cdot (\calO_T^{T+\Delta}-\calO_T) = \E^{\Q}_T\Big[\frac{1}{\Delta}\int_T^{T+\Delta}w(X_s)V_sds\Big], $$  
for some function $w:\R\to \R$ and horizon $\Delta >0$ (typically $30$ days).



\begin{example}\label{ex:BergomiModels} 
Inspired by  \citet{Bergomi}, let  $\tau = T$, $\kappa:\R_+\to \R_+^d$ be a vector-valued kernel,  and choose 
$\eta(t,T,\mo_t^{T}) = \eta_0 \kappa(T-t)\mo_t^T$, that is 
\begin{equation}
    \mo_t^T(A) = \mo_0^T(A) \exp\big\{\eta_0 Z_t - \frac{1}{2}\eta_0^2\E^{\Q}[Z_t^2]\big\},
\end{equation}
with the Gaussian Volterra process $Z_t = \int_0^t\kappa(t-s)\cdot dW_s $ and constant $\eta_0>0$. Note that $\eta(t,T,\mo_t^{T})$ is indeed a measure, possibly vector-valued when $d\ge 2$, 
and  $\mo_t^T(A)$ remains nonnegative. When $A = \R$, we recover the  Bergomi model $d\xi_t^T / \xi_t^T  = \eta_0 \kappa(T-t) \cdot dW_t$.  

\end{example}

\begin{example}\label{ex:AffineFwdModels}  (Affine forward occupation)
  Let  $\tau = t$ and choose $\eta(t,T,\mo_t^{t}) = \eta_0\sqrt{\mo_t^t} \kappa(T-t)$, $\kappa:\R_+\to \R_+^d$, namely
  $\eta_t^T(A) = \mathds{1}_{A}(X_t)\eta_0\sqrt{V_t}\kappa(T-t).$ When $A = \R$, we recover the affine forward model of \citet{GatheralKeller}, 
  $$ d\mo_t^T(\R) = d\xi_t^T = \eta_0\sqrt{V_t}\kappa(T-t)\cdot dW_t.$$ 
  Note that the log price $X$ is still an affine process since it depends on $\mo_t^T$ only through its total mass $\mo_t^T(\R) = \xi_t^T$, and $(X,\xi)$ admits an affine cumulant generating function as shown in  \cite{GatheralKeller}. The volatility of $\calO_t^{T} = \int_0^T \mo_t^sds$ is given by 
  \begin{equation}
      H_t^T = \delta_{X_t}\eta_0\sqrt{V_t}K(t,T),\quad K(t,T) = \int_t^T\kappa(s-t)ds \in \R^d.
  \end{equation}
  Consequently, the chain rule \eqref{eq:fwdChain} simplifies to 
  \begin{align*}
   df(\calO^T_t) =   \partial_{\mo}f(\calO_t^T,X_t) \eta_0 \sqrt{V_t}K(t,T)\cdot dW_t +  \frac{1}{2} \partial_{\mo\mo}f(\calO_t^T,X_t) \eta_0^2 V_t |K(t,T)|^2dt ,
\end{align*}
with the second order occupation derivative $\partial_{\mo\mo}f(\mo,x) = \delta_{\mo\mo}f(\mo)(x,x)$. 
It would be interesting to see if one can describe the characteristic function of the forward occupation times $\calO_t^T(A)$, e.g. using the  forest expansion in \citet{Friz2023forest}. This would give access to Fourier methods for  variance  derivatives contingent on  $w\cdot \calO_T = \int_0^Tw(X_t)V_tdt$, the floating leg of weighted variance swaps. 
Alternatively, one may adapt  the deep-learning algorithm of \citet{JackCurve} for curve-dependent PDEs to price weighted variance derivatives using finitely many forward occupation  curves $T\to \mo_t^{T}(A)$.

\end{example}

%% file: Appendix/Invariance.tex
\section{Chronology Invariance}\label{sec:occFunctionalEx}

Intuitively, the chronology of a path has no impact on its calendar time occupation measure. 
 Indeed, given $A\in \calB(\R)$ and $T>0$,  then   $\tilde{\calO}_T(A)$,  
 tells us how much time $X$ has spent  in $A$ up to $T$, but not \textit{when}. We can therefore "shuffle" pieces of the path  while leaving the occupation measure unchanged.  
This invariance  
carries over to occupation functionals $ f(\tilde{\calO}_T)$.  
We here provide some insights by  adopting the pathwise setting from \cref{sec:FITO}. 
Let  $\Omega_T = \{\omega: [0,T] \to \R, \, \text{càdlàg}\}$ as in \cref{sec:FITO}, and  $\text{Sym}(A)$ denote the symmetric group of some arbitrary set $A$. Consider the  partition $\{t_n = \frac{nT}{N}  :  n=0,\ldots,N\}$  of $[0,T]$ and the   subgroup $\frakS^N_T$ of $(\text{Sym}([0,T]),\circ)$  containing the time permutations
\begin{equation}\label{eq:timePerm}
    \psi(s) =  t_{\sigma(n)-\xi_n^+} \ + \ \xi_n (s-t_{n-1}) , \quad t \in  I_n := [t_{n-1},t_n) \quad  (I_N = [t_{N-1},T]),  
\end{equation}
where 
   $\xi \in \{\pm 1\}^{N}$ and $\sigma \in \text{Sym}(\{1,\ldots,N\})$. 
The transformations consists of two steps: (i) permute the intervals $I_n$ according to  $\sigma$  (ii) use   $\xi_n$ to determine if the restriction $\omega|_{I_{\sigma(n)}}$ is run forward ($\xi_n=1$) or backward $(\xi_n=-1)$.  See  \cref{fig:permutation}.  Next, consider the action  $ (\psi,\omega) \  \mapsto \ \omega \circ   \psi \in \Omega_T$ and set also  $\frakS_T := \bigcup_{N\ge 1} \frakS^N_T$. 

\begin{figure}[H]
\centering
\caption{ Original path $\omega$ (left), time permutation $\psi$ defined in \eqref{eq:timePerm} with $\xi = (-1,1,1,-1)$ and $\sigma((1,2,3,4)) = (3,2,4,1)$ (middle) and  transformed path $\omega \circ  \psi$ (right). }
     \centering
    \includegraphics[height=1.9in,width=5.7in]{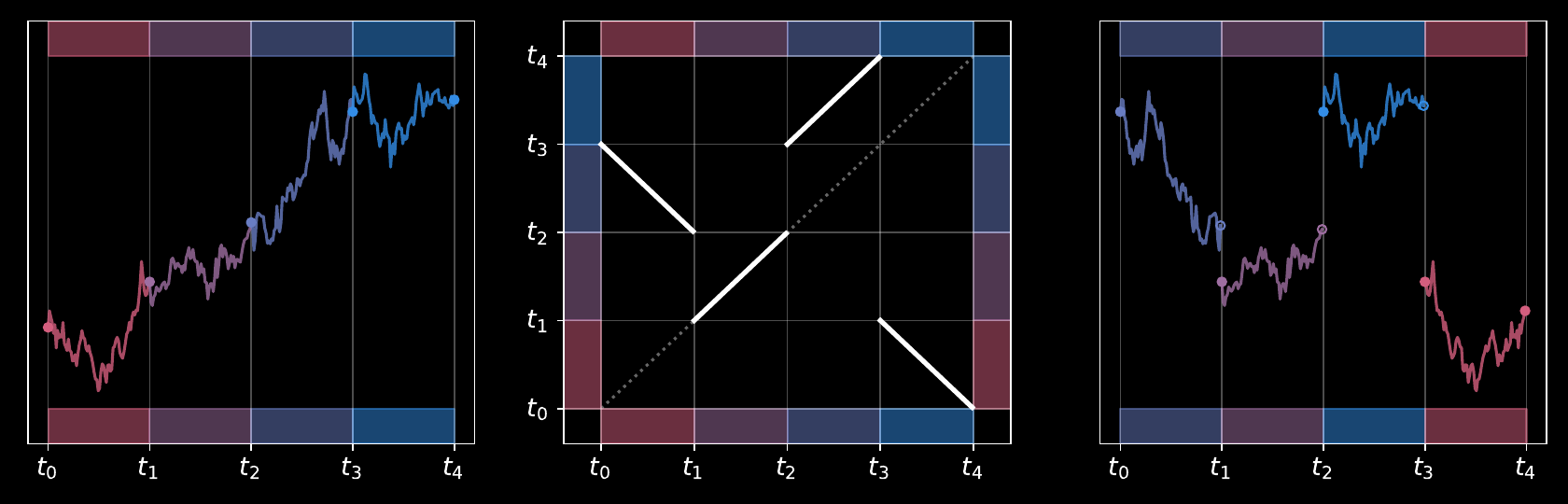}
 \label{fig:permutation}
 \end{figure}

\begin{definition}
    A path functional $\textnormal{f}:\Omega_T \to \R$ is called \textit{chronology-invariant} if $\textnormal{f}(\omega\circ \psi ) = \textnormal{f}(\omega)$ for all $\omega\in \Omega_T$, $\psi \in \frakS_T$.  
\end{definition}    
The next Theorem finally connects chronology invariance functionals with occupation functionals. 

\begin{theorem}
    \label{thm:chronology}  
The following are equivalent:
 
 \begin{itemize} 
     \item[\textnormal{I.}]  $\textnormal{f}:\Omega_T \to \R$ is chronology-invariant.
     \item[\textnormal{II.}]There exists   $f:\calM \to \R$ such that $\textnormal{f}(\omega) = f(\tilde{\calO}_T(\omega)).$
 \end{itemize}
\end{theorem}

In finance, the above equivalence  comes as  a useful tool to check if an exotic payoff is function of the calendar time occupation measure; see  \cref{sec:occPayoffs}.

\begin{proof} 
Clearly,  II. implies I. as the occupation measure is chronology invariant. For completeness,  fix $\omega \in \Omega_T$,  $\psi \in \frakS_T^N$ with permutation $\sigma$ and sign vector $\xi\in \{\pm 1\}^{N}$. First, we may assume that $\xi_n=1$ for all $n$ as clearly $\tilde{\calO}_T(\overleftarrow{\omega})=\tilde{\calO}_T(\omega)$, where $\overleftarrow{\omega}_t:= \omega_{T-t}$. From the additivity of $\tilde{\calO}_T$, we conclude that 
\begin{align*}
    \tilde{\calO}_T(\omega \circ \psi) = \sum_{n=1}^N  \int_{I_n} \delta_{\omega_{\psi(s)}} ds 
             = \sum_{n=1}^N  \int_{I_{\sigma(n)}} \! \delta_{\omega_s} ds  
             = \tilde{\calO}_T(\omega).
\end{align*} 
To prove I. $ \Longrightarrow $ II. ,  introduce the equivalence relation 
$$\omega \sim \omega' \ \Longleftrightarrow \  \tilde{\calO}_T(\omega) = \tilde{\calO}_T(\omega').$$ 
We then construct a  representative for the equivalence classes in $\Omega_T /\! \sim$ by characterizing the pre-image of $\tilde{\calO}_T:\Omega_T \to \calM$. 
Given $\mo = \tilde{\calO}_T(\omega)\in  \calM$ for some $\omega\in \Omega_T$, let  $\omega^{\mo}\in \Omega_T$ defined as 
\begin{equation}
  \omega_t^{\mo} = \sup\{x\in \R \ : \ \mo((-\infty,x]) \le t\}, \quad t \le T. 
\end{equation}
Noting  $ \omega^{\mo}_t\le x \ \Longleftrightarrow \ \mo((-\infty,x]) \ge t$ by construction, it is then easily seen that $\tilde{\calO}_T(\omega^{\mo}) = \mo$. 
We can therefore define the representation map 
$\calR:\calM \to \Omega_T/ \! \sim$, $\calR(\mo) = \omega^{\mo}$ and 
projection $\pi: \Omega_T \to  \Omega_T /\!  \sim$ as $\pi = \calR \circ \ \tilde{\calO}_T$. In short, 
$$ \pi: \Omega_T \ni \omega \ \overset{\tilde{\calO}_T}{\longrightarrow} \   \mo \ \overset{\calR}{\longrightarrow} \  \omega^{\mo} \in \Omega_T / \! \sim. $$

Next,  fix a chronology-invariant functional $\textnormal{f}:\Omega_T \to \R$ and $\omega\in \Omega_T$. 
If $\omega$ is constant in each subinterval of the partition $\{\frac{nt}{N} : n\le N\}$ for some $N\in \N$, then $\mo = \tilde{\calO}_T(\omega)$ is a discrete measure so the representative $\omega^{\mo}$ is piecewise constant as well. Thus, there exists a sorting permutation $\psi \in \frakS_T^N$  such that  $\omega \circ \psi$ is non-decreasing, and as $\tilde{\calO}_T(\omega\circ \psi) = \tilde{\calO}_T(\omega) = \mo$, we gather that $\pi(\omega) = \omega^{\mo} = \omega \circ \psi$. 
For general càdlàg $\omega$, the same reasoning can be applied by first taking a piecewise constant approximation of $\omega$ as seen in \cref{sec:occFunctionalEx}, followed by a passage to the limit. 
To conclude, set $f:= \textnormal{f}\circ \calR:\calM \to \R$, $\omega\in \Omega$, and $\psi\in \frakS_T$ such that $\pi(\omega) =  \omega \circ \psi$. Hence,   
$$f(\tilde{\calO}_T(\omega))  = (\textnormal{f}\circ \calR \circ \  \tilde{\calO}_T)(\omega) = (\textnormal{f} \circ\pi)(\omega) =  \textnormal{f}(\omega \circ \psi) = \textnormal{f}(\omega). $$
\end{proof}

%% file: Appendix/Proofs.tex
\section{Proofs}\label{app:Proofs}

\subsection{\cref{thm:ItoOX}}\label{app:ItoOX}

\begin{proof} 
  Let  $\Pi^N = \{s=t_0 < \cdots < t_N = t\}$ be a sequence of partitions with mesh size   $\Vert \Pi^N\Vert = \max_{n}(t_{n+1} - t_n)$ decreasing to zero as $N\to \infty$.  Writing  $Y_t = f(\calO_{t},X_{t})$, then, 
   \begin{align}
       Y_t - Y_s &= \sum_{t_n\in\Pi^N} (Y_{t_{n+1}} - Y_{t_{n}}) \quad \qquad (t_{N+1} := t)\nonumber \\
        &= \sum_{t_n\in\Pi^N} (Y_{t_{n+1}} - f(\calO_{t_{n}},X_{t_{n+1}}))\label{eq:do}\\
        &+ \sum_{t_n\in\Pi^N} (f(\calO_{t_{n}},X_{t_{n+1}}) - Y_{t_{n}}).\label{eq:dx}
   \end{align}
    From the classical Itô's formula, 
  \eqref{eq:dx}  converges to 
   $$\int_s^t \partial_x f(\calO_u,X_u) dX_u + \int_s^t \frac{1}{2}\partial_{xx} f(\calO_u,X_u) d\langle X \rangle_u, $$ 
   since the occupation measure each  summand of \eqref{eq:dx} is "frozen" at time $t_n$. 
   We now compute the limit of \eqref{eq:do} as $N\to \infty$. In light of the definition of $\delta_\mo$ (\cref{def:linearderivative}), we set $\Delta_n = \calO_{t_n,t_{n+1}} = \calO_{t_{n+1}} - \calO_{t_{n}}$ for simplicity (note that $\Delta_n$ depends on $N$) 
   and rewrite  the summands in \eqref{eq:do} as 
   \begin{align*}
       Y_{t_{n+1}} - f(\calO_{t_{n}},X_{t_{n+1}}) = \int_0^1  \delta_\mo f(\calO_{t_{n}}+ \eta \Delta_n,X_{t_{n+1}}) \cdot \Delta_n d\eta  
      =  \delta_\mo f(\calO_{t_{n}}+ \eta_n \Delta_n,X_{t_{n+1}}) \cdot \Delta_n,
   \end{align*}
  for some $\eta_n \in [0,1]$ by the mean value theorem. Next,  the occupation time formula \eqref{eq:OTF} yields 
    \begin{align*}
      \sum_{t_n\in\Pi^N} ( Y_{t_{n+1}} - f(\calO_{t_{n}},X_{t_{n+1}})) &=   \sum_{t_n\in\Pi^N} 
      \delta_\mo f(\calO_{t_{n}}+ \eta_n \Delta_n,X_{t_{n+1}}) \cdot \Delta_n \\
      &= \sum_{t_n\in\Pi^N} 
      \int_{t_n}^{t_{n+1}} 
    \delta_\mo f(\calO_{t_{n}}+ \eta_n \Delta_n,X_{t_{n+1}})(X_u) d\langle X \rangle_u \\
       &= \int_s^t Z^N_u d\langle X \rangle_u, 
   \end{align*}
   with the step process
   $$    
     Z^N_u =   \sum_{t_n\in\Pi^N} \delta_\mo f(\calO_{t_{n}}+ \eta_n  \Delta_n,X_{t_{n+1}})(X_u) \ \mathds{1}_{[t_n,t_n+1)}(u),  \qquad [t_N,t_{N+1}) = \{t\}.$$
Recalling that  $\Vert \Pi^N\Vert = \max_{n}(t_{n+1} - t_n)$, then for every $n$ and $u\in [t_n,t_{n+1})$, 
$$\frac{1}{2}\frakm_1(\calO_{u},\calO_{t_{n}}+ \eta_n  \Delta_n)\le \frac{\lVert \calO_{t_n,u}\Vert_1  + \lVert \Delta_n\Vert_1}{2} \le \lVert \Delta_n\Vert_1 = \langle X \rangle_{t_{n+1}} - \langle X \rangle_{t_{n}}\le \varrho(\Vert \Pi^N\Vert),$$
where $\varrho = \varrho(\omega)$ is the modulus of continuity of $\langle X \rangle$ on the compact $[s,t]$. Similarly,   
$|X_{u} - X_{t_{n+1}}| \le \rho(\Vert \Pi^N\Vert)$ for all $u\in [t_n,t_{n+1})$,  where $\rho = \rho(\omega)$ is the modulus of continuity of $X$ on  $[s,t]$. 
Since  $(o,x,y) \mapsto \delta_\mo f(o,x)(y)$ is jointly continuous and $\lim_{N\to \infty}\Vert \Pi^N\Vert = 0$,
$$Z^N_u \to Z_u^{\infty} = \delta_\mo  f(\calO_{u},X_{u})(X_u) = \partial_o  f(\calO_{u},X_{u}), \;  \text{uniformly  on  $ [s,t]$},\; \Q-\text{a.s.}$$
Hence,   
$\lim_{N\to \infty} \int_s^t Z^N_u d\langle X \rangle_u = \int_s^t \partial_o  f(\calO_{u},X_{u}) d\langle X \rangle_u, \  \Q-$a.s. which completes the proof.

\end{proof}

\subsection{\cref{prop:Malliavin}}
\label{app:Malliavin}

\begin{proof}
    First, the additivity property of $\calO$  gives 
    $$\calO_T(\omega + h\mathds{1}_{[t,T]}) = \calO_t(\omega) + \int_t^T \delta_{\omega_s+h}ds = \calO_t(\omega) + \calO_{t,T}^h(\omega), \quad \forall A \in \calB(\R),$$
    where $\calO^h_{t,T}(\omega)(A) = \calO_{t,T}(\omega)(A-h)$. 
    That is,  parallel shifts of $\omega$ in $[t,T]$ translate into   translations of the increment occupation measure $\calO_{t,T}$. 
   Writing $\Delta_{t,T}^{h} = \calO_{t,T}^h - \calO_{t,T}$, then from the definition of $\delta_{\mo}$ and $\textnormal{f}$, we have 
    \begin{align*}
      \textnormal{f}(\omega + h\mathds{1}_{[t,T]}) - \textnormal{f}(\omega) 
        &= \int_0^1  (\delta_{\mo} f(\calO_T + \eta \Delta_{T}^{t,h}) \cdot  \Delta_{t,T}^{h})(\omega)\  d\eta = (\delta_\mo f(\calO_T + \eta^h \Delta_{T}^{t,h}) \cdot  \Delta_{t,T}^{h})(\omega), 
    \end{align*}
    for some $\eta^h\in [0,1]$. 
Formally,   note that 
$\lim\limits_{h \to 0} \frac{\Delta_{t,T}^{h}(A)}{h} = \lim\limits_{h \to 0}\int_A \frac{L^{x-h}_{t,T}-L^{x}_{t,T}}{h}dx = - \int_A \partial_xL_{t,T}^x dx $ (omitting the dependence on $\omega$), 
hence 
\begin{equation*}
    D_{t}f(\calO_T) = - \int_{\R}\delta_\mo f(\calO_T)(x) \partial_xL_{t,T}^x dx = \int_{\R}\partial_{xo} f(\calO_T,x) L_{t,T}^x dx =
    \int_t^T \partial_{xo}f(\calO_T,X_s)ds,
\end{equation*}
using integration by parts and the occupation time formula. 
Although local time of Brownian motion is $\Q-$a.s. nowhere differentiable (e.g. as a consequence of the first Ray-Knight theorem \cite[Chapter XI]{RevuzYor}), 
this heuristics   turns out to be correct. Indeed, 
    for all $\phi \in \calC^1(\R)$, 
     we have using the occupation time formula that 
     $$\frac{1}{h}  (\phi \cdot  \Delta_{t,T}^{h})=  \int_t^T \frac{\phi(X_s+h) - \phi(X_s)}{h} ds \longrightarrow \int_t^T \partial_x \phi(X_s) ds.$$ 
Noting that $ \delta_\mo f(\calO_T + \eta^h \Delta_{t,T}^{h})   =  \delta_\mo  f(\calO_T) +  \varepsilon_h$, $\lim_{h\to 0} \varepsilon_h = 0$, we can apply the above limit to $\phi  = \delta_\mo f(\calO_T)\in \calC^1(\R)$ and obtain \eqref{eq:malliavin}. 
 \end{proof}

    \subsection{\cref{thm:FK}}\label{app:FKBM}
 \begin{proof}
The first assertion follows from   the fact that $v(\calO,X)$ is an $(\F,\Q)-$martingale, as a consequence of \eqref{eq:heatPDE} and Itô's formula (\cref{thm:ItoOX})  applied to $v\in \calC^{1,2}(\calD)$. 
Hence $v(\mo,x) = \E_{\mo,x}^{\Q}[v(\calO_T,X_T)]$, which is \eqref{eq:valueFctal} in light of the terminal condition \eqref{eq:heatTerminal}. 

For the second part, it is enough to  verify that $v \in \calC^{2}_2(\calD)$ as the uniqueness then follows from the previous assertion. For all $\mo,\mo'\in \calM$, $\Delta =\mo'-\mo$, and condition \eqref{eq:linDer}, observe that  
\begin{align*}
    v(\mo',x)-v(\mo,x) &= \E^{\Q}[\varphi(\mo'+\calO_{T-t}^x,x + X_{T-t}) - \varphi(\mo+\calO_{T-t}^x,x + X_{T-t})]\\
   &=  \E^{\Q}\left[ \int_0^1 \delta_\mo\varphi(\mo+\calO_{T-t}^x + \eta \Delta,x + X_{T-t}) \cdot \Delta \ d\eta  \right]\\
        &=  \int_0^1 w(\mo+ \eta \Delta,x,\cdot)  \cdot \Delta \ d\eta, \\[1em]
        w(\mo,x,y) &= \E^{\Q}[\delta_\mo\varphi(\mo+\calO_{T-t}^x ,x + X_{T-t})(y)].
\end{align*}
Clearly, the joint continuity of $(\mo,x,y) \mapsto \delta_\mo\varphi(\mo,x)(y)$ carries over to $w$. Moreover, defining $C := \sup_{(\mo,x)\in \calD, y\in \R}\frac{|\delta_\mo\varphi(\mo ,x)(y)|}{1+|y|^2} $ which is finite as  $ \varphi \in \calC_2^{2}(\calD)$, we have  from Jensen's inequality  that 
$|w(\mo,x,y)| \le C(1+|y|^2)$, uniformly in $(\mo,x)$. 
Hence $\delta_\mo v = w$ in light of \cref{def:linearderivative}, which implies that  $v \in \calC^{1,0}(\calD)$.  Next, it is easily seen that  
$$\delta_{\mu \mu}v(\mo,x)(y,y') = \delta_\mo w(\mo ,x,y)(y') = \E^{\Q}[\delta_{\mu\mu}\varphi(\mo+\calO_{T-t}^x ,x + X_{T-t})(y,y')], $$ 
which is, similar to  $\varphi$,  continuous in $(\mo,x)$,    twice differentiable with respect to $(y,y')$ and satisfies the quadratic growth condition II. in \cref{def:linearderivative} (with respect to $y'$). 

Moving to the space derivatives, observe that  for every $f\in \calC^1(\calM)$ such that $y\mapsto \delta_\mo f(\mo)(y) \in \calC^1(\R)$, we have 
$\frac{d}{dh} f(\mo+\calO_{s}^{x+h}) = \partial_y \delta_\mo f(\mo+\calO_{s}^{x})\cdot \calO_{s}^{x}$ 
using similar arguments as in the proof of \cref{prop:Malliavin};  see \cref{app:Malliavin}. Thus, the chain rule and bounded convergence theorem yields 
\begin{align*}
     \partial_x v(\mo,x) &= \underbrace{\E^{\Q}[\partial_x\varphi(\mo+\calO_{T-t}^{x},x+X_{T-t})]}_{=: \ w_1^{\varphi}(\mo,x)} + \underbrace{\E^{\Q}[\partial_y \delta_\mo\varphi(\mo+\calO_{T-t}^{x},x+X_{T-t}) \cdot \calO_{T-t}^{x}]}_{=: \ w_2^{\varphi}(\mo,x)},
\end{align*}
which is jointly continuous.  
Next, it is immediate that $\partial_x w_1^{\varphi} = w_1^{\partial_x\varphi}+ w_2^{\partial_x\varphi} \in \calC(\calD)$. For $w_2^{\varphi}$, we compute 
\begin{align*}
   \partial_x w_2^{\varphi}(\mo,x) &=
     \E^{\Q}[(\partial_{x}+ \partial_{y})\partial_{y}\delta_\mo\varphi(\mo+\calO_{T-t}^{x},x+X_{T-t})\cdot \calO_{T-t}^{x}],\\[0.5em]
     &+     \E^{\Q}[\partial_{y'y} \delta_{\mu \mu} \varphi (\mo+\calO_{T-t}^{x},x+X_{T-t}) \cdot (\calO_{T-t}^{x} \otimes \calO_{T-t}^{x}) ] 
\end{align*}
where $\phi \cdot (\mo \otimes \mo' ) = \int_{\R^2} \phi(y,y')\mo(dy)\mo'(dy')$, $\phi:\R^2\to \R$,  
and 
using \cite[Lemma 2.2.4.]{CardDelLasLio}  to rewrite $\partial_{y'}\delta_\mo\partial_{y}\delta_\mo$ more compactly as $\partial_{y'y} \delta_{\mu \mu}$. Given our assumptions on $\varphi$, we  conclude that $ \partial_x w_2^{\varphi}\in\calC(\calD)$ and thus   $v\in\calC^{2}_2(\calD)$. 
 \end{proof}

 \subsection{\cref{thm:OSDE}}\label{app:OSDE}
\begin{proof}
     We use a fixed point argument similar to  \cite[Chapter IX]{RevuzYor}. 
Let $\calR$ consist  of all the pairs $(\calO,X)$, where 
$\calO$ (resp. $X$) is a continuous  measure-valued (resp. real-valued) $\F-$adapted process. Note that $\calO$ need not be the occupation flow of $X$. Moreover,  
let  $\vertiii{X}_{t} := \lVert X^*_t\rVert_{L^2(\Q)}$, $X^*_t = \sup_{s\in [0,t]}|X_s|$,   
and define the unit ball  $\calR_1\subset \calR$ as 
\begin{align}
     (\calO,X)\in \calR_1 \ &\Longleftrightarrow \ \lVert(\calO,X)\rVert :=  \vertiii{\varrho(\calO,X)}_{T}\le 1,\label{eq:OSDESup} 
\end{align} 
with  $\varrho(\calO,X) = (\varrho(\calO_t,X_t))_{t\in [0,T]}$. 
Next, set $\calO_0=X_0=0$ in \eqref{eq:OSDEOcc}$-$\eqref{eq:OSDEOcc}  for simplicity  
and 
introduce the map $\Psi = (\Psi_o,\Psi_x):\calR\to \calR$, with  
\begin{align*}
    \Psi_\mo(\calO, X)_t &=   \int_0^t \delta_{X_s}\sigma(\calO_s,X_s)^2ds \in \calM, \\[0.5em]
        \Psi_x(\calO, X)_t &=  \int_0^t b(\calO_s,X_s)ds +  \int_0^t \sigma(\calO_s,X_s)dW_s.
\end{align*}
We show that $\Psi$ admits a unique fixed point  in $\calR_1$ using the contraction mapping theorem. 

\noindent
\textit{Step 1} ($\Psi(\calR_1)\subseteq \calR_1$).
Let $(\calO,X)\in \calR_1$ and  write 
$(\bar{\calO},\bar{X}) = \Psi(\calO,X)$. If 
$f_t = f(\calO_t,X_t)$, $f\in \{b,\sigma\}$, 
we obtain from Cauchy-Schwarz inequality that 
$
    \frac{1}{2}|\bar{X}_t|^2  
    \le t \int_0^t b_s^2ds + M_t^2, 
  $
where  $M_t := \int_0^t \sigma_s dW_s$ is a  true martingale due to the growth condition \eqref{eq:OSDEGrowth} satisfied by $\sigma$.  
Hence    $\vertiii{M}_T\le 2 \lVert M_T\rVert_{L^2(\Q)}$ from  Doob's inequality. Next, Itô isometry  gives 
\begin{align*}
    \frac{1}{2}\vertiii{\bar{X}}_T^2 
    &\le \big  \lVert \sup_{t\in [0,T]} t\int_0^t  b_s^2 ds \big \rVert_{L^2(\Q)}^2 + 4 \lVert M_T\rVert_{L^2(\Q)}^2 \\[0.75em] 
      &\le   \int_0^T \left[T\lVert b_t \rVert_{L^2(\Q)}^2 + 8 \lVert \sigma_t\rVert_{L^2(\Q)}^2\right]dt 
      \\[0.75em] 
     &\le  K^2(T+8)\int_0^T (1+\lVert \varrho(\calO_t,X_t) \rVert_{L^2(\Q)}^2)dt \\[0.75em] 
      &\le    K^2T(T+8)(1 +  \lVert(\calO,X)\rVert^2), 
\end{align*}
using that $\lVert \sigma_t\rVert_{L^2(\Q)}^2 \le K^2 (1+\lVert \varrho(\calO_t,X_t) \rVert_{L^1(\Q)}) \le 2K^2(1+\lVert \varrho(\calO_t,X_t) \rVert_{L^2(\Q)}^2)$, which follows from  \eqref{eq:OSDEGrowth} and the elementary inequality $1+\rho \le 2(1+\rho^2)$. 
Recalling that  $\lVert(\calO,X)\rVert\le 1$, we thus obtain $\vertiii{\bar{X}}_T^2 \le 4K^2T(T+8)$. 
Next, observe that     $\sup_{t\in [0,T]}|\bar{\calO}_t|_{\textnormal{\tiny{BL}}}  = |\bar{\calO}_T|_{\textnormal{\tiny{BL}}} $ since the flow $\bar{\calO}$ is  nonnegative and nondecreasing.   
Then, for a bounded Lipschitz function $\phi$ with $\lVert \phi\rVert_{\textnormal{\tiny{BL}}}\le 1$,  the occupation time formula \eqref{eq:OTF} and \eqref{eq:OSDEGrowth} yield 
   \begin{align*}
       \phi\cdot \bar{\calO}_T 
   &=   \int_0^T \phi(X_t) \sigma_t^2 dt \le K^2\lVert \phi\rVert_{\infty} \int_0^T (1+ \varrho(\calO_t,X_t))dt,\\[1em]
   \Longrightarrow |\! \calO_T\! |_{\textnormal{\tiny{BL}}}  &= \sup\{\phi\cdot \bar{\calO}_T : \lVert \phi \rVert_{\textnormal{\tiny{BL}}} \le 1\}\le  K^2\int_0^T (1+ \varrho(\calO_t,X_t))dt. 
   \end{align*}
Thus $\lVert|\! \calO_T\! |_{\textnormal{\tiny{BL}}}\rVert_{L^1(\Q)}\le 2K^2 T$,  
and
$\Psi(\calR_1)\subseteq \calR_1$ for all $T$ that satisfies  $C_1(T):= 2K^2T(2T+17)\le 1$. 

\noindent
\textit{Step 2} (Contraction property of $\Psi$). Let   $(\calO,X),(\calO',X')\in \calR_1$ and write $(\bar{\calO}, \bar{X}) = \Psi(\calO,X)$, $(\bar{\calO}', \bar{X}') = \Psi(\calO',X')$. Using   condition \eqref{eq:OSDELip}, then 
\begin{align*}
    \frac{1}{2}\vertiii{\bar{X}-\bar{X}'}_T^2 
      &\le   \int_0^T T\lVert b_t' - b_t\rVert_{L^2(\Q)}^2  + 4 \lVert \sigma_t' - \sigma_t\rVert_{L^2(\Q)}^2 dt \\ 
      &\le K^2T(T+4)\lVert(\calO'-\calO,X'-X)\rVert^2.  
\end{align*}
We now provide an estimate of $|\bar{\calO}_t' -\bar{\calO}_t|_{\textnormal{\tiny{BL}}} $. 
For any bounded Lipschitz function $\phi$ such that $\lVert \phi\rVert_{\textnormal{\tiny{BL}}}\le 1$,  the occupation time formula \eqref{eq:OTF} yields  
\begin{align*}\label{eq:phiOSDE}
   \phi\cdot (\bar{\calO}_t'-\bar{\calO}_t) 
   =  
   \int_0^t (\phi(X_s') (\sigma'_s)^2 -   \phi(X_s) \sigma_s^2) ds 
    \le  \int_0^t\left| \phi(X_s') (\sigma'_s)^2 - \phi(X_s) \sigma_s^2 \right|ds.
\end{align*} 
Moreover, 
 the  conditions  \eqref{eq:OSDEGrowth}$-$ \eqref{eq:OSDELip} give for all  $ (\mo,x),(\mo',x') \in \calD$,
\begin{align*}
   \frac{| \phi(x') \sigma(\mo',x')^2 - \phi(x) \sigma(\mo,x)^2 |}{\varrho(\mo'-\mo,x'-x)} &\le  \underbrace{\frac{|\phi(x')-\phi(x)|}{|x'-x|}}_{\le 1} |\sigma(\mo',x')|^2 \\
   &+ \underbrace{| \phi(x)|}_{\le 1} \frac{|\sigma(\mo',x')^2 -  \sigma(\mo,x)^2|}{\varrho(\mo'-\mo,x'-x)} \\[1em] 
   &\le K^2(1 + \varrho(\mo',x'))+ K \big( |\sigma(\mo',x')| + |\sigma(\mo,x)|\big)\\[1em] 
     &\le \underbrace{K^2\left[ 3 +  2 \varrho(\mo',x') +  \varrho(\mo,x)\right]}_{=: \ \Xi(\mo,\mo',x,x')}
\end{align*}
using the  coarse estimate 
$ \sqrt{ 1+\varrho} \le 1+\varrho$
in the last inequality. 
Defining $ \xi_s := \Xi(\calO'_s,\calO_s,X_s',X_s)$, $\varrho_s := \varrho(\calO'_s-\calO_s,X_s'-X_s)$, then 
$$| \bar{\calO}_t'-\bar{\calO}_t |_{\textnormal{\tiny{BL}}} 
   \le   
 \int_0^t\xi_s \varrho_s ds \le T \xi^* \varrho^*, $$ 
 with  $\xi^* = \sup_{t\in [0,T]}\xi_t$ and similarly for $\varrho^*$. 
Since $(\calO,X),(\calO',X')\in \calR_1$, we obtain that $\lVert \xi^*  \rVert_{L^2(\Q)}^2 \le 3K^4 \big[9+4\lVert (\calO',X')\rVert^2 + \lVert (\calO,X)\rVert^2  \big] \le 42K^4$, hence      
\begin{align*}
  \frac{1}{T} \E^{\Q} \big[ \sup_{t\in [0,T]}| \bar{\calO}_t'-\bar{\calO}_t |_{\textnormal{\tiny{BL}}}\big] \le  
  \lVert \xi^*  \rVert_{L^2(\Q)}\lVert  \varrho^* \rVert_{L^2(\Q)} \le \sqrt{42}K^2 \lVert (\calO'-\calO,X'-X)\rVert.
\end{align*}
Altogether, we conclude that 
$\lVert \Psi(\calO',X') -  \Psi(\calO,X)\rVert^2 \le C_2(T)\lVert (\calO',X')-(\calO,X)\rVert$, with $  C_2(T) =  K^2T(T+4+\sqrt{42}) < C_1(T).$ 
When $C_1(T) < 1$, $\Psi$ is thus a contraction from $\calR_1$ into itself and  admits a unique fixed point, i.e.,  a strong solution of \eqref{eq:OSDEOcc}$-$\eqref{eq:OSDE}. When $C_1(T) \ge 1$, we split $[0,T]$ into $N\in \N$ intervals of length $\delta T = \frac{T}{N}$ with  $C_1(\delta T) < 1$ and   successively concatenate  the  strong solutions from   each interval. 
\end{proof}

\subsection{\cref{prop:vEpsEU}}\label{app:vEpsEU}

\begin{proof}
   Let $\overleftarrow{X}_s = X_{T-s}-X_T$ and denote its flow of occupation measures by $\overleftarrow{\calO}$ and local time field by $\overleftarrow{L}$. Then, $L_T^{X_T+z}  = \overleftarrow{L_T^{z}}$ for all $z \in \R$, which implies that 
   $$\varphi^{\varepsilon}(\calO_T,X_T) = \avint_{B_{\varepsilon}}L_T^{X_T+z}dz = \avint_{B_{\varepsilon}} \overleftarrow{L_T} \ d\lambda = \varphi^{\varepsilon}(\overleftarrow{\calO}_T,0).$$
Thus,   
    \begin{align*}
        v^{\varepsilon}(T) &=  \E^{\Q}[\varphi^{\varepsilon}(\overleftarrow{\calO}_T,0)] = \avint_{B_{\varepsilon}} \E^{\Q}[\overleftarrow{L}_T]d\lambda = \avint_{[0,\varepsilon]}\E^{\Q}[L_T^x]dx \  \uparrow \ \E^{\Q}[L_T^0],
    \end{align*}
    where we used that $\text{Law}_{\Q}(\overleftarrow{L}_T^x) = \text{Law}_{\Q}(L_T^x)$  and $\partial_x \E^{\Q}[L_T^x] < 0$ $\forall x> 0$ to prove the monotone convergence. 
    Finally, recall from \cref{prop:EUPrice} that $\E^{\Q}[L_T^0]=v(T)$. For the second assertion, notice from Tanaka's formula that $\E^{\Q}[L_T^x] = 2c(T,x)$, where  $c(T,x) = \E^{\Q}[(X_T-x)^+]$  is the  price of a call option struck at $x$ with maturity $T$ in a Bachelier model with unit volatility. Similar to \cite{BreedenLitzenberger}, 
    we have $\partial_x c(T,x) = \Q(X_T \le x) - 1 = \Phi_T(x)-1 $ and $\partial_{xx} c(T,x)  = \phi_T(x)$, where $\Phi_T,\phi_T$ is the CDF and PDF of $\calN(0,T)$, respectively. 
    Hence, a second order Taylor expansion of $x\mapsto c(T,x)$ at $x=0$ yields
    $$v^{\varepsilon}(T) = \frac{2}{\varepsilon}\int_0^{\varepsilon}c(T,x)dx = v(T) - \frac{\varepsilon}{2} + \frac{\varepsilon^2/3}{\sqrt{2\pi T}} (1 +  r_\varepsilon), \quad \lim_{\varepsilon\downarrow 0} r_\varepsilon = 0. $$
\end{proof}

\subsection{\cref{prop:conv}}\label{app:conv}

\begin{proof} 
    For every stopping time $\tau\in \vartheta(\calT)$, $x\in \R$, we have by Jensen's inequality that 
$$|\varphi^{\varepsilon}(\calO_\tau,x) - L_\tau^x|^2 = \left|\avint_{B_{\varepsilon}}(L_\tau^{x+z} - L_\tau^x) \  dz \right|^2 \le \avint_{B_{\varepsilon}}|L_\tau^{x+z} - L_\tau^x|^2 \  dz. $$
Moreover, there exists $C >0$ (to be determined)  such that 
$$\left\lVert\sup_{(t,x)\in [0,T]\times  \R}|L_t^{x+z} - L_t^{x}|\right\rVert_{L^2(\Q)}^2 \le C |z|, \quad |z|\le 1.$$
Indeed, writing $L^x(Y)$ for the local time of some process $Y$ at $x$ and $\vertiii{Y} :=  \lVert \sup_{t\in [0,T]} Y_t \rVert_{L^2(\Q)}$,
  we can apply Theorem 6.5 in \citet{BarlowYor} with $\gamma = 2$ to the continuous martingales $M = X - z$, $N = X$. This yields
\begin{align*}
    \left\lVert\sup_{(t,x)\in [0,T]\times  \R}|L_t^{x+z} - L_t^{x}|\right\rVert_{L^2(\Q)}^2 &= \left\lVert\sup_{(t,x)\in [0,T]\times  \R}|L_t^{x}(M) - L_t^{x}(N)|\right\rVert_{L^2(\Q)}^2 \\[1em] 
    &\le C_2^2 \ (\vertiii{M} + \vertiii{N}  ) \underbrace{\vertiii{M-N}}_{ = \ |z|} \  \max \left(\log \frac{\vertiii{M} + \vertiii{N} }{\vertiii{M-N}}, 1\right) \\
    &\le 2 C|z|, 
\end{align*}
with  $C = 2 C_2^2 \ \sqrt{T+1} \max(\log  4 \sqrt{T+1} / |z|,1) $\footnote{note that 
$\frac{1}{4}\vertiii{N}^2 \le \frac{1}{4}\vertiii{M}^2 \le \ \rVert X_T-z\lVert_{L^2(\Q)}^2 = T +|z|^2 \le T+1$ using Doob's  inequality and $|z| \le 1$. Hence $\vertiii{M} + \vertiii{N} \le 4\sqrt{T+1}$.}   
and $C_2$ as in \cite[Theorem 6.5]{BarlowYor}. 
Thus,
      \begin{align*}
         \left\lVert\varphi^{\varepsilon}(X_{\tau},\calO_{\tau}) - L_{\tau}^{X_{\tau}} \right\rVert_{L^2(\Q)}^2 &\le \avint_{B_{\varepsilon}} \left\lVert L_{\tau}^{X_{\tau}+z} - L_{\tau}^{X_{\tau}}\right\rVert_{L^2(\Q)}^2  \  dz
          \le 2C \avint_{B_{\varepsilon}}|z|  dz = C \varepsilon, \quad \forall \ \varepsilon \le 1.  
      \end{align*}
      The second assertion  follows from Jensen's inequality. 
\end{proof}

%% file: Appendix/Algorithm.tex
\section{Algorithms}\label{app:Algos}

\subsection{Least Square Monte Carlo for Spot Local Time}\label{app:LSMC}

\cref{alg:LSMC} summarizes the Least Square Monte Carlo approach seen in \cref{sec:American}. Notice the exponential transformation of the discrete Brownian in step IV.2. so that all variables  are nonnegative. We  choose a family of functions on $\R_+$, namely  the Laguerre polynomials $(p_k)$, up to degree $3$, which are orthonormal in the weighted Hilbert space $L^2(\R_+,w d\lambda)$, $w(x) = e^{-x}$. The basis functions $\phi_k$ in \cref{alg:LSMC} then corresponds to  evaluations of the feature variables $\calO_{t_n}^{\bar{m}}, \bar{X}_{t_n}$  through the  polynomials. 
Note that $p_1(\calO_{t_n}^{\bar{m}}) := (p_1(L_{t_n}^{x_{m-\bar{m}}}),\ldots, p_1(L_{t_n}^{x_{m+\bar{m}}}) )$ 
contains the intrinsic value $L_{t_n}^{X_{t_n}}$ itself (up to a constant) which is of course valuable information. 
One can also choose instead, as in \cite{LS}, the better-behaved Laguerre polynomials weighted by $\sqrt{w}$, yielding similar results. 

In both phases, we use antithetic sampling, i.e.,  we simulate $X(\omega_j)$ for $j=1,\ldots,J/2$ (provided that $J$ is even) and set $X(\omega_j) = -X(\omega_{j-J/2})$ for $j=J/2 +1,\ldots,J$,  $J\in \{J_{\text{off}},J_{\text{on}}\}$.  

\begin{algorithm}[H]
\caption{(Least Square Monte Carlo,  Spot Local Time) }\label{alg:LSMC}

\vspace{1mm}
 $T>0$, $N \in \N$,  $\bar{m}\ll N$ (truncation level), $J_{\text{off}}$ (\# simulations, offline phase), 
  $J_{\text{on}}$ (\# simulations,  online phase), $\{\phi_k : k=1,\ldots ,K\}$ (basis functions)
 \\[-0.8em] 
\noindent\rule{\textwidth}{0.5pt} 
\textbf{\underline{OFFLINE}}
\begin{itemize}
 \setlength \itemsep{0.15em}

\item[I.] \textbf{Space-time grid} $\begin{cases}
     t_n \ = n\delta t, \; \  &n=0,\ldots, N, \quad  \delta t = \frac{T}{N}, \\
     x_m =  m \sqrt{3\delta t},  \; \  &m=-N,\ldots,N
\end{cases}$
 
\item[II.] \textbf{Generate} paths $X(\omega_j)$,  $j=1,\ldots,J_{\text{off}}$, where  
\begin{align*}
  X_{t_{n+1}} = X_{t_n} + Z_{n+1} \sqrt{\delta t}, \quad  X_0 = 0, \quad Z_{1},\ldots,Z_{N} \overset{i.i.d.}{\sim} \nu = \frac{1}{6}\left(\delta_{-\sqrt{3}} + 4\delta_0 + \delta_{\sqrt{3}}\right) 
\end{align*}

\item[III.] \textbf{Terminal value} $Y_{T}= L_T^{X_T} = \frac{1}{2\varepsilon}\sum_{n=1}^N  \mathds{1}_{\{X_{t_n} = X_T\}} \delta t$, $\ \varepsilon = \frac{x_{m+1}-x_m}{2} = \frac{\sqrt{3\delta t }}{2}$.

\item[IV.] \textbf{For} $n=N-1,\ldots,\bar{m}$: 
\begin{enumerate}


     \item \textbf{Let} $\calJ = \{j =1,\ldots,J_{\text{off}} \ : \ X_{t_n}(\omega_j) = x_m, \ |m| \le n- \bar{m}\}$

    \item \textbf{Feature variables}:  $\bar{X}_{t_n} = e^{X_{t_n}}$,  $\calO_{t_n}^{\bar{m}} =(L_{t_n}^{x_{m-\bar{m}}},\ldots, L_{t_n}^{x_{m+\bar{m}}} )$, 
    $x_m = X_{t_n}$.
    
    \item \textbf{Regression: } 
    $\hat{\beta}^n = \text{argmin}_{\beta} \sum_{j\in \calJ} \left | Y_{t_{n+1}}- \sum_{k=1}^K \beta_k \ \phi_k(\calO_{t_n}^{\bar{m}}, \bar{X}_{t_n})\right|^2(\omega_j)$.

     \item \textbf{Continuation value:} $C_{t_n} = \sum_{k=1}^K \hat{\beta}_k^n \ \phi_k(\calO_{t_n}^{\bar{m}}, \bar{X}_{t_n})$.
     
     \item \textbf{Update:} $Y_{t_n}(\omega_j) = \begin{cases}
         
         Y_{t_{n+1}}(\omega_j) , & C_{t_n}(\omega_j) > L_{t_n}^{X_{t_n}}(\omega_j) \text{ or } j\notin \calJ, \\
         L_{t_n}^{X_{t_n}}(\omega_j), & \text{otherwise}. 
     \end{cases}  $

\end{enumerate}
   \end{itemize}

\textbf{\underline{ONLINE}}   
\begin{itemize}
    \item[I'.] \textbf{Generate} paths $X(\omega_j)$,  $j=1,\ldots,J_{\text{on}}$ as in II.  

\item[II'.] \textbf{Repeat} steps III., IV.1, 2, 4, 5 using $\hat{\beta}^{N-1},\ldots, \hat{\beta}^{\bar{m}}$ from IV.3.

\item[III'.] \textbf{Return} $Y_0 =\frac{1}{J_{\text{on}}}\sum_{j=1}^{J_{\text{on}}} Y_{t_1}(\omega_j) $.
\end{itemize}

\end{algorithm}